\documentclass[final,3p,times,11pt]{elsarticle}

\usepackage{amssymb}
\usepackage{amsmath,amssymb,latexsym,
natbib}

\newcommand\bcmdtab{\noindent\bgroup\tabcolsep=0pt%
  \begin{tabular}{@{}p{10pc}@{}p{20pc}@{}}}
\newcommand\ecmdtab{\end{tabular}\egroup}

\usepackage[hidelinks]{hyperref}
\usepackage{mathtools}
\usepackage{amstext,mathabx,amsthm}
\usepackage{caption,color}
\newcommand*{\normally}{\mathrel{\ooalign{$|$\hfil\cr\kern+1pt$\thicksim$}}} 
\newcommand*{\nnormally}{\mathrel{\ooalign{$|$\hfil\cr\kern+1pt$\thicksim$}\negthickspace \negthickspace /} } 

\def\I{\mathcal{I}}
\def\K{\mathcal{K}}
\def\F{\mathcal{F}}
\def\P{\mathcal{P}}

\def\Systemsigma{\mathfrak{S}}
\newtheorem{theorem}{Theorem}
\newtheorem{proposition}{Proposition}
\newtheorem{algorithm}{Algorithm}
\newtheorem{remark}{Remark}
\newtheorem{definition}{Definition}
\newtheorem{example}{Example}
\newcommand{\no}[1]{\widebar{#1}}

\journal{}

\begin{document}

\begin{frontmatter}



\title{Probability propagation rules  for  Aristotelian  syllogisms}


\author[inst1]{Niki Pfeifer\corref{cor1}}

\affiliation[inst1]{organization={Department of Philosophy, University of Regensburg},
            city={Regensburg},
            country={Germany}}
\ead{niki.pfeifer@ur.de}
\author[inst2]{Giuseppe Sanfilippo\corref{cor1}}
\affiliation[inst2]{organization={Department of Mathematics and Computer Science, University of Palermo},
            city={Palermo},
            country={Italy}}

\ead{giuseppe.sanfilippo@unipa.it}

\cortext[cor1]{Shared first authorship (both authors contributed equally to this work).}

\begin{abstract}
We present a coherence-based probability semantics
and probability propagation rules 
for (categorical) Aristotelian syllogisms.
For framing the Aristotelian syllogisms as  probabilistic 
inferences,  we interpret basic syllogistic sentence types A, E, I, O by suitable  precise and imprecise conditional probability assessments. 
 Then, we define validity of  probabilistic inferences and  probabilistic notions of the existential import which is required, for the validity of the syllogisms. 
 Based on a generalization of de Finetti's fundamental theorem  to conditional probability, 
we investigate the coherent probability propagation rules of  argument forms of the syllogistic Figures I, II, and III, respectively. 
These results allow to show, 
for all three figures,  that each traditionally valid syllogism is also valid in our coherence-based probability semantics. 
Moreover, 
we interpret the basic syllogistic sentence types by suitable defaults and negated defaults. Thereby,
we build a  bridge from our probability semantics of Aristotelian syllogisms to nonmonotonic reasoning. Then we show that \emph{reductio} by conversion does not work while  \emph{reductio ad impossibile} can be applied in our approach.  
Finally, we  show how the proposed probability propagation rules can be used to analyze syllogisms involving generalized quantifiers (like \emph{Most}).
\end{abstract}

\begin{highlights}
\item Combining logic and coherence-based probability theory
\item Coherence-based probability semantics for Aristotelian syllogisms
\item New notions of validity of probabilistic syllogisms and probabilistic existential import assumptions
\item Probability propagation rules for categorical syllogisms of Figure I, II, and III
\item Applications and bridges  to nonmonotonic reasoning
\item Analysis of the Aristotelian methods of proof by conversion and \emph{reductio ad impossibile}
\item Generalized syllogisms involving intermediate quantifiers by suitable instantiations of probability propagation rules

\end{highlights}


\end{frontmatter}


\section{Motivation and outline.}
\label{SECT:Motivation}
There is a long tradition in logic to investigate categorical syllogisms that goes back to Aristotle's  \emph{Analytica Priora}.  
However, not many  authors proposed \emph{probability}  semantics for categorical syllogisms  (see, e.g., 
\cite{
	AmDP91,Amarger91,boole54b,chater99,cohen99,demorgan47,dubois93,hailperin96,gilio16,lambert1764,thierry11}) to overcome formal restrictions imposed by logic, like its \emph{monotonicity} (i.e., the inability to retract conclusions in the light of new evidence) or its qualitative nature (i.e., the inability to express \emph{degrees of belief}).  
In particular, \emph{universally} and \emph{existentially quantified} statements are hardly ever used in commonsense contexts: even if people mention words like ``all'' or ``every'',  they usually don't mean \emph{all} in the modern sense of the universal quantifier $\forall$. Indeed, universal quantified statements are usually not falsified  by one exception in everyday life. Likewise, people mostly don't mean by ``some'' \emph{at least one} in the sense of the existential quantifier $\exists$. Our aim is  to provide a richer and  more flexible framework for managing quantified statements in common sense reasoning. Specifically,  our probabilistic approach is scalable in the sense that the proposed semantics allows for managing not only traditional logical quantifiers but also the much bigger superset of \emph{generalized} or \emph{intermediate quantifiers}
(see, e.g., \cite{barwise81,peters06,peterson00,
westerstahl89}).   Such a framework will also be useful as a rationality framework for the psychology of reasoning, which has a long tradition in investigating syllogisms (see, e.g., \cite{storring08,khemlani12,pfeifer05a}). 
Finally, a further aim within our probabilistic approach is to build a  bridge from ancient syllogisms to relatively recent approaches in  nonmonotonic reasoning. 

Among various approaches to  probability,  we use the  subjective  interpretation.  
Specifically, we use the theory of subjective probability   based on the  coherence principle  of Bruno de Finetti (see, e.g., \cite{definetti31,definetti74}).
This coherence principle has been investigated  by many authors  and it has been generalized  to the conditional probability and to imprecise probability
(see, e.g., \cite{BeRR98,biazzo02,biazzo05,CaGV03,CaLS07,coletti02,Gili90,Gili96,gilio13ins,gilio16,Holz85,lad96,PeVa17,Rega85,WaPV04}).
The  coherence principle  plays a key role in probabilistic reasoning.  Coherence is a flexible approach as
it allows  to 
assign conditional probability \emph{directly} on an arbitrary family of conditional events---without requiring algebraic structures---and to propagate  coherent probabilities to   further conditional events. 
Moreover, 
coherence  is  more general than approaches which require positive probability for the conditioning events.  In such   approaches the  conditional probability $p(E|H)$ is defined by the ratio $\frac{p(E\wedge H)}{p(H)}$, which requires  positive probability of the conditioning event,  $p(H)>0$ (or by making \emph{ad hoc} assumptions, like setting $P(E|H)=1$, when $P(H)=0$). 
 However,   in the  coherence-based approach, conditional probability $p(E|H)$ is a \emph{primitive} notion and it is properly defined and managed even if the  conditioning event  has  probability zero, i.e., ${p(H)=0}$. If $E$ and $H$ are logically independent, then, by coherence,  $p(E|H)$ can take any value in the unit interval $[0,1]$. 
 Moreover,  for any coherent $p$, the equation  $p(E|H)p(H)=p(E\wedge H)$   follows as a \emph{theorem} (\emph{compound probability theorem}). In particular, when $p(H)>0$, of course coherence requires that  $p(E|H)=\frac{p(E\wedge H)}{p(H)}$; however, when $p(H)=0$, coherence requires that $p(E|H)\in[0,1]$.
 
In the subjective approach to probability of de Finetti no algebraic structure of events is required. For each (conditional) event  of interest, the uncertainty  is directly evaluated in terms of a degree of belief, by means of coherent probability. This evaluation concerns only the (conditional) events of interest, without the necessity of evaluating degrees of belief of  all  events (possibly unrealistic many and irrelevant ones) of a presupposed suitable algebra. This approach is therefore more flexible, epistemically economic, and more realistic compared to the approaches which presuppose to give probability values to each element of the whole algebraic structure. Hence, as degrees of belief are primitive, conditional probabilities are also primitive. Moreover,  any event $E$ coincides with the conditional event $E|\Omega$ and hence the (unconditional) probability  $p(E)$ coincides with  the conditional probability $p(E|\Omega)$. Therefore,  conditional probability is primitive in our approach and does not necessarily require a basis on  unconditional probabilities.  

 For other axiomatic approaches to conditional probability which allow for such zero probabilities but which  presuppose---differently from coherence---an algebraic structure 
  see e.g.,
\cite{Csaszar55,Dubins75,popper59,renyi55,Renyi56}.
For a discussion of different axiomatic approaches and interpretations of conditional probability see \cite{easwaran19}. 
 We recall that  a coherent  assessment $\P$ on an
arbitrary family $\F$  of conditional events---possibly without any algebraic structures---can always be extended as a full axiomatic (finitely additive) conditional probability $p$ on $\mathcal{A}\times \mathcal{A}^0$, where $\mathcal{A}$ is a Boolean algebra and $\mathcal{A}^0=\mathcal{A}\setminus \{\bot\}$, such that,  for all $E|H\in \F$, it holds that $E\in \mathcal{A}$, $H\in \mathcal{A}^0$, and  $p(E|H)=\P(E|H)$
(\cite{Rega85}; see also \cite{coletti02,Holz85,Rigo88}).
Moreover, given a real-valued function  $p$ on $\mathcal{A}\times \mathcal{B}$, where $\mathcal{A}$ is a Boolean algebra and $\mathcal{B}$ is an arbitrary nonempty subset of $\mathcal{A}^0$ (meaning  that  no restrictions are made for the class of conditioning events $\mathcal{B}$), 
which satisfies the following properties of an axiomatic (finitely additive) conditional probability:
 \begin{enumerate}[$(i)$]
 \item   $p(\cdot|H)$ is a finitely additive probability on $\mathcal{A}$, for each $H\in \mathcal{B}$;
 \item $p(H|H)$=1, for each $H\in \mathcal{B}$;
 \item $p(AB|H) = p(A|BH)p(B|H)$,  for every $A,B,H,$ with $A, B\in \mathcal{A}$ and  
  $BH,H\in \mathcal{B}$,
 \end{enumerate}
then  it could be that the function $p$ is not coherent (\cite{coletti02,Gili95a,GiSp92}). However,  Rigo (\cite{Rigo88}) has shown that Cs\`{a}sz\`{a}r's condition (\cite{Csaszar55}, see also \cite{Gili04})  is necessary and sufficient for
coherence of $p$ on $\mathcal{A}\times \mathcal{B}$.
Furthermore,  
a function $p:\mathcal{A}\times \mathcal{B}$ satisfying $(i)$, $(ii)$ and $(iii)$, is coherent, when 
$\mathcal{B}$ has a particular structure, for instance, if 
 $\mathcal{B}$ is additive (\cite{coletti02,Holz85}), or if  $\mathcal{B}$ is quasi-additive
(\cite{gili89,Sanf12}).

In the present context dealing with zero probability antecedents will be important for analyzing the validity of the probabilistic syllogisms and  for investigating 
probabilistic  existential import assumptions. We also interpret the premise set of each syllogism as  a  suitable (precise/imprecise) conditional probability assessment on the respective sequence of conditional events (without presupposing particular algebraic structures). Specifically,  we are interested in the probability propagation from the premise set to the conclusion. Coherence provides therefore tools to systematically investigate  these aspects.

Traditional categorical syllogisms  are valid argument forms consisting of two premises and a conclusion, which are composed of
basic syllogistic sentence types   (see, e.g., \cite{pfeifer06}): \emph{Every  $a$ is $b$} (A),
\emph{No  $a$ is $b$} (E),
\emph{Some $a$ is $b$} (I), and
\emph{Some $a$ is not $b$} (O), where ``$a$'' and ``$b$'' denote two of the three categorical terms $M$ (``middle term''), $P$  (``predicate term''), or $S$  (``subject term''). 
As an example of sentence type A consider \emph{Every man is mortal}. 
 The $M$ term appears only in the premises and is combined with $P$ in the first premise (``major premise'') and $S$ in the second premise (``minor premise''). 
In the conclusion only the  $S$ term and  the $P$ term appear, traditionally in the fixed order $S$--$P$.  
By all possible permutations of the predicate order, four syllogistic \emph{figures} result under the given restrictions. Following Aristotle's \emph{Analytica Priora}, we will focus on the first three figures. Specifically, on the traditionally valid Aristotelian syllogisms of  Figure I, II, and III (see Table \ref{TAB:AristSyl}).  
\begin{table}

  \begin{minipage}{\textwidth}\centering 
	\begin{tabular}{lll}      \hline
\multicolumn{3}{c}{Figure I (term order: $M$--$P$, $S$--$M$, \emph{therefore} $S$--$P$)}\\\hline
AAA & Barbara & \emph{Every M is P, Every S is M, therefore Every S is P}\\
AAI\textsuperscript{*} & Barbari &\emph{Every M is P, Every S is M, therefore Some S is P}\\
AII & Darii &\emph{Every M is P, Some S is M, therefore Some S is P}\\
EAE & Celarent &\emph{No M is P, Every S is M, therefore No S is P}\\
EAO\textsuperscript{*}\, & Celaront &\emph{No M is P, Every S is M, therefore Some S is not P}\\
EIO & Ferio &\emph{No M is P, Some S is M, therefore Some S is not P}\\
 \hline

\multicolumn{3}{c}{Figure II (term order: $P$--$M$, $S$--$M$, \emph{therefore} $S$--$P$)}\\\hline

AEE & Camestres &\emph{Every P is M, No S is M, therefore No S is P}\\ 
AEO\textsuperscript{*} & Camestrop \, &\emph{Every P is M, No S is M, therefore Some S is not P}\\
AOO & Baroco &\emph{Every P is M, Some S is not M, therefore Some S is not P}\\
EAE & Cesare &\emph{No P is M, Every S is M, therefore No S is P}\\
EAO\textsuperscript{*} & Cesaro &\emph{No P is M, Every S is M, therefore Some S is not P}\\
EIO & Festino &\emph{No P is M, Some S is M, therefore Some S is not P}\\
\hline

\multicolumn{3}{c}{Figure III (term order: $M$--$P$, $M$--$S$, \emph{therefore} $S$--$P$)}\\\hline
AAI\textsuperscript{*} & Darapti &\emph{Every M is P, Every M is S, therefore Some S is P}\\
AII & Datisi &\emph{Every M is P, Some M is S, therefore Some S is P}\\
IAI & Disamis &\emph{Some M is P, Every M is S, therefore Some S is P}\\
EAO\textsuperscript{*} & Felapton &\emph{No M is P, Every M is S, therefore Some S is not P}\\
EIO & Ferison &\emph{No M is P, Some M is S, therefore Some S is not P}\\
OAO & Bocardo &\emph{Some M is not P, Every M is S, therefore Some S is not P}\\
\hline
\end{tabular}
\end{minipage}
	\caption{Traditional and logically valid Aristotelian syllogisms. \textsuperscript{*} denotes syllogisms with implicit existential import assumptions.}
\label{TAB:AristSyl}
\end{table}
 Consider  \emph{(Modus) Barbara}, which is  a valid syllogism of Figure~I:  \emph{Every 
 	$M$ is $P$, Every $S$ is $M$, 
 	therefore Every $S$ is $P$}. 

Note that some traditionally valid syllogisms require existential import assumption for the validity. For example, Barbari (\emph{Every 
	$M$ is $P$, Every $S$ is $M$, 
	therefore Some $S$ is $P$}) is valid under the assumption that the $S$ term is not ``empty'' (in the sense that there is some $S$).
   The names of the syllogisms  traditionally encode logical properties. For the present purpose, we only recall that  vowels refer to the syllogistic sentence type: for instance,  B\underline{a}rb\underline{a}r\underline{a} involves three sentences of type~A, i.e., AAA (for details see, e.g., \cite{pfeifer06}).  
   
In our approach we interpret the syllogistic terms as events. An event  is conceived  as a bi-valued logical entity which can be true or false.
Moreover, we associate  (ordered) pair of  terms $S$--$P$ with the corresponding conditional event $P|S$, that is as a tri-valued logical object (\cite{gilio16}).
 
For giving a probabilistic interpretation of the premises and the conclusions of the syllogisms,   we interpret basic syllogistic sentence types A, E, I, O by suitable  imprecise conditional probability assessments. 
Specifically, we  interpret the degree of belief in syllogistic sentence A by $p(P|S)=1$, E by $p(P|S)=0$,  I by $p(P|S)>0$, and we interpret O by $p(\widebar{P}|S)>0$ (Table~\ref{Table:BS}; see  also \cite{chater99,gilio16}). 
Thus, A and E are interpreted as precise probability assessments and I and O by imprecise probability assessments.  
\begin{table}	
	\centering %
	\begin{tabular}{llcc}\hline 
		Type \quad & \quad Sentence \quad & Probabilistic interpretation \quad&\quad Equivalent interpretation    \\
		(A) & \emph{Every $S$ is $P$} & $p(P|S)=1$  & $p(\no{P}|S)=0$   \\
		(E) & \emph{No $S$ is $P$ (Every $S$ is not $P$)}& $p(P|S)=0$ & $p(\no{P}|S)=1$\\
		(I) & \emph{Some $S$ is $P$}& $p(P|S)>0$ & $p(\no{P}|S)<1$\\
		(O) & 
		\emph{Some $S$ is not $P$} \quad & $p(\no{P}|S)>0$ & $p(P|S)<1$\\
		\hline 
	\end{tabular}
\caption{Probabilistic interpretations of the basic syllogistic sentence types based on $P|S$ and $\no{P}|S$.}
\label{Table:BS} %
\end{table}
The basic logical relations among this interpretation of the  syllogistic sentence types are analyzed in the probabilistic Square and in the probabilistic  Hexagon of Opposition  (\cite{PS17SH,2016:SMPS2}).
\begin{remark}\label{REM:Asymmetry}
 We note that $p(P|S)$ does not constrain $p(S|P)$. Indeed, as we will show in Proposition \ref{PROP:ASYMMETRY},  given two logically independent events $S$ and $P$, the probability assessment $(x,y)$ on $(P|S,S|P)$, where $x=p(P|S)$ and  $y=p(S|P)$, is coherent for every $(x,y)\in[0,1]^2$.  Therefore, the interpretation of all the basic syllogistic sentence types in terms of conditional probabilities is not symmetric. For instance, $p(P|S)>0$ does not constrain $p(S|P)>0$, and hence \emph{Some $S$ 
 is $P$} does not imply \emph{Some $P$ is $S$}.
 Moreover, the interpretation of (A), (E), (I), (O) in terms of the conditional probabilities is weaker than in terms of   conjunction probabilities. For instance, the sentence type (I) interpreted by   $p(S\wedge P)>0$
 implies 
 $p(P|S)>0$  but not \emph{vice versa}.
 Indeed, as coherene requires that  
\begin{equation}\label{EQ:CONG_I}
 p(S\wedge P)=p(P|S)p(S),
 \end{equation}
 when $p(S\wedge P)>0$, it must be that $p(P|S)>0$. 
 Concering the converse, however, we recall that 
  in the coherence approach, 
  the assessment $p(P|S)=x$ and $p(S)=p(S\wedge P)=0$ is coherent for every $x\in[0,1]$, and in particular for $x>0$ (in such a case equation (\ref{EQ:CONG_I}) is  satisfied by $0=0$, even if $p(P|S)>0$).
 Hence, $p(P|S)>0$ does not imply that $p(S\wedge P)$ is  necessarily positive and therefore the conditional interpretation is weaker than the conjunction interpretation. Notice that, this asymmetry cannot be expressed in  approaches which require positive probability for the conditioning events.\footnote{This asymmetry is also not present in predicate logical interpretations of Aristotelian syllogisms (under appropriate existential import assumptions), since for example (I) can be equivalently expressed by ``\emph{for at least one $x$: $x$ is $S$ and $x$ is $P$}'' and by  ``\emph{it is not the case that for all $x$: if $x$ is $S$ then $x$ is not-$P$}''.}
\end{remark}

 For framing the Aristotelian syllogisms as  probabilistic 
inferences, we define validity of  probabilistic inferences. We recall that in classical logic some  Aristotelian syllogisms, like Barbari,  require existential import assumptions for logical validity (marked by * in Table~\ref{TAB:AristSyl}).  In the present approach we  require probabilistic versions of existential import assumptions for the validity of \emph{all} traditionally valid syllogisms. 
For example, we do not only require an existential import assumption for syllogisms like Barbari but also for syllogisms like Barbara.
Indeed, from the probabilistic premises of Barbari and Barbara,  i.e., $p(P|M)=1$ and  $p(M|S)=1$, 
we cannot validly infer the respective conclusion because
only a non-informative conclusion can be obtained,  i.e., 
every value of $p(P|S)$ in $[0,1]$  is coherent. 
In order to validate the conclusions of Barbari and Barbara, that is $p(P|S)>0$ and $p(P|S)=1$, respectively, we add the probabilistic constraints $p(S|(S\vee M))>0$ as a probabilistic existential import assumption. 

Based on a generalization of de Finetti's fundamental theorem  to (precise and imprecise) conditional probability, 
we study the coherent probability propagation rules of  argument forms of the syllogistic Figures I, II, and III.
These results allow to show, 
for all three Figures,  that each traditionally valid syllogism is also valid in our coherence-based probability semantics. 
Moreover, 
we build a  bridge from our probability semantics of Aristotelian syllogisms to nonmonotonic reasoning by interpreting the basic syllogistic sentence types by suitable defaults (A: $S\normally P$, E: $S\normally \no{P}$) and negated defaults (I: $S\nnormally \no{P}$, O: $S\nnormally P$).   
We also show how the proposed semantics can be used to analyze syllogisms involving generalized quantifiers (like \emph{most $S$ are $P$}).

 The paper is organized as follows.
 In Section~\ref{SEC:PN}
 we recall  preliminary notions and results on the  coherence of conditional  probability assessments and recall an algorithm for  coherent probability propagation. 
In Section \ref{SEC:VALIDITY} we define  \emph{validity} and \emph{strict validity} of probabilistic inferences and  probabilistic notions of the existential import, which is required for the validity of the syllogisms. 
In sections \ref{SEC:WT}, \ref{SEC:FIGUREII}, and  \ref{SEC:FIGUREIII} we study the coherent probability propagation rules of  argument forms of the syllogistic Figures I, II, and III, respectively. Then, we show for all three Figures  that each traditionally valid syllogism is also valid in our coherence-based probability semantics.  
In Section~\ref{SEC:NM} we build a  bridge from our probability semantics of Aristotelian syllogisms to nonmonotonic reasoning by interpreting the basic syllogistic sentence types by suitable defaults and negated defaults. Then we discuss  Aristotle's methods of proof: we show why  \emph{reductio} by \emph{conversion}  does not hold   and to what extent \emph{reductio ad impossibile} holds in our approach in Section~\ref{SEC:CONV}.   
In Section~\ref{SEC:GQ} we show how the proposed probability propagation rules can be used to analyze syllogisms involving generalized quantifiers (like \emph{Most}). 
Section~\ref{SEC:CONCL} concludes the paper by a brief summary of the main results and by an outlook to future work.

\section{Preliminary notions and results on coherence.} 
\label{SEC:PN}
In this section we recall selected key features of coherence (for more details see, e.g., \cite{biazzo00,BiGS12,coletti02,CoSV15,gilio13,gilio13ins,PeVa17,SPOG18}).
We denote events (which can be true or false) and their indicators (which can be 1 or 0) by the same symbols (e.g., the indicator of the event $E$ is denoted by the same symbol $E$).
Given two events $E$ and  $H$, with $H\neq \bot$,  the \emph{conditional event} $E|H$ is defined as a three-valued logical entity which is \emph{true} if $EH$ (i.e., $E\wedge H$) is true, \emph{false} if $\widebar{E}H$ is true, and  \emph{void} if $H$ is false.
\paragraph{Coherence and betting scheme}
In betting terms, assessing $p(E|H)=x$ means that, for every real number $s$,  you are willing to pay 
an amount $s\cdot x$ and to receive $s$, or 0, or $s\cdot x$, according
to whether $EH$ is true, or $\widebar{E}H$ is true, or $\widebar{H}$
is true (i.e., the
bet is called off), respectively. In these cases the random gain is $G=sH(E-x)$.
More generally speaking, 
consider a
real-valued function $p : \; \mathcal{K} \, \rightarrow\, \mathbb{R}$,
where $\mathcal{K}$ is an arbitrary (possibly not finite) family of
conditional events.
Let  $\mathcal{F} = (E_{1}|H
_{1}, \ldots, E_{n}|H_{n})$ be a sequence of $n$ conditional events, where $E_{j}|H_{j}\in\mathcal{K}$,
$j=1,\ldots,n$,  and let $\mathcal{P} =(p_{1}, \ldots, p_{n})$ be the vector of values   $p_{j} = p(E_{j}|H_{j})$, where $j = 1, \ldots, n$. We denote by
$\mathcal{H}_{n}$ the disjunction $H_{1} \vee\cdots\vee H_{n}$.  With
the pair $(\mathcal{F}, \mathcal{P})$ we associate the random gain
$G = \sum_{j=1}^{n} s_{j}H_{j}(E_{j} - p_{j})$, where
$s_{1}, \ldots, s_{n}$ are $n$ arbitrary real numbers. $G$
represents the net gain of $n$ transactions. Let $\mathcal{G}_{\mathcal{H}_{n}}$ denote the set of possible values of 
$G$ restricted to $\mathcal{H}_{n}$, that is, the  values of 
$G$ when at least one conditioning event is true (bet is not called off).
\begin{definition}
\label{COER-BET}
Function $p$ defined on $\mathcal{K}$ is  \emph{coherent}
if and only if, for every integer $n$, for every sequence $
\mathcal{F}$ of $n$ conditional events in $\mathcal{K}$ and for every
$s_{1}, \ldots, s_{n}$, it holds that: $\min \mathcal{G}_{\mathcal{H}_{n}}
\leq0 \leq\max \mathcal{G}_{\mathcal{H}_{n}}$.
\end{definition}
Intuitively,  Definition \ref{COER-BET} means in betting terms that a probability assessment is
coherent if and only if, in any finite combination of $n$ bets, it cannot happen that
the values in  $\mathcal{G}_{\mathcal{H}_{n}}$---that is the values of the random gain by ignoring  the cases where the bet is called off---are all
positive, or all negative (\emph{no Dutch Book}).
\paragraph{Geometrical interpretation of coherence}
Coherence can also be characterized geometrically. Let $\mathcal{F} = (E_{1}|H
_{1}, \ldots, E_{n}|H_{n})$.
As $\Omega=E_jH_j \vee \widebar{E}_jH_j \vee \widebar{H}_j   \,,\;\; j = 1, \ldots,
n$, it holds that $\Omega =\bigwedge_{j=1}^n(E_jH_j \vee \widebar{E}_jH_j \vee \widebar{H}_j)$.
By  applying the  distributive property  it follows that 
$\Omega$
can also  be written as the disjunction of $3^n$ logical
conjunctions, some of which may be impossible.  The remaining ones
are the constituents, generated by  $\mathcal{F}$ and, of course, form a partition of $\Omega$. We denote by
$C_1, \ldots, C_m$ the constituents contained in $\mathcal{H}_n$ and (if
$\mathcal{H}_n \neq \Omega$) by $C_0$ the remaining constituent
$\widebar{\mathcal{H}}_n =
\widebar{H}_1 \cdots \widebar{H}_n $, so that
\[
\mathcal{H}_n = C_1 \vee \cdots \vee C_m \,,\;\;\; \Omega =
\widebar{\mathcal{H}}_n \vee
\mathcal{H}_n = C_0 \vee C_1 \vee \cdots \vee C_m \,,\;\;\; m+1 \leq 3^n
\,.
\]
Let $\P=(p_1,\ldots,p_n)$, where $p_j=P(E_j|H_j)$, $j=1,\ldots,n$.  For  each constituent $C_h$, $h=1,\ldots,m$, 
we associate a point
$Q_h = (q_{h1}, \ldots, q_{hn})$, where $q_{hj} = 1$, or 0, or $p_j$, according to whether $C_h \subseteq E_jH_j$, or $C_h \subseteq \no{E}_jH_j$, or $C_h \subseteq \no{H}_j$. 
The point $Q_0=\P$
is associated with $C_0$. We say that the points $Q_0,Q_1,\ldots, Q_m$ are assoaciated with the pair $(\F,\P)$.
For an instance on how the constituents and the associated points are generated we
consider  the following
\begin{example}\label{EX:RUNEXAMPLE}
 Let  $\F=(E_1|H_1,E_2|H_2)=(C|B,B|A)$, where $A,B,C$ are three logically independent events, and $\P=(p_1,p_2)$ be a probability assessment on $\F$.  It holds that: 
\[
\begin{array}{ll}
\Omega=(BC\vee B\no{C} \vee \no{B})\wedge (AB\vee A\no{B} \vee \no{A})=C_0\vee C_1\vee \cdots \vee C_5,
\end{array}
\]
where the constituents are  $C_1=ABC$, $C_2=\no{A}BC$, $C_3=AB\no{C}$, $C_4=\no{A}B\no{C}$,  $C_5=A\no{B}$, and 
$C_0=\no{A}\,\no{B}$. We observe that  $\mathcal{H}_2=C_1\vee \cdots \vee C_5=A\vee B$. Moreover, the points  $Q_1=(1,1)$, $Q_2=(1,p_2)$, $Q_3=(0,1)$, $Q_4=(0,p_2)$, $Q_5=(p_1,0)$, and $Q_0=\P=(p_1,p_2)$ are associated with $(\F,\P)$.
\end{example}
Denoting by $\mathfrak{I}$ the convex hull of
$Q_{1}, \ldots , Q_{m}$, by a suitable alternative theorem
(Theorem~2.9 in \cite{Gale60}), the condition $\mathcal{P}
\in \mathfrak{I}$ is equivalent to the condition $\min \mathcal{G}_{\mathcal{H}_n} \leq 0 \leq \max \mathcal{G}_{\mathcal{H}_n}$ given
in Definition~\ref{COER-BET} 
(see, e.g., \cite{Gili96,gilio13ins}). Moreover, the condition $\mathcal{P}
\in \mathfrak{I}$ amounts to the solvability of the following system $(\Systemsigma)$ in the unknowns $\lambda_{1}, \ldots , \lambda_{m}$
\begin{equation}
\begin{array}{ll}
\label{SYST-SIGMA}
(\Systemsigma) :
\quad \quad
\sum_{h=1}^{m} q_{hj} \lambda_{h} = p_{j} \; , \; \; j\in J_{n} \,
;
\;\; \sum_{h=1}^{m} \lambda_{h} = 1 \; ;\; \; \lambda_{h} \geq 0 \, ,
\; h \in J_{m}\,,
\end{array}
\end{equation}
where,  $J_n=\{1,2,\ldots,n\}$, for every integer $n$.
We say that system $(\Systemsigma) $ is associated with the pair $(
\mathcal{F},\mathcal{P})$. Hence, the following result provides
a characterization of the notion of coherence given in Definition~\ref{COER-BET} (Theorem~4.4 in \cite{Gili90}; see also
\cite{Gili92,GSisipta11,gilio13ins}).

\begin{theorem}
	\label{CNES}
	The function
	$p$ defined on an arbitrary family of conditional events $\mathcal{K}$ is coherent if and only if, for every finite subsequence  $\mathcal{F}=(E_{1}|H_{1}, \ldots , E_{n}|H_{n})$ of $
	\mathcal{K}$, denoting by $\mathcal{P}$ the vector $(p_{1},
	\ldots , p_{n} )$, where $p_{j}=p(E_{j}|H_{j})$, $j=1,2,\ldots ,n$, the
	system $(\Systemsigma) $ associated with the pair $(\mathcal{F},
	\mathcal{P})$ is solvable.
\end{theorem}
\paragraph{Coherence checking}
We recall now some results on the coherence checking of a probability
assessment on a finite family of conditional events. Given a probability
assessment $\mathcal{P}=(p_{1}, \ldots , p_{n} )$ on  $\mathcal{F}=(E_{1}|H_{1}, \ldots , E_{n}|H
_{n})$, let $\mathcal{S}$ be the set of solutions of the form $\Lambda = (\lambda_{1},
\ldots , \lambda_{m})$ of the system $(\Systemsigma) $. Then, assuming
$\mathcal{S} \neq \emptyset $, we define
\begin{equation}\label{EQ:PHI}
\begin{array}{l}
\Phi_{j}(\Lambda ) = \Phi_{j}(\lambda_{1}, \ldots , \lambda_{m}) =
\sum_{r :
	C_{r} \subseteq H_{j}} \lambda_{r} \; ,
\; \; \;
j \in J_{n} \,;\; \Lambda \in \mathcal{S} \,;
\\
M_{j} = \max_{\Lambda \in \mathcal{S} } \; \Phi_{j}(\Lambda ) \; ,
\; \; \;
j\in J_{n},
\end{array}
\end{equation}
and 
\begin{equation}\label{EQ:I0}
    I_{0} = \{ j\in J_n \, : \, M_{j}=0 \}.
\end{equation}
Assuming $\mathcal{P}$ coherent, each
solution $\Lambda =(\lambda_{1}, \ldots , \lambda_{m})$ of system
$(\Systemsigma) $ is a coherent extension of the assessment $\mathcal{P}$ on $\mathcal{F}$ to the sequence $(C_{1}|\mathcal{H}_{n},C
_{2}|\mathcal{H}_{n},\, \ldots ,\, C_{m}|\mathcal{H}_{n})$. 
Then, for each solution $\Lambda $ of system
$(\Systemsigma) $ the quantity $\Phi_{j}(\Lambda )$ is a coherent extension of the
conditional probability $p(H_{j}|\mathcal{H}_{n})$. Moreover, the quantity
$M_{j}$ is the upper probability $p''(H_{j}|\mathcal{H}_{n})$ over all
the solutions $\Lambda $ of system $(\Systemsigma) $. Of course, $j \in I
_{0}$ if and only if $p''(H_{j}|\mathcal{H}_{n})=0$. Notice that
$I_{0}$ is a strict subset of $J_{n}$.
If $I_{0}$ is nonempty,
we set  $\mathcal{F}_{0}=(E_i|H_i\in \F,i\in I_0)$ and $\mathcal{P}_{0}=(p(E_i|H_i), i\in I_0)$.
We say that   the pair  $(\mathcal{F}_{0},\mathcal{P}_{0})$  is associated with $I_{0}$.
Then, we have (Theorem~3.3 in \cite{Gili93}):
\begin{theorem}
	\label{COER-P0}The assessment $\mathcal{P}$ on $\mathcal{F}$ is coherent
	if and only if the following conditions are satisfied: (i) 	the system $(\Systemsigma) $ associated with the pair $(\mathcal{F},
	\mathcal{P})$ is solvable; (ii) if $I_{0} \neq \emptyset $, then
	$\mathcal{P}_{0}$ is coherent.
\end{theorem}
Let  $\mathcal{S}'$ be a nonempty subset of the set of solutions
$\mathcal{S}$ of system $(\Systemsigma) $. We denote by  $I_0'$ the set $I_0$ defined as in (\ref{EQ:I0}), where $\mathcal{S}$ is replaced by $\mathcal{S}'$, that is
\begin{equation}\label{EQ:I0'}
I_{0}' = \{ j\in J_n \, : \, M'_{j}=0 \}, \mbox{ where } M'_{j}=\max_{\Lambda \in \mathcal{S}' } \; \Phi_{j}(\Lambda ) \; ,
\; \; \;
j\in J_{n}.
\end{equation}
 Moreover, we denote by $(\F_0',\P_0')$ the pair associated with $I_0'$. Then, we obtain (see, e.g., Theorem 7 in \cite{BiGS03a})
\begin{theorem}
	\label{COER-P0'}The assessment $\mathcal{P}$ on $\mathcal{F}$ is coherent
	if and only if the following conditions are satisfied: (i) 	the system $(\Systemsigma) $ associated with the pair $(\mathcal{F},
	\mathcal{P})$ is solvable; (ii) if $I_{0}' \neq \emptyset $, then
	$\mathcal{P}_{0}'$ is coherent.
\end{theorem}
For an illustration of  Theorem \ref{COER-P0'} we consider\\ \ \\
\noindent {\bf Example 1 (continued).}\emph{
We observe that 
$\P=(p_1,p_2)=p_1Q_2+(1-p_1)Q_4$ and 
$\P=p_1p_2Q_1+(1-p_1)p_2Q_3+(1-p_2)Q_5
$. 
Then, $\P=\frac{1}{2}( p_1p_2Q_1+p_1Q_2+
(1-p_1)p_2Q_3+
(1-p_1)Q_4 
+(1-p_2)Q_5)
$. 
By assuming that $(p_1,p_2)\in[0,1]^2$, it follows that the system
$(\Systemsigma)$  associated with $(\F,\P)$ is solvable with a solution given by $\Lambda=\frac{1}{2}(p_1p_2,p_1,(1-p_1)p_2,(1-p_1),(1-p_2))$.
Moreover, $\Phi_{1}(\Lambda)=\lambda_1+\lambda_2+\lambda_3+\lambda_4=\frac{p_2}{2}+\frac{1}{2}>0$ and  
$\Phi_{2}(\Lambda)=\lambda_1+\lambda_3+\lambda_5=\frac{p_1p_2}{2}+\frac{(1-p_1)p_2}{2}+
+\frac{1-p_2}{2}=\frac{1}{2}>0$. Then,
by setting
 $\mathcal{S}'=\{\Lambda\}$ it holds that $M_1'>0$, $M_2'>0$ and hence $I_0'=\emptyset$. Thus, by Theorem \ref{COER-P0'} the assessment $(p_1,p_2)$ is coherent for every $(p_1,p_2)\in[0,1]^2$.}
\paragraph{Algorithm for probability propagation}
We recall the following extension theorem for  conditional probability, which is a generalization of de Finetti's fundamental theorem of probability  to conditional events  (see, e.g., \cite{biazzo00,CoSc96,Holz85,Pelessoni09,Rega85,williams75}). 
\begin{theorem}\label{THM:FUND}
Let a coherent probability assessment $\P= (p_1, \ldots , p_n)$ on a sequence $\F=(E_1|H_1,\ldots ,E_n|H_n)$
and  a further conditional event $E_{n+1} |H_{n+1}$ be given. Then, there exists a suitable closed interval $[z',z'']\subseteq [0,1]$ such that the extension   $(\P,z)$ of $\P$ to  $(\F,E_{n+1} |H_{n+1})$ is coherent if and only if $z \in [z',z'']$.
\end{theorem}
Theorem \ref{THM:FUND}   states
that
a coherent assessment of premises  can always be coherently extended to a conclusion, specifically there always exists an interval $[z',z'']\subseteq [0,1]$ of all coherent extensions on the  conclusion.  
A non informative or illusory restriction is obtained when 
$[z',z'']=[0,1]$. The extension is unique when  $z'=z''$. For applying Theorem~\ref{THM:FUND},  we now recall an algorithm (see Algorithm~1 in \cite{gilio16}, which is originally based on Algorithm~2 in \cite{biazzo00}) which allows for  computing the lower and upper bounds $z'$ and $z''$ of the  interval  of all coherent extensions  on  $E_{n+1} |H_{n+1}$. 
\begin{algorithm}\label{Alg}{\rm
		Let  $\F=(E_1|H_1,\ldots ,E_n|H_n)$ be  a sequence of conditional events and $\mathcal{P}=(p_1,\ldots,p_n)$ be a coherent precise probability assessment    
		on  $\F$, where  $p_j=p(E_j|H_j)\in[0,1]$,  $j=1,\dots,n$. 
		Moreover, let $E_{n+1}|H_{n+1}$ be  a further conditional event  and denote by $J_{n+1}$ the set
		$\{1, \ldots, n+1\}$. The steps below describe the computation of the lower bound $z'$ (resp., the upper bound $z''$) for the coherent extensions $z=p(E_{n+1}|H_{n+1})$.
		\begin{itemize}
			\item {\bf Step 0.} Expand the expression	$	\bigwedge_{j \in J_{n+1}} \left( E_j H_j \vee \widebar{E}_j H_j \vee \widebar{H}_j \right)$
			and denote by $C_1, \ldots, C_m$ the constituents contained in $\mathcal{H}_{n+1} =
			\bigvee_{j \in J_{n+1}} H_j$ associated with $(\F,E_{n+1}|H_{n+1})$. Then, construct the following system in
			the unknowns $\lambda_1, \ldots, \lambda_m, z$
			\begin{equation}\label{S_{n+1}}
			\begin{array}{l}
			\left\{
			\begin{array}{l}
			\sum_{r: C_r \subseteq E_{n+1}H_{n+1}}\lambda_r = z
			\sum_{r: C_r \subseteq H_{n+1}}\lambda_r \; ; \\
			\sum_{r: C_r \subseteq E_{j}H_{j}}\lambda_r = p_{j}
			\sum_{r: C_r \subseteq H_{j}}\lambda_r, \;  \; j \in J_n\;; \\
			\sum_{r\in J_m}\lambda_r = 1;\;\;
			\lambda_r \geq 0 , \; r\in J_m\;.
			\end{array}
			\right.
			\end{array}
			\end{equation}
			
			\item {\bf Step 1.}  Check the solvability of system
			(\ref{S_{n+1}})
			under the condition ${z = 0}$ (resp., ${z = 1}$).
			If it is not solvable, go to Step 2; otherwise,
			go to Step 3.
			\item {\bf Step 2.} Solve the following linear programming problem
			\[
			\mbox{Compute}: \hspace{1 cm}
			\gamma' \; = \; \min  \; \sum_{r: C_r \subseteq E_{n+1}H_{n+1}}\lambda_r
			\]
			\[
			\mbox{(respectively :} \hspace{0.7 cm}
			\gamma'' \; = \; \max  \;
			\sum_{r: C_r \subseteq E_{n+1}H_{n+1}} \lambda_r \; )
			\]
			subject to:
			\[
			\left\{
			\begin{array}{l}
			\sum_{r: C_r \subseteq E_{j}H_{j}}\lambda_r \;
			= \; p_j\sum_{r: C_r \subseteq H_{j}}\lambda_r, \; j \in J_n\; ; \\
			\sum_{r: C_r \subseteq H_{n+1}}\lambda_r = 1 ; \;
			\lambda_r \geq 0, \; r \in J_m\;.
			\end{array}
			\right.
			\]
			The minimum $\gamma'$ (respectively the maximum  $\gamma''$)
			of the  {\em objective function} coincides with $z'$ (respectively with
			$z''$) and the procedure stops.
			\item {\bf Step 3.} For each subscript $j\in J_{n+1}$, compute the maximum $M_j$ of the
			function $\Phi_j=\sum_{r: C_r\subseteq H_j}\lambda_r$, subject to the constraints given by the system
			(\ref{S_{n+1}}) with $z = 0$ (respectively $z = 1$). 
			We have the following three cases:
			\begin{enumerate}
				\item $M_{n+1} > 0\;$;
				\item $M_{n+1} = 0 \; , \; M_j > 0 \;$ for every $ \;j \neq n+1\;$;
				\item $M_j = 0 \;$ for $\; j \in I_0 = J \cup \{n+1\} \;$, with
				$\; J \neq \emptyset\;$.
			\end{enumerate}
			In the first two cases $z' = 0$ (respectively $z'' = 1$) and the
			procedure stops. \\ In the third case, defining $I_0 = J \cup \{n+1\}, \;$ set
			$J_{n+1} = I_0$ and $(\F,\P) = (\F_J,\P_J)$, where $\F_J=(E_i|H_i:i\in J)$ and $\P_J=(p_i:i\in J)$. Then, go to Step 0.
		\end{itemize}
}\end{algorithm}
The procedure ends in a finite number of cycles by computing the value
$z'$ (respectively $z''$).
\begin{remark}\label{REM1}
Assuming $(\P,z)$ on $(\F, E_{n+1}|H_{n+1})$  coherent, each solution $\Lambda=(\lambda_1,\ldots,\lambda_m)$ of System (\ref{S_{n+1}}) is a coherent extension of the assessment $(\P,z)$  to the sequence $(C_1|\mathcal{H}_{n+1},\ldots,C_m|\mathcal{H}_{n+1})$.
\end{remark}
For a software implementation of an algorithm based on  \cite{Capotorti02,coletti02}, which is similar to Algorithm \ref{Alg}, see the {\em Check-Coherence software} (\cite{capotorti16}).
\paragraph{Imprecise probability}
Now, we recall the notion of an imprecise (probability) assessment in set-valued terms.
\begin{definition}\label{DEF:IA}
An \emph{imprecise, or set-valued, assessment} $\I$ on a finite sequence of $n$  conditional events   $\F$  is a (possibly empty) set of precise assessments $\P$ on $\F$.
\end{definition}
Definition~\ref{DEF:IA},  introduced in  \cite{gilio98}, states that  an \emph{imprecise (probability) assessment} $\I$ on   a finite sequence $\F$ of $n$ conditional events is just a (possibly empty)  subset of $[0,1]^n$. 
If an imprecise assessment $\I$ on $\F$, with  
$\I=\I_1\times \cdots \times \I_n$, where $\I_i\subseteq [0,1]$, $i=1,\ldots,n$, then  
 $\I$  on $\F$ can be formulated in terms of  constraints on the probability
 of the single events in $\F$, i.e.,
\begin{equation}
(p(E_1|H_1)\in \I_1,\ldots, p(E_n|H_n)\in \I_n).
\end{equation}
We recall   the notions of g-coherence and total-coherence  for  imprecise (in the sense of set-valued) probability assessments
(\cite{gilio98,gilio16}).
\begin{definition}\label{DEF:GCOH}
Let  a sequence  of $n$ conditional events  $\F$ be given. An imprecise assessment $\I\subseteq [0,1]^n$  on $\F$ is  \emph{g-coherent}  if and only if there exists  a coherent precise assessment  $\P$  on $\F$ such that $\P\in \I$.
\end{definition}

\begin{definition}
An imprecise assessment $\I$  on $\F$  is  \emph{totally coherent} (t-coherent) if and only if the following two conditions are satisfied: (i) $\I$ is non-empty;   (ii) if $\P\in \I$, then $\P$ is a coherent precise assessment on $\F$.
\end{definition}
We denote by $\Pi$ the set of \emph{all  coherent precise} assessments on $\F$.
We recall that if there are no logical relations among the events $E_1,H_1,\ldots, E_n,H_n$ involved in $\F$, that is   $E_1,H_1,\ldots, E_n,H_n$ are logically independent, then the set $\Pi$ associated with $\F$ is the whole unit hypercube $[0,1]^n$. 
If there are  logical relations, then  the set $\Pi$ \emph{could be} a strict subset of  $[0,1]^n$. As it is well known  $\Pi\neq \emptyset$; therefore, 
$\emptyset\neq \Pi\subseteq [0,1]^n$.
\begin{remark}
We observe that:
\[
\begin{array}{l}\label{REM:GTCOHERENCE}
\I \mbox{ is g-coherent }\; \Longleftrightarrow \; \Pi \cap \I \neq \emptyset

\\
\I \mbox{ is t-coherent } \;\Longleftrightarrow \; \emptyset \neq  \Pi \cap \I =\I\,.
\end{array}
\]
Then: $\I$ is t-coherent $\Rightarrow \I$ is  g-coherent.
\end{remark}
\begin{definition}\label{DEF:GCOHEXT}
 Let $\I$ be a g-coherent assessment on   $\F = (E_1|H_1,\ldots ,E_n|H_n)$; moreover, let $E_{n+1}|H_{n+1}$ be a further conditional event and let $\mathcal{J}$ be an extension of $\I$  to $(\F ,E_{n+1}|H_{n+1})$. We say that $\mathcal{J}$  is a \emph{g-coherent extension} of $\I$ if and only if $\mathcal{J}$ is g-coherent.
\end{definition}
Given a g-coherent assessment $\I$ on  a sequence of $n$ conditional events $\F$, for each coherent precise assessment  $\P$ on $\F$, with $\P\in \I$, we denote by $[\alpha_{\P},\beta_{\P}]$ the interval of coherent extensions of $\P$  to $E_{n+1}|H_{n+1}$; that is, the assessment  $(\P,z)$ on $(\F,E_{n+1}|H_{n+1})$ is coherent if and only if
$z\in [z'_{\P},z''_{\P}]$. Then, defining the set 
\begin{equation} \label{EQ:UEXT}
\begin{array}{ll}
\Sigma=\bigcup_{\P\in\Pi \cap \I}[z'_{\P},z''_{\P}]\,,
\end{array}
\end{equation}
for every $z\in \Sigma$, the assessment $\I\times \{z\}$ is a g-coherent extension of $\I$ to $(\F,E_{n+1}|H_{n+1})$; moreover, 
 for every $z\in [0,1]\setminus \Sigma$, the extension $\I\times \{z\}$  of $\I$ to $(\F,E_{n+1}|H_{n+1})$ is  not g-coherent. 
 We say that $\Sigma$ is the \emph{set of (all) coherent extensions}  of the imprecise assessment  $\I$ on $\F$ to the  conditional event $E_{n+1}|H_{n+1}$. Of course, as $\I$ is g-coherent,  $\Sigma \neq \emptyset$. 
\paragraph{Coherence and penalty criterion}
We recall that  de Finetti (\cite{deFi62,deFi64,definetti74})  introduced the notion of coherence (for the case of unconditional events)
by means of a penalty criterion based on the Brier quadratic scoring rule (\cite{Brie50}), which is beyond the betting scheme (for a discussion on these two 
different  justifications of probabilistic accounts of belief see, e.g.,  \cite{Joyce98}).
De Finetti also gave a geometrical interpretation of coherence and  showed that the notion of coherence based on the betting scheme is equivalent to the notion  of coherence based on the penalty criterion   (\cite{definetti74}).
The relationship between the notions of coherence and of \emph{non-dominance}, with respect to  proper scoring rules, for  the case of unconditional events has been investigated by exploiting the Bregman divergence in 
\cite{predd09}.
For related work in terms of accuracy and credence functions see, e.g., \cite{pettigrew16}.
Coherence  
based on the penalty criterion has been extended to the case of conditional events in \cite{Gili90} (see also \cite{Gili96,GiSa22}) as follows. Let the assessment $\mathcal{P} =(p_1, \ldots, p_n)$ on $\mathcal{F}=\{E_1|H_1,\ldots,E_n|H_n\}$ be associated with a random loss  $\mathcal{L}=\sum_{i=1}^n H_i(E_i-p_i)^2$ (Brier score adapted to conditional events). 
Then, the value $L_h$ of the random loss $\mathcal{L}$ when the constituent $C_h$ is true is given by
\[L_h=\sum_{i=1}^n (q_{hi}-p_i)^2,
\]
where $Q_h=(q_{h1},\ldots,q_{hn})$ is the corresponding point associated with $C_h$. Of course, $L_0=\sum_{i=1}^n (p_i-p_i)^2=0$ is associated with the constituent $C_0=\no{H}_1\cdots \no{H}_n$, which means that the loss is zero when all the conditional events are void.  Then, the following definition of coherence can be given:
\begin{definition}\label{DEF:COER-CONDPENALTY}  A function $p$ defined on an arbitrary family of conditional events $\mathcal{K}$ is said to be {\em coherent} if and only if, for every integer $n$, for every subfamily $\mathcal{F}=\{E_1|H_1,\ldots,E_n|H_n\} \subseteq \mathcal{K}$,
		denoting by $\P=(p_1,\ldots,p_n)$  the restriction of $p$ to $\F$, by $\mathcal{L}=\sum_{i=1}^n H_i(E_i-p_i)^2$ the associated random loss,  
		there does not exist $\mathcal{P}^* =(p_1^*, \ldots, p_n^*)$ such that:
		$\mathcal{L}^*\leq \mathcal{L}$ and $\mathcal{L}^*\neq \mathcal{L}$, that is
		$L_h^* \leq L_h$, for every $h$, with $L_h^* < L_h$ for at least an index $h$.
	\end{definition}
In \cite{Gili90} it has been shown that   coherence based on  the betting scheme (Definition \ref{COER-BET}) and coherence based on the penalty criterion (Definition \ref{DEF:COER-CONDPENALTY}) are equivalent (see also \cite{Gili96,GiSa22}). In other words, a function  $p$ on $\K$ is coherent according to Definition \ref{COER-BET} if and only if $p$ on $\K$ is coherent according to Definition \ref{DEF:COER-CONDPENALTY}.
We also recall that, based on Bregman divergences,  coherence for conditional events can  be characterized in terms of  admissibility with respect to any given proper scoring rule.
More precisely, in 
\cite[Theorem 3]{GSisipta11} (see also \cite{BiGS12}) it is shown that, given any bounded (strictly) proper scoring rule $s$, a  probability assessment $p$ on a family of conditional events $\K$ is coherent if and only if it is \emph{admissible} with respect to $s$.
 
 \section{Validity and existential import.}
 \label{SEC:VALIDITY}
We define the validity of a probabilistic inference rule as follows:
\begin{definition}\label{DEF:VALID}
	Given a g-coherent assessment $\I$ on  a sequence of $n$ conditional events $\F$ and  a non-empty imprecise assessment $\I'$ on a conditional event $E_{n+1}|H_{n+1}$, we say that the (probabilistic) inference 
	\[
	\mbox{from }  \I \mbox{ on } \F \mbox{ infer } \I' \mbox{ on } E_{n+1}|H_{n+1}
	\]
	is \emph{valid} (denoted by $\models$) if and only if  $\Sigma\subseteq \I'$, where $\Sigma$ is the set of coherent extensions  of the imprecise assessment  $\I$ on $\F$.
	Moreover, we call a valid inference  \emph{strictly valid (s-valid, denoted by $\models_s$)} if and only if  $\I'=\Sigma$.
\end{definition}

\begin{remark}\label{REM:VALID}
Let $	\mbox{from }  \I \mbox{ on } \F \mbox{ infer } \I' \mbox{ on } E_{n+1}|H_{n+1}$ be a valid inference, let $\I_s$ be a g-coherent  subset of   $\I$, and let $\I_w$ be a supset of  $\I'$. 
Denoting by $\Sigma_s$ the set of coherent extensions of the imprecise assessment $\I_s$, we observe  that $\emptyset\neq \Sigma_s\subseteq \Sigma$.
Then, by Definition \ref{DEF:VALID}, the following inference is valid
	\[
	\mbox{from }  \I_s \mbox{ on } \F \mbox{ infer } \I_w' \mbox{ on } E_{n+1}|H_{n+1}.
	\]
Thus, by starting from a valid inference  we obtain valid inferences if the premises are strengthened or the conclusion is weakened. 
\end{remark}  
In the next remark we explain how Adams' notion of p-validity (\cite{adams75}) is interpreted in the framework of coherence and how it relates to our notion of s-validity.
\begin{remark} \label{REM:PVALIDITY}We recall that a
finite sequence of conditional events
$\mathcal{F} = (E_{1}|H_{1},E_2|H_2 ,\ldots,E_n|H_n)$ is \emph{p-consistent}
if and only if the assessment $(1,1,\ldots,1)$ on $\mathcal{F}$ is coherent.
In addition, a p-consistent family $\mathcal{F}$ \emph{p-entails} a
conditional
event $E_{n+1}|H_{n+1}$ if and only if the unique coherent extension on $E_{n+1}|H_{n+1}$ of the
assessment $(1,1,\ldots,1)$ on $\mathcal{F}$ is $p(E_{n+1}|H_{n+1})=1$ (see, e.g., \cite{gilio13}). The inference from $\mathcal{F}$ to $E_{n+1}|H_{n+1}$ is 
\emph{p-valid} if and only if $\mathcal{F}$ p-entails $E_{n+1}|H_{n+1}$.  We observe that  p-valid inferences are   special cases of s-valid inferences, specifically  
when, in Definition \ref{DEF:VALID}, $\I=(1,1,\ldots,1)$ and $\I'=\{1\}$.
\end{remark}
\begin{definition}\label{DEF:CEI}
	The \emph{conditional event existential import assumption}  is defined by assuming that the conditional probability of the conditioning event of the minor  premise of a syllogism given the disjunction of all conditioning events of the syllogism is positive. 
\end{definition}

For Datisi, the conditional event existential import assumption is 
${p(M|(S \vee M))>0}$, which makes Datisi  probabilistically informative:
\begin{eqnarray*}
	\mbox{(Datisi)} & p(P|M)=1,\; p(S|M)>0,\mbox{ and } p(M|S \vee M)>0 \Longrightarrow  p(P|S)> 0\,.
\end{eqnarray*}

We will also consider the following   \begin{definition}\label{DEF:EI}
	The \emph{unconditional event existential import assumption}  is defined by assuming that the  probability of the conditioning event of the minor  premise  is positive. 
\end{definition}
For example, for Datisi, the unconditional event existential import assumption is 
${p(M)>0}$.
In the next remark we observe that Definition~\ref{DEF:EI} is stronger than Definition~\ref{DEF:CEI}, and hence for Datisi it means that  $p(M)>0$ implies that $p(M|S\vee M)>0$ (but not \emph{vice versa}).
\begin{remark}\label{REM:EI>CEI}
	Let $H_1,H_2$, and $H_3$ (where some of them may coincide)
	denote the conditioning event of the major premise,  the minor premise, and of the conclusion, respectively.
	Then, the unconditional event existential import assumption is $p(H_2)>0$ while 
	the conditional event existential import assumption is $p(H_2|(H_1\vee H_2 \vee H_3))>0$. 
	We observe that in general $p(H_2)=p(H_2\wedge (H_1\vee H_2 \vee H_3))=p(H_2|(H_1\vee H_2 \vee H_3))p(H_1\vee H_2 \vee H_3)$. Then,
	\begin{equation}\label{EQ:EISTRENGHT}
	p(H_2)>0 \Longrightarrow p(H_2|(H_1\vee H_2 \vee H_3))>0.
	\end{equation}
	However, the converse of (\ref{EQ:EISTRENGHT})  does not hold. Indeed,   in the framework of coherence it could be that $p(H_2)=0$ even if $p(H_2|(H_1\vee H_2 \vee H_3))>0$, because 
	$p(H_2|(H_1\vee H_2 \vee H_3))>0$, $p(H_1\vee H_2 \vee H_3)=0$, and  $p(H_2)=0$ is coherent. 
	Therefore,   Definition~\ref{DEF:EI} is stronger than Definition~\ref{DEF:CEI}. 
\end{remark}
\begin{remark}
We are aware that it is not straightforward to find a natural language pendent to our conditional event existential import assumption, while the stronger unconditional event existential import assumption can be seen as a reading of the assumption that the subject term $S$ must not be empty. 
However, we recall that in our approach we can have that $p(S|S\vee P)>0$ even if $p(S)=0$ (Remark \ref{REM:EI>CEI}). This situation cannot be represented in purely logical terms and hence there is no corresponding interpretation of the weaker existential import assumption. Moreover, we will see that, in order to prove the validity of the traditionally valid syllogisms, it is sufficient to use the conditional event existential import assumption.
We also recall that it is traditionally  at least tacitly assumed that the subject term must not be empty in Aristotelian syllogistics (for example, in his overview on Aristotle's logic, Smith claims that ``Aristotle in effect supposes that \emph{all} terms in syllogisms are non-empty'' \cite{sep-aristotle-logic}). However, we note that there are also  arguments that challenge this view, i.e.,  Aristotle in fact ``places no requirement that the terms be non-empty'' \cite[p. 543]{read15}.  
 We leave the question of whether our interpretation of the our existential import comes closer to Aristotle than  stronger ones to historians of logic.
\end{remark}

\begin{remark}
We define the  conditional event existential import assumption by considering the conditioning events of all the conditional events involved in the premises and the conclusion. Single syllogistic  sentences are interpreted by  suitable probability assessments on a corresponding  conditional events. For example, \emph{Every $S$ is $P$} by $p(P|S)=1$.  The corresponding conditional event existential import is $p(S|S)>0$. For all $S\neq \bot$ our existential import  is always satisfied because coherence requires that $p(S|S)$ must be 1, even if $p(S)=0$. 
Indeed, it can easily be proved that the assessment $(x,y)$ on $(S,S|S)$, with $S\neq \bot$, is coherent if and only if $x\in[0,1]$ and $y=1$. We also notice that in this case the equation $p(S)=p(S\wedge S)=p(S|S)p(S)$ is always satisfied. 
Sentences where the conditioning event $S$ is a contradiction, i.e. $S=\bot$, are not considered because the corresponding conditional event is undefined in our approach.
\end{remark}

\section{Figure I.}
\label{SEC:WT}
In this section,
we prove that the probabilistic inference of $C|A$ from the premise set $(C|B, B|A)$, which corresponds to the transitive structure of the general form of syllogisms of Figure I, is probabilistically non-informative. Specifically, we prove that the imprecise assessment $[0,1]^3$ on $(C|B, B|A,C|A)$ is t-coherent. 
This t-coherence implies that: $(i)$ the assessment $[0,1]^2$ on $(C|B,B|A)$ is  t-coherent, which means that  any  assessment $(x,y)\in [0,1]^2$  on the premise set $(C|B,B|A)$ is coherent;
$(ii)$ the degree of belief in the conclusion $C|A$ is not constrained by the degrees of belief in the premises, since any value $z\in[0,1]$ is a coherent extension of a given pair $(x,y)$ on $(C|B,B|A)$.
Then, in order to obtain probabilistic informativeness, we add the probabilistic constraint $p(B|(A \vee  B))>0$ to the premise set.
This constraint serves as  the conditional event existential import assumption of syllogisms of Figure I according to Definition \ref{DEF:CEI}. 
We show that the imprecise assessment  $[0,1]^3$ on  $(C|B, B|A,B|(A \vee  B))$ is t-coherent. Then, we recall  the precise and imprecise probability propagation rules for the inference from $(C|B,A|B, B|(A\vee B))$ to  $C|A$.
We apply these results  in Section~\ref{SEC:FIGISYLLO}, where we  study the valid syllogisms of Figure I.
Contrary to first order monadic predicate logic, which requires existential import assumptions for Barbari and Celaront only (see Table \ref{TAB:AristSyl}), our probabilistic existential import assumption is required for all valid syllogisms of Figure I. 
\subsection{Coherence and  probability propagation   in Figure I.}
We now prove the t-coherence of the  imprecise assessment  $[0,1]^3$ on the sequence of conditional events involved in our probabilistic interpretation of syllogisms of Figure I.
\begin{proposition}\label{PROP:FIGUREI}
Let $A,B,C$ be logically independent events.
The imprecise assessment $[0,1]^3$ on $\mathcal{F}=(C|B,B|A,C|A)$ is t-coherent.
\end{proposition}
\begin{proof}
 Let $\P=(x,y,t)\in[0,1]^3$ be  any precise assessment  on $\mathcal{F}$.
	The constituents $C_h$ and  the  points $Q_h$ associated with   $(\F,\P)$ are
given in Table \ref{TAB:TABLE_FIGUREI}. 
\begin{table}[!h]
  \begin{minipage}{\textwidth}\centering 
	\caption{Constituents $C_h$ and  points $Q_h$ associated with  the probability  assessment    $\mathcal{P}=(x,y,z)$  on 
		$(C|B,B|A,C|A)$ involved in Figure I.}
	\label{TAB:TABLE_FIGUREI}
	\begin{tabular}{llll}      \hline
		& $C_h$            & $Q_h$                          &  \\
		\hline
		$C_1$ \, & $ABC$    \,       & $(1,1,1)$         \,           & $Q_1$   \\
		$C_2$ & $AB\no{C}$      & $(0,1,0)$                    & $Q_2$   \\
		$C_3$ & $A\no{B}C$       & $(x,0,1)$                    & $Q_3$   \\
		$C_4$ & $A\no{B}\no{C}$      & $(x,0,0)$                    & $Q_4$   \\
		$C_5$ & $\no{A}BC$ & $(1,y,z)$                    & $Q_5$   \\
		$C_6$ & $\no{A}B\no{C}$  & $(0,y,z)$                  & $Q_6$   \\
		$C_0$ & $\no{A} \, \no{B}$   &$(x,y,z)$                       & $Q_0=\mathcal{P} $ \\ \hline
	\end{tabular}     
\end{minipage}
\end{table}
By Theorem \ref{COER-P0}, coherence of  $\P$ on $\F$ requires that the following system
\[\begin{array}{l}
(\Systemsigma)  \hspace{1 cm}
\mathcal{P}=\sum_{h=1}^6 \lambda_hQ_h,\;
\sum_{h=1}^6 \lambda_h=1,\; \lambda_h\geq 0,\, h=1,\ldots,6.
\end{array}
\]
is solvable. In geometrical terms, this  means that  
the condition $\P \in \mathfrak{I}$ is satisfied, where   $\mathfrak{I}$ is the convex hull of $Q_1, \ldots, Q_6$. We observe that 
$\P=xQ_5+(1-x)Q_6$, indeed it holds that $(x,y,z)=x(1,y,z)+(1-x)(0,y,z)$.  Thus,  system 	$(\Systemsigma)  $ is solvable and a solution is
$\Lambda=(\lambda_1,\ldots,\lambda_6)=(0,0,0,0,x,1-x)$. 
From  (\ref{EQ:PHI}) we obtain that
$\Phi_{1}(\Lambda)=\sum_{h:C_h\subseteq B}\lambda_h=\lambda_1+\lambda_2+\lambda_5+\lambda_6=x+1-x=1$, 
 $\Phi_{2}(\Lambda)=\Phi_{3}(\Lambda)=\sum_{h:C_h\subseteq A}\lambda_h=
\lambda_1+\lambda_2+\lambda_3+\lambda_4=0$. Let $\mathcal{S}'=\{(0,0,0,0,x,1-x)\}$ denote a subset of the set $\mathcal{S}$ of all solutions of $(\Systemsigma) $.   Then, ${M_1'=\max\{\Phi_{1}:\Lambda\in\mathcal{S}'\}>0}$ and hence $I_0'=\{2,3\}$ (as defined in (\ref{EQ:I0'})).
 By Theorem \ref{COER-P0'},
 as $(\Systemsigma) $ is solvable and $I_0'=\{2,3\}$,
  it is sufficient to check the coherence of 
 the sub-assessment  $\P_0'=(y,z)$ on $\F_0'=(B|A,C|A)$ in order to check the coherence of $(x,y,z)$.
The constituents $C_h$ associated with the new pair $((B|A,C|A),(y,z))$ contained in  $\mathcal{H}_2=A$ are $C_1=ABC, C_2=AB\no{C}, C_3=A\no{B}C, C_4=A\no{B}\no{C}$ and the corresponding points $Q_h$ are $Q_1=(1,1), Q_2=(1,0), Q_3=(0,1), Q_4=(0,0)$.
The convex hull $\mathfrak{I}$ of the points $Q_1,Q_2,Q_3,Q_4$ is the unit square  $[0,1]^2$. 
Then $(y,z)\in[0,1]^2$ trivially belongs to $\mathfrak{I}$ and hence the system 
\[
(\Systemsigma) : \;\;
(y,z)=\lambda_1Q_1+\lambda_2Q_2+\lambda_3Q_3+\lambda_4Q_4,\;\; \lambda_1+\lambda_2+\lambda_3+\lambda_4=1,\;\; \lambda_h\geq 0,
\]
is solvable.
Moreover, as $\Phi_{1}(\Lambda)=\Phi_{2}(\Lambda)=\sum_{h:C_h\subseteq A}\lambda_h=\lambda_1+\lambda_2+\lambda_3+\lambda_4=1$, for every solution $\Lambda$ of $(\Systemsigma) $, it follows that (the new) $I_0$ (as defined in (\ref{EQ:I0})) is empty and,  by Theorem \ref{COER-P0}, $(y,z)$ is coherent.  Then, $\P=(x,y,z)$ is coherent. Therefore, as any precise probability assessment
$\P=(x,y,t)\in[0,1]^3$  on $\mathcal{F}$ is coherent, it follows that the imprecise assessment $\I=[0,1]^3$ on $\mathcal{F}$ is t-coherent.
\end{proof}

We now prove the t-coherence of the  imprecise assessment  $[0,1]^3$ on the sequence of conditional events $(C|B,B|A,A|(A\vee B))$. This sequence is involved in our probabilistic interpretation of the premise set of Figure I and includes the conditional event used in our existential import assumption.
\begin{proposition}\label{PROP:FIGUREIBIS}
Let $A,B,C$ be logically independent events.
The imprecise assessment $[0,1]^3$ on $\mathcal{F}=(C|B,B|A,A|(A\vee B))$ is t-coherent.	
\end{proposition}
\begin{proof}
Let $\P=(x,y,t)\in[0,1]^3$ be a probability assessment on $\mathcal{F}$. 
	The constituents $C_h$ and  the  points $Q_h$ associated with   $(\F,\P)$ are
given in Table \ref{TAB:TABLE_FIGUREIBIS}. 
\begin{table}
  \begin{minipage}{\textwidth}\centering 
	\caption{Constituents $C_h$ and  points $Q_h$ associated with  the probability  assessment    $\mathcal{P}=(x,y,t)$  on 
		$(C|B,B|A,A|(A\vee B))$ involved in the premise set of  Figure I.
	}
	\label{TAB:TABLE_FIGUREIBIS}
	\begin{tabular}{llll}\hline
		& $C_h$            & $Q_h$                          &  \\
		\hline
		$C_1$ \, & $ABC$  \,         & $(1,1,1)$ \,                   & $Q_1$   \\
		$C_2$ & $AB\no{C}$      & $(0,1,1)$                    & $Q_2$   \\
		$C_3$ & $A\no{B}C$       & $(x,0,1)$                    & $Q_3$   \\
		$C_4$ & $A\no{B}\no{C}$      & $(x,0,1)$                    & $Q_4$   \\
		$C_5$ & $\no{A}BC$ & $(1,y,0)$                    & $Q_5$   \\
		$C_6$ & $\no{A}B\no{C}$  & $(0,y,0)$                  & $Q_6$   \\
		$C_0$ & $\no{A} \, \no{B}$   &$(x,y,t)$                       & $Q_0=\mathcal{P} $ \\\hline
	\end{tabular}
\end{minipage}
\end{table}
By Theorem \ref{COER-P0}, coherence of  $\P=(x,y,z)$ on $\F$ requires that the following system is solvable
\[\begin{array}{l}
(\Systemsigma)  \hspace{1 cm}
\mathcal{P}=\sum_{h=1}^6 \lambda_hQ_h,\;
\sum_{h=1}^6 \lambda_h=1,\; \lambda_h\geq 0,\, h=1,\ldots,6,
\end{array}
\]
that is
\begin{equation}\label{EQ:SIGMA_FIG1BIS}
\begin{array}{ccc}
\left\{
\begin{array}{lllllll}
\lambda_1+x\lambda_3+x\lambda_4+\lambda_5=x, \\
\lambda_1+\lambda_2+y\lambda_5+y\lambda_6=y,\\
\lambda_1+\lambda_2+\lambda_3+\lambda_4=t,\\
\lambda_1+\lambda_2+\lambda_3+\lambda_4+\lambda_5+\lambda_6=1, \;\;\\
\lambda_h\geq 0,\;\; h=1,\ldots,6\,.
\end{array}
\right.
\Longleftrightarrow
\left\{
\begin{array}{lllllll}
\lambda_1+\lambda_5=x(\lambda_1+\lambda_2+\lambda_5+\lambda_6), \\
\lambda_1+\lambda_2=y(\lambda_1+\lambda_2+\lambda_3+\lambda_4),\\
\lambda_1+\lambda_2+\lambda_3+\lambda_4=t,\\
\lambda_1+\lambda_2+\lambda_3+\lambda_4+\lambda_5+\lambda_6=1, \;\;\\
\lambda_h\geq 0,\;\; h=1,\ldots,6\,,
\end{array}
\right.
\end{array}
\end{equation}
or equivalently 
\begin{equation}\label{EQ:SIGMA_FIG1TRIS}
\begin{array}{ccc}
\left\{
\begin{array}{lllllll}
\lambda_1+\lambda_5=xyt+x(1-t), \\
\lambda_1+\lambda_2=yt,\\
\lambda_3+\lambda_4=1-yt,\\
\lambda_5+\lambda_6=1-t, \;\;\\
\lambda_h\geq 0,\;\; h=1,\ldots,6\,.
\end{array}
\right.
&\Longleftrightarrow&
    \left\{
\begin{array}{lllllll}
\lambda_5=xyt+x(1-t)-\lambda_1, \\
\lambda_2=yt-\lambda_1,\\
\lambda_3=t(1-y)-\lambda_4,\\
\lambda_6=(1-t)(1-x)-xyt+\lambda_1\\
\lambda_h\geq 0,\;\; h=1,\ldots,6\,.
\end{array}
\right.
\end{array}
\end{equation}
As $(x,y,t)\in[0,1]^3$ it holds  that
$\max\{0,xyt-(1-t)(1-x)\} \leq\min\{xyt+x(1-t),yt\}$.
Then, the System $(\Systemsigma) $ is solvable and the set of all solutions $\mathcal{S}$ is the set of vectors $\Lambda=(\lambda_1,\ldots,\lambda_6)$ such that
\[ 
    \left\{
\begin{array}{ll}
\max\{0,xyt-(1-t)(1-x)\}\leq \lambda_1\leq\min\{xyt+x(1-t),yt\},
\\
\lambda_2=yt-\lambda_1,\\
\lambda_3=t(1-y)-\lambda_4,\\
0\leq \lambda_4\leq t(1-y),\\
\lambda_5=xyt+x(1-t)-\lambda_1, \\
\lambda_6=(1-t)(1-x)-xyt+\lambda_1.\\
\end{array}
\right.
\]
For each conditional event $A$, $B$, and $A\vee B$ in $\F$  we associate the function $\Phi_1(\Lambda)=\sum_{h:C_h\subseteq A}\lambda_h=\lambda_1+\lambda_2+\lambda_3+\lambda_4$, $\Phi_2(\Lambda)=\sum_{h:C_h\subseteq B}\lambda_h=\lambda_1+\lambda_2+\lambda_5+\lambda_6$, and $\Phi_3(\Lambda)=\lambda_1+\ldots+\lambda_6$, respectively, as defined in (\ref{EQ:PHI}).
We observe that   $\Phi_{3}(\Lambda)=1>0$ for each solution $\Lambda \in\mathcal{S}$ and hence $M_3=\max\{\Phi_{3}:\Lambda\in\mathcal{S}\}>0$.
Then, concerning the strict subset $I_0$ of $\{1,2,3\}$ (defined in (\ref{EQ:I0})),
we obtain $I_0\subseteq\{1,2\}$.
Notice that $I_0$ cannot be  equal to $\{1,2\}$, because $\Phi_{3}(\Lambda)>0$  implies that at least  $\Phi_{1}(\Lambda)$ or  $\Phi_{2}(\Lambda)$ is positive, for each $\Lambda \in\mathcal{S}$. Then,  $M_1=\max\{\Phi_{1}:\Lambda\in\mathcal{S}\}$ and $M_2=\max\{\Phi_{2}:\Lambda\in\mathcal{S}\}$ cannot be  equal to zero and hence $I_0\subset\{1,2\}$.
Therefore, we distinguish the following three cases: $(i)$ $I_0=\emptyset$; $(ii)$ $I_0=\{1\}$; $(iii)$ $I_0=\{2\}$.

Case $(i)$. As $(\Systemsigma) $ is solvable, we obtain that  the assessment $\P=(x,y,t)$ is coherent
by Theorem \ref{COER-P0}.

Case $(ii)$. The assessment $\P_0=x\in[0,1]$ on  $\F_0=\{C|B\}$ is coherent because $B$ and $C$ are logically independent. Then,  as $(\Systemsigma) $ is solvable and $\P_0$ on $\F_0$ is coherent, we obtain   by Theorem \ref{COER-P0} that 
 the assessment $\P=(x,y,t)$ is coherent.
 
Case $(iii)$ is analogous to  Case $(ii)$, where $C$ and $B$ are replaced by $B$ and $A$, respectively.

 Therefore, the assessment  $\P=(x,y,t)\in[0,1]^3$ is coherent for every $(x,y,t)\in[0,1]^3$, that is the imprecise assessment $[0,1]^3$ on $\F$ is t-coherent.
\end{proof}

We recall the following probability propagation rule for the inference form: from $(C|B,B|A,A|(A\vee B))$ to $C|A$ (Theorem 3 in \cite{gilio16}).
\begin{theorem}\label{THM:PROPWTFIGI}
	Let $A,B,C$ be three logically independent events and $(x,y,t)\in[0,1]^3$ be any (coherent)  assessment on the sequence $(C|B,B|A,A|A\vee B)$. Then, the  extension $z=p(C|A)$ from $(x,y,t)$  is coherent if and only if $z\in[z',z'']$, where 
	\[
	\begin{array}{ll}
	[z',z''] =
	\left\{
	\begin{array}{ll}
	[0,1], & t=0;\\
	\left[\max\left\{0,xy-\frac{(1-t)(1-x)}{t}\right\}, \min\left\{1,(1-x)(1-y)+\frac{x}{t}\right\}\right] \,, & t>0\,.\\
	\end{array}
	\right.
	\end{array}
	\]
\end{theorem}
Theorem \ref{THM:PROPWTFIGI} 
has been generalized to the case  of interval-valued probability assessments, which results into the following imprecise probability propagation rule (Theorem 4  in \cite{gilio16}):
\begin{theorem}\label{THM:PROPWTIVFIGI}
		Let $A,B,C$ be three logically independent events and $\mathcal{I}=([x_1,x_2]\times [y_1,y_2]\times [t_1,t_2])\subseteq [0,1]^3$  be a (t-coherent) interval-valued probability   assessment on  $(C|B,B|A,A|A\vee B)$. Then, the set $\Sigma$ of the coherent extension of $\mathcal{I}$ is the interval  $[z^{*},z^{**}]$, where $[z^{*},z^{**}] =$
		\[
		\begin{array}{ll}
		\left\{
		\begin{array}{ll}
		[0,1], & t=0;\\
		\left[\max\left\{0,x_1y_1-\frac{(1-t_1)(1-x_1)}{t_1}\right\}, \min\left\{1,(1-x_2)(1-y_1)+\frac{x_2}{t_1}\right\}\right]
		\,, & t>0\,.\\
		\end{array}
		\right.
		\end{array}
		\]
	\end{theorem}

	\subsection{Traditionally valid syllogisms of Figure I.}
	\label{SEC:FIGISYLLO}
In this section we consider the probabilistic interpretation of the valid  syllogisms of Figure I (see Table~\ref{TAB:AristSyl}). 
Specifically, we firstly adapt the results on  Barbara, Barbari, and  Darii given in \cite{gilio16} applying the notions of (s-)validity. 
 Secondly, we prove s-validity of Celarent and Ferio and validity of Celaront. 
We use  the probabilistic interpretation of the basic syllogistic sentence types given in Table \ref{Table:BS}. By instantiating  in Proposition \ref{PROP:FIGUREI} 
the subject $S$, the middle $M$, and the predicate $P$ term for the events  $A,B,C$, respectively,  we observe that  the imprecise assessment $[0,1]^3$ on $(P|M,M|S,P|S)$ is t-coherent. This implies that all syllogisms of Figure I are probabilistically non-informative. For instance, modus Barbara (``\emph{Every $M$ is $P$, Every $S$ is $M$, therefore Every $S$ is $P$}'') without  existential import assumption corresponds to the probabilistically non-informative 
inference:  from the premises $p(P|M)=1$ and $p(M|S)=1$ infer  that every $p(P|S)\in[0,1]$ is coherent.
Indeed, by Proposition \ref{PROP:FIGUREI}, 
a  probability assessment $(1,1,z)$ on $(P|M,M|S,P|S)$ is coherent for every $z\in[0,1]$.  In order to construct  probabilistically informative versions of valid syllogisms  of Figure I, we add the conditional event  existential import assumption  to the  probabilistic interpretation of the respective premise set:  $p(S|(S \vee M))>0$ (see Definition \ref{DEF:CEI}). 
We now demonstrate the validity (and when possible the s-validity) of  traditionally valid syllogisms  by  suitable instantiations in Theorem~\ref{THM:PROPWTFIGI} within our semantics. 
Of course, some syllogisms will turn be equivalent when some terms are negated (and the corresponding probabilities are adjusted accordingly).  However, we provide for each syllogism within each figure a direct way of showing its validity by simply instantiating the respective probability propagation rule.  
\paragraph{Barbara} 
By instantiating $S,M,P$ in Theorem~\ref{THM:PROPWTFIGI} for $A,B,C$ with $x=y=1$ and any value $t>0$
it follows that  $z'=\max\left\{0,xy-\frac{(1-t)(1-x)}{t}\right\}=1$ and $z''=\min\{1,(1-x)(1-y)+\frac{x}{t}\}=1$.  Then,  the set $\Sigma$ (see Equation (\ref{EQ:UEXT}))   of coherent extensions on $P|S$ of the  imprecise assessment $\{1\}\times\{1\}\times(0,1]$ on  $(P|M,M|S,S|(S\vee M))$ is $\Sigma=\{1\}$.
 Thus, by Definition \ref{DEF:VALID},
 \begin{equation}\label{EQ:BARBARAIP}
\{1\}\times\{1\}\times(0,1] \mbox{ on } (P|M,M|S,S|(S\vee M)) \;\models_s\; \{1\} \mbox{ on } P|S. \end{equation}
In terms of probabilistic constraints, (\ref{EQ:BARBARAIP}) can be expressed by
\begin{equation}\label{EQ:BARBARA} 
(p(P|M)=1, p(M|S)=1, p(S|(S \vee M))>0) \,\models_s \, p(P|S)= 1\,,
\end{equation}
which is a  s-valid (and probabilistically informative) version of Barbara (under the conditional event existential import assumption).
\begin{remark}\label{REM:EIFIG1}
	By instantiating Remark \ref{REM:EI>CEI} to syllogisms of Figure I we obtain 
	that  $p(S)=p(S\wedge (S \vee M))=p(S|(S \vee M))p(S\vee M)$. Hence,  if $p(S)>0$  then $p(S|(S \vee M))>0$.
	Then,  as $p(S)>0$  implies $p(S|(S \vee M))>0$, 
	from (\ref{EQ:BARBARA}) it follows that 
	\begin{equation}
	(p(P|M)=1, p(M|S)=1,p(S)>0)\; \models_s \; p(P|S)= 1,
	\end{equation}
	which is  an s-valid version of 	Barbara   under the (stronger) unconditional existential import assumption.	
\end{remark}

\paragraph{Barbari}
From  (\ref{EQ:BARBARA}), by weakening the conclusion (see Remark \ref{REM:VALID}), it follows that 
\begin{equation}\label{EQ:BARBARI}
(p(P|M)=1, p(M|S)=1, p(S|(S \vee M))>0) \;\models\; p(P|S)>0\,,
\end{equation}
which is a valid (but not s-valid) version of  Barbari (``\emph{Every $M$ is $P$, Every $S$ is $M$, therefore Some $S$ is $P$}''). 	
\paragraph{Darii}
By instantiating $S,M,P$ in Theorem~\ref{THM:PROPWTFIGI} for $A,B,C$ with $x=1$, any $y>0$, and any $t>0$,
it follows that  $z'=\max\left\{0,xy-\frac{(1-t)(1-x)}{t}\right\}=y>0$ and $z''=\min\{1,(1-x)(1-y)+\frac{x}{t}\}=1$.
Then,  the set $\Sigma$    of coherent extensions on $P|S$ of the  imprecise assessment $\{1\}\times(0,1]\times(0,1]$ on  $(P|M,M|S,S|(S\vee M))$ is $\Sigma=
\bigcup_{\{(x,y,t)\in\{1\}\times (0,1]\times (0,1]\}} [y,1]=
(0,1]$.
Thus, by Definition \ref{DEF:VALID},
\begin{equation}\label{EQ:DARIIIP}
\{1\}\times(0,1]\times(0,1] \mbox{ on } (P|M,M|S,S|(S\vee M)) \;\models_s\; (0,1] \mbox{ on } P|S. \end{equation}
In terms of probabilistic constraints, (\ref{EQ:DARIIIP}) can be expressed by
\begin{equation}\label{EQ:DARII} 
(p(P|M)=1, p(M|S)>0, p(S|(S \vee M))>0) \,\models_s \, p(P|S)>0 \,,
\end{equation}
which is a  s-valid version of Darii (\emph{Every $M$ is $P$, Some $S$ is $M$, therefore Some $S$ is $P$}).
 Notice that Barbari (\ref{EQ:BARBARI})  also follows from Darii (\ref{EQ:DARII}) by strengthening the  minor premise (see Remark \ref{REM:VALID}).

\paragraph{Celarent}
By instantiating $S,M,P$ in Theorem~\ref{THM:PROPWTFIGI} for $A,B,C$ with $x=0$, $y=1$, and  $t>0$,
it follows that  $z'=\max\left\{0,xy-\frac{(1-t)(1-x)}{t}\right\}=0$ and $z''=\min\{1,(1-x)(1-y)+\frac{x}{t}\}=0$.
Then,  the set $\Sigma$    of coherent extensions on $P|S$ of the  imprecise assessment $\{0\}\times(0,1]\times(0,1]$ on  $(P|M,M|S,S|(S\vee M))$ is $\Sigma=\{0\}$.
Thus, by Definition \ref{DEF:VALID},
\begin{equation}\label{EQ:CELARENTIP}
\{0\}\times\{1\}\times(0,1] \mbox{ on } (P|M,M|S,S|(S\vee M)) \;\models_s\; \{0\} \mbox{ on } P|S. \end{equation}
In terms of probabilistic constraints, (\ref{EQ:CELARENTIP}) can be expressed by
\begin{equation}\label{EQ:CELARENT} 
(p(P|M)=0, p(M|S)=1, p(S|(S \vee M))>0) \,\models_s \, p(P|S)=0 \,,
\end{equation}
which is a  s-valid version of Celarent (\emph{No $M$ is $P$, Every $S$ is $M$, therefore No $S$ is  $P$}).
Notice that Celarent is equivalent to Barbara, because 
(\ref{EQ:CELARENT}) 
is equivalent to 
(\ref{EQ:BARBARA})
when  $P$ is replaced by  $\no{P}$ and the  probabilities are adjusted accordingly.

\paragraph{Celaront}
From  (\ref{EQ:CELARENT}), by weakening the conclusion, it follows that 
\begin{equation}\label{EQ:CELARONT}
(p(P|M)=0, p(M|S)=1, p(S|(S \vee M))>0)\;\models\;  p(\no{P}|S)>0\,,
\end{equation}
which is valid version of  Celaront (\emph{No $M$ is $P$, Every $S$ is $M$, therefore Some $S$ is not $P$}). 
Notice that Celaront (\ref{EQ:CELARONT}) is equivalent to Barbari (\ref{EQ:BARBARI}), where
 $P$ is replaced by  $\no{P}$.

\paragraph{Ferio}
By instantiating $S,M,P$ in Theorem~\ref{THM:PROPWTFIGI} for $A,B,C$ with $x=0$, any $y>0$, and any $t>0$,
it follows that  $z'=\max\left\{0,xy-\frac{(1-t)(1-x)}{t}\right\}=0$ and $z''=\min\{1,(1-x)(1-y)+\frac{x}{t}\}=1-y<1$.
Then,  the set $\Sigma$    of coherent extensions on $P|S$ of the  imprecise assessment $\{0\}\times(0,1]\times(0,1]$ on  $(P|M,M|S,S|(S\vee M))$ is $\Sigma=\bigcup_{\{(x,y,t)\in \{0\}\times (0,1]\times (0,1]\}} [0,1-y]=[0,1)$.
Thus, by Definition \ref{DEF:VALID},
\begin{equation}\label{EQ:FERIOIP}
\{0\}\times(0,1]\times(0,1] \mbox{ on } (P|M,M|S,S|(S\vee M)) \;\models_s\; [0,1) \mbox{ on } P|S. \end{equation}
In terms of probabilistic constraints, (\ref{EQ:FERIOIP}) can be equivalently expressed by (see Table~\ref{Table:BS})
\begin{equation}\label{EQ:FERIO} 
(p(P|M)=0, p(M|S)>0, p(S|(S \vee M))>0) \,\models_s \, p(\no{P}|S)>0 \,,
\end{equation}
which is a  s-valid version of Ferio (\emph{No $M$ is $P$, Some $S$ is $M$, therefore Some $S$ is not $P$}).  Notice that Ferio (\ref{EQ:FERIO}) is equivalent to Darii (\ref{EQ:DARII}), where
 $P$ is replaced by  $\no{P}$.   Celaront (\ref{EQ:CELARONT})  also follows from Ferio (\ref{EQ:FERIO}) by strengthening the  minor premise (i.e., $p(M|S)>0$ is replaced by the stronger constraint $p(M|S)=1$).  

	\section{Figure II.}
\label{SEC:FIGUREII}
In this section,
we prove that the probabilistic inference of $\no{C}|A$ from the premise set $(B|C,\widebar{B}|A)$, which corresponds to  the general form of syllogisms of Figure II, is probabilistically non-informative. Like in Section \ref{SEC:WT}, we show that the imprecise assessment $[0,1]^3$ on $(B|C,\widebar{B}|A,\no{C}|A)$ is t-coherent. 
Then, in order to obtain probabilistic informativeness, we add the probabilistic constraint $p(A|(A \vee  C))>0$ to the premise set, which corresponds to the conditional event existential import assumption of syllogisms of Figure II according to Definition \ref{DEF:CEI}. 
After showing that the imprecise assessment  $[0,1]^3$ on  $(B|C,\widebar{B}|A,A|(A \vee  C))$ is t-coherent, we prove  the precise and imprecise probability propagation rules for the inference from $(B|C,\widebar{B}|A,A|(A \vee  C))$ to  $\no{C}|A$.
We apply these results  in Section~\ref{SEC:FIGIISYLLO}, where we  study the valid syllogisms of Figure II.

\subsection{Coherence and  probability propagation   in Figure II.}
We prove the t-coherence of the  imprecise assessment  $[0,1]^3$ on the sequence of conditional events 
$(B|C,\widebar{B}|A,\widebar{C}|A)$.
\begin{proposition}\label{PROP:FIGUREII}
	Let $A,B,C$ be logically independent events.
	The imprecise assessment $[0,1]^3$ on $\mathcal{F}=(B|C,\widebar{B}|A,\widebar{C}|A)$ is t-coherent.
\end{proposition}	
\begin{proof} 
	Let $\P=(x,y,z)\in[0,1]^3$ be any probability assessment on $\F$.
	The constituents $C_h$ and  the  points $Q_h$ associated with   $(\F,\P)$  are
	given in Table \ref{TAB:TABLE_FIGUREII}. 
	\begin{table}[!h]
  \begin{minipage}{\textwidth}\centering 
			\caption{Constituents $C_h$ and  points $Q_h$ associated with  the probability  assessment    $\mathcal{P}=(x,y,z)$  on 
			$(B|C,\widebar{B}|A,\widebar{C}|A)$ involved in Figure II.
		}
		\label{TAB:TABLE_FIGUREII}
		\begin{tabular}{llll}\hline
			& $C_h$            & $Q_h$                          &  \\
			\hline
			$C_1$ \, & $ABC$         \,  & $(1,0,0)$ \,                   & $Q_1$   \\
			$C_2$ & $AB\no{C}$      & $(x,0,1)$                    & $Q_2$   \\
			$C_3$ & $A\no{B}C$       & $(0,1,0)$                    & $Q_3$   \\
			$C_4$ & $A\no{B}\no{C}$      & $(x,1,1)$                    & $Q_4$   \\
			$C_5$ & $\no{A}BC$ & $(1,y,z)$                    & $Q_5$   \\
			$C_6$ & $\no{A}\no{B}C$  & $(0,y,z)$                  & $Q_6$   \\
			$C_0$ & $\no{A} \, \no{C}$   &$(x,y,z)$                       & $Q_0=\mathcal{P} $  \\\hline
		\end{tabular}
		\end{minipage}
	\end{table}
	The constituents $C_h$ contained in  $\mathcal{H}_3=A\vee C$ are $C_1,\ldots,C_6$.
	We recall that coherence of  $\P=(x,y,z)$ on $\F$ requires that the condition $\P \in \mathfrak{I}$ is satisfied, where   $\mathfrak{I}$ is the convex hull of $Q_1, \ldots, Q_6$.
	This  amounts to the solvability of the following system:
	\[\begin{array}{l}
	(\Systemsigma)  \hspace{1 cm}
	\mathcal{P}=\sum_{h=1}^6 \lambda_hQ_h,\;
	\sum_{h=1}^6 \lambda_h=1,\; \lambda_h\geq 0,\, h=1,\ldots,6.
	\end{array}
	\]
We observe that $\P=(x,y,z)=x(1,y,z)+(1-x)(0,y,z)=xQ_5+(1-x)Q_6$, that is  system 	$(\Systemsigma)  $ is solvable and a solution is
 $\Lambda=(0,0,0,0,x,1-x)$. As $\Phi_{2}(\Lambda)=\Phi_{3}(\Lambda)=\lambda_1+\lambda_2+\lambda_3+\lambda_4=0$, it holds that $I_0'=\{2,3\}$.
 Then, by Theorem \ref{COER-P0'}, in order to check coherence of $(x,y,z)\in[0,1]^3$ it is sufficient to check the coherence of the sub-assessment $(y,z)\in[0,1]^2$ on $(\no{B}|A,\no{C}|A)$. 
  	The constituents $C_h$ associated to the pair $((\no{B}|A,\no{C}|A), (y,z) )$ contained in  $\mathcal{H}_2=A$ are $C_1=ABC, C_2=AB\no{C}, C_3=A\no{B}C, C_4=A\no{B}\no{C}$ and the corresponding points $Q_h$ are $Q_1=(0,0), Q_2=(0,1), Q_3=(1,0), Q_4=(1,1)$.
The convex hull $\mathfrak{I}$ of the points $Q_1,Q_2,Q_3,Q_4$ is the unit square  $[0,1]^2$. Then $(y,z)\in[0,1]^2$ trivially belongs to $\mathfrak{I}$ and hence the system $(y,z)=\lambda_1Q_1+\lambda_2Q_2+\lambda_3Q_3+\lambda_4Q_4$  has always a nonnegative solution $(\lambda_1,\lambda_2,\lambda_3,\lambda_4)$ with $\lambda_1+\lambda_2+\lambda_3+\lambda_4=1$.
 Moreover, as $\Phi_{1}(\Lambda)=\Phi_{2}(\Lambda)=\lambda_1+\lambda_2+\lambda_3+\lambda_4=1$, it follows that $I_0=\emptyset$ and hence $(y,z)$ is coherent.
\end{proof}
\begin{proposition}\label{PROP:FIGUREIIBIS}
	Let $A,B,C$ be logically independent events. The assessment $(x,y,t)$  on $(B|C,\widebar{B}|A,A|(A\vee C))$ is coherent for every $(x,y,t)\in [0,1]^3$.
\end{proposition}	
\begin{proof} 
	Let $\P=(x,y,t)\in[0,1]^3$ be a probability assessment on $\mathcal{F}$.
	The constituents $C_h$ and  the  points $Q_h$ associated with   $(\F,\P)$ are
	given in Table \ref{TAB:TABLE_FIGUREIIPREMISE}. 
	\begin{table}[!h]
		\begin{minipage}{\textwidth}
			\caption{Constituents $C_h$ and  points $Q_h$ associated with  the probability  assessment    $\mathcal{P}=(x,y,t)$  on 
				$\F=(B|C,\no{B}|A,A|(A\vee C))$ involved in the premise set of Figure~II.
			}
			\label{TAB:TABLE_FIGUREIIPREMISE}
		\centering 	\begin{tabular}{llll}\hline
		& $C_h$            & $Q_h$                          &  \\
		\hline
		$C_1$ \, & $ABC$         \,  & $(1,0,1)$ \,                   & $Q_1$   \\
		$C_2$ & $AB\no{C}$      & $(x,0,1)$                    & $Q_2$   \\
		$C_3$ & $A\no{B}C$       & $(0,1,1)$                    & $Q_3$   \\
		$C_4$ & $A\no{B}\no{C}$      & $(x,1,1)$                    & $Q_4$   \\
		$C_5$ & $\no{A}BC$ & $(1,y,0)$                    & $Q_5$   \\
		$C_6$ & $\no{A}\no{B}C$  & $(0,y,0)$                  & $Q_6$   \\
		$C_0$ & $\no{A} \, \no{C}$   &$(x,y,t)$                       & $Q_0=\mathcal{P} $  \\\hline
	\end{tabular}
		\end{minipage}
	\end{table}
	By Theorem \ref{COER-P0}, coherence of  $\P=(x,y,t)$ on $\F$ requires that the following system is solvable
	\[\begin{array}{l}
	(\Systemsigma)  \hspace{1 cm}
	\mathcal{P}=\sum_{h=1}^6 \lambda_hQ_h,\;
	\sum_{h=1}^6 \lambda_h=1,\; \lambda_h\geq 0,\, h=1,\ldots,6,
	\end{array}
	\]
or equivalently
\begin{equation}\label{EQ:FIGURIISYSTEMBIS}
\begin{array}{lll}
\left\{
\begin{array}{lllllll}
\lambda_1+\lambda_5=x(\lambda_1+\lambda_3+\lambda_5+\lambda_6), \\
\lambda_3+\lambda_4=yt, \\
\lambda_1+\lambda_2+\lambda_3+\lambda_4=t, \\
\lambda_1+\lambda_2+\lambda_3+\lambda_4+\lambda_5+\lambda_6=1, \;\\
\lambda_i\geq 0,\; i=1,\ldots,6\,,
\end{array}
\right.
&  \Longleftrightarrow  &
\left\{
\begin{array}{lllllll}
\lambda_1+\lambda_5=x(\lambda_1+\lambda_3+\lambda_5+\lambda_6), \\
\lambda_3+\lambda_4=yt, \\
\lambda_1+\lambda_2=t(1-y), \\
\lambda_5+\lambda_6=1-t, \;\\
\lambda_i\geq 0,\; i=1,\ldots,6.
\end{array}
\right.
\end{array}
\end{equation}
System $(\Systemsigma) $ is solvable and a subset $\mathcal{S}'$ of the set of  solutions consists of  $\Lambda=(\lambda_1,\ldots,\lambda_6)$ such that
\begin{equation}
\left\{
\begin{array}{lllllll}
\lambda_1=\lambda_3=0, \;
\lambda_2=t(1-y), \\
\lambda_4=yt, \;
\lambda_5=x(1-t), \\
\lambda_6=(1-x)(1-t), \\
\lambda_i\geq 0,\; i=1,\ldots,6\,.
\end{array}
\right.
\end{equation}
Moreover, for each $\Lambda\in \mathcal{S'}$ it holds that $\Phi_{1}(\Lambda)=\sum_{h:C_h\subseteq C}=\lambda_1+\lambda_3+\lambda_5+\lambda_6=1-t$, $\Phi_{2}(\Lambda)=\sum_{h:C_h\subseteq A}=\lambda_1+\lambda_2+\lambda_3+\lambda_4=t$  and $\Phi_{3}(\Lambda)=\sum_{h:C_h\subseteq A\vee C}\lambda_h=1>0$.
If $0<t<1$, it holds that $I_0'=\emptyset$, hence, by Theorem \ref{COER-P0'}, $(x,y,t)$ is coherent.
If $t=0$, it holds that $I_0'=\{2\}$ and as the sub-assessment $y\in[0,1]$ on $\no{B}|A$ is  coherent, it follows by  Theorem \ref{COER-P0'} that $(x,y,t)$ is coherent.
Likewise, if $t=1$, it holds that $I_0'=\{1\}$ and as the sub-assessment $x\in[0,1]$ on $B|C$ is  coherent, it follows by Theorem \ref{COER-P0'} that  $(x,y,t)$ is coherent. Then, $(x,y,t)$ is coherent for every $(x,y,t)\in[0,1]^3$.
\end{proof}

The next  result  allows for computing  the lower  and upper bounds, $z'$ and $z''$ respectively,  for the coherent extensions $z=p(\widebar{C}|A)$ from the assessment $(x,y,t)$ on $(B|C,\widebar{B}|A, A|(A\vee C))$.
\begin{theorem}\label{THM:PROPF2}
	Let $A,B,C$ be three logically independent events and $(x,y,t)\in[0,1]^3$ be any   assessment on the family $(B|C,\widebar{B}|A,A|(A\vee C))$. Then, the  extension $z=p(\widebar{C}|A)$ is coherent if and only if $z\in[z',z'']$, where 
	\[
	\begin{array}{ll}
	[z',z''] =
	\left\{
	\begin{array}{cl}
	\left[0,1\right], &  \mbox{if }t\leq x+yt\leq 1, \\~
    \displaystyle \big[\frac{x + yt - 1}{t\, x},\; 1\big] \,, &\mbox{if } x+yt> 1,\\~
    \displaystyle \big	[\frac{t-x-yt }{t\, (1-x)},\; 1\big] \,, & \mbox{if } x+yt< t. 
	\end{array}
	\right.
	\end{array}
	\]	
\end{theorem}
\begin{proof}
Let  $(x,y,t)\in [0,1]^3$ be  a generic assessment on $(B|C,\widebar{B}|A,A|(A\vee C))$. 
We recall that  $(x,y,t)$ is coherent (Proposition \ref{PROP:FIGUREIIBIS}).
In order to prove the theorem  we
derive the coherent lower and upper probability bounds $z'$ and $z''$ by applying Algorithm~\ref{Alg} in a symbolic way.
\paragraph{Computation of the lower probability bound  $z'$ on $\widebar{C}|A$}~\\
\emph{Input.\,} $\F=(B|C,\widebar{B}|A,A|(A\vee C))$, $E_{n+1}|H_{n+1}=\widebar{C}|A$.\\
\emph{Step 0.\,} The constituents  associated with  $(B|C,\widebar{B}|A,A|(A\vee C),\widebar{C}|A)$
and contained in $\mathcal{H}_{n+1}= A \vee C$ are
$C_1=ABC\,, C_2=AB \widebar{C}\,, C_3=A \widebar{B}C\,, C_4=A \widebar{B}  \widebar{C}\,, C_5=\widebar{A}BC\,,$ and  $C_6=\widebar{A} \widebar{B} C$.
We construct the following starting
 system with unknowns $\lambda_1,\ldots,\lambda_6,z$ (see Remark \ref{REM1}):
\begin{equation}\label{S_0'}
\left\{
\begin{array}{lllllll}
\lambda_2+\lambda_4=z(\lambda_1+\lambda_2+\lambda_3+\lambda_4), \;
\lambda_1+\lambda_5=x(\lambda_1+\lambda_3+\lambda_5+\lambda_6), \\
\lambda_3+\lambda_4=y(\lambda_1+\lambda_2+\lambda_3+\lambda_4), \\
\lambda_1+\lambda_2+\lambda_3+\lambda_4=t(\lambda_1+\lambda_2+\lambda_3+\lambda_4+\lambda_5+\lambda_6), \\
\lambda_1+\lambda_2+\lambda_3+\lambda_4+\lambda_5+\lambda_6=1, \;
\lambda_i\geq 0,\; i=1,\ldots,6\,.
\end{array}
\right.
\end{equation}
\emph{Step 1.\,} By setting  $z=0$ in  System (\ref{S_0'}), we   obtain
\begin{equation}\label{S_0'z=0}
\begin{array}{lcl}
\left\{
\begin{array}{lllllll}
\lambda_2+\lambda_4=0, \;
\lambda_1+\lambda_5=x, \\
\lambda_3=y(\lambda_1+\lambda_3), \;
\lambda_1+\lambda_3=t, \\
\lambda_1+\lambda_3+\lambda_5+\lambda_6=1, \\
\lambda_i\geq 0,\; i=1,\ldots,6\,.
\end{array}
\right.
&
\quad \Longleftrightarrow \quad &
\left\{
\begin{array}{lllllll}
\lambda_1=t(1-y), \;
\lambda_2=0, \;
\lambda_3=yt, \\
\lambda_4=0,\;
\lambda_5=x-t(1-y), \\
\lambda_6=1-x-yt, \\
\lambda_i\geq 0,\; i=1,\ldots,6\,.
\end{array}
\right.
\end{array}
\end{equation}
The solvability of System (\ref{S_0'z=0}) is a necessary condition for the coherence of the assessment $(x,y,t,0)$ on $(B|C,\widebar{B}|A,A|(A\vee C),\widebar{C}|A)$. As $(x,y,t) \in [0,1]^3$, it holds  that: $\lambda_1=t(1-y)\geq 0$, $\lambda_3=yt\geq 0$.
 Thus, System (\ref{S_0'z=0})   is solvable  if and only if  $\lambda_5\geq 0$ and $\lambda_6\geq 0$, that is  
\[
t-yt\leq x\leq 1-yt \Longleftrightarrow t\leq x+yt\leq 1.
\] 
We distinguish two cases:  $(i)$ $x+yt >1 \vee x+yt<t$; $(ii)$ $ t\leq x+yt\leq 1$.  In Case~$(i)$,  System~(\ref{S_0'z=0})   is  not solvable (which implies that the coherent extension $z$ of $(x,y,t)$ must be positive). Then,
we go to Step~2 of the algorithm where the (positive) lower bound $z'$ is obtained by optimization. In Case~$(ii)$,  System~(\ref{S_0'z=0})  is  solvable and in order to check whether $z=0$ is a coherent extension,  we go to Step~3. 

\emph{Case $(i)$.} We observe that in this case $t$ cannot be $0$. By Step 2 we have the following linear programming problem:\\
\emph{Compute $z'=\min(  \lambda_2+\lambda_4)$ subject to:}
\begin{equation}\label{Sigma'}
\left\{
\begin{array}{lllllll}
\lambda_1+\lambda_5=x(\lambda_1+\lambda_3+\lambda_5+\lambda_6), \;
\lambda_3+\lambda_4=y(\lambda_1+\lambda_2+\lambda_3+\lambda_4), \\
\lambda_1+\lambda_2+\lambda_3+\lambda_4=t(\lambda_1+\lambda_2+\lambda_3+\lambda_4+\lambda_5+\lambda_6), \\
\lambda_1+\lambda_2+\lambda_3+\lambda_4=1, \;
\lambda_i\geq 0,\; i=1,\ldots,6.
\end{array}
\right.
\end{equation}
In this case, 
the constraints in (\ref{Sigma'}) can be rewritten in the following way
\[
\left\{
\begin{array}{lllllll}
\lambda_1+\lambda_5=x(\lambda_1+\lambda_3+\lambda_5+\lambda_6), \\
\lambda_3+\lambda_4=y,\;
\lambda_5+\lambda_6= \frac{1-t}{t},\\
\lambda_1+\lambda_2+\lambda_3+\lambda_4=1, \\
\lambda_i\geq 0,\; i=1,\ldots,6\,,
\end{array}
\right.
\hspace{-0.3cm} \Leftrightarrow 
\left\{
\begin{array}{lllllll}
1-y-\lambda_2+\lambda_5=x(1-\lambda_2-\lambda_4+\frac{1-t}{t}),  \\
\lambda_3=y-\lambda_4,\;
\lambda_6= \frac{1-t}{t}-\lambda_5,\\
\lambda_1=1-y-\lambda_2, \\
\lambda_i\geq 0,\; i=1,\ldots,6\,,
\end{array}
\right.
\]
or equivalently
\[
\left\{
\begin{array}{lllllll}
x\lambda_4+(1-y)+\lambda_5=\lambda_2(1-x)+\frac{x}{t},  \;
\lambda_3=y-\lambda_4,\\
\lambda_5= \frac{1-t}{t}-\lambda_6,\;
\lambda_1=1-y-\lambda_2, \;
\lambda_i\geq 0,\; i=1,\ldots,6.
\end{array}
\right.
\]
We distinguish two  (alternative) cases: $(i.1)$ $x+yt >1$; $(i.2)$ $ x+yt<t$.\\ 
Case $(i.1)$. The constraints in (\ref{Sigma'}) can be rewritten in the following way
\[
\left\{
\begin{array}{lllllll}
x(\lambda_2+\lambda_4)=\frac{x}{t}-(1-y)-\frac{1-t}{t}+\lambda_2+\lambda_6, \;
\lambda_3=y-\lambda_4,\\
\lambda_5= \frac{1-t}{t}-\lambda_6,\;
\lambda_1=1-y-\lambda_2, \;
\lambda_i\geq 0,\; i=1,\ldots,6\,.
\end{array}
\right.
\]
 As $x>1-ty$, we observe that  $x>0$.  Then, the minimum of $z=\lambda_2+\lambda_4$, obtained when $\lambda_2=\lambda_6=0$, is 
 \begin{equation}
 z'=\frac{1}{x}\left(\frac{x}{t}-(1-y)-\frac{1-t}{t}\right)=\frac{x-t+yt-1+t}{xt}=\frac{x+yt-1}{xt}.
 \end{equation}
By choosing $\lambda_2=\lambda_6=0$ the constraints in (\ref{Sigma'}) are satisfied with
\[
\left\{
\begin{array}{lllllll}
\lambda_1=1-y, \;
\lambda_2=0,\;
\lambda_3=y-\frac{x+yt-1}{xt},\;
\lambda_4=\frac{x+yt-1}{xt},  \\
\lambda_5= \frac{1-t}{t},\;
\lambda_6=0,\;
\lambda_i\geq 0,\; i=1,\ldots,6.
\end{array}
\right.
\]
In particular $\lambda_3\geq0$ is satisfied because the condition  $\frac{x+yt-1}{xt}\leq y$, which in this case amounts to  
$yt(1-x)\leq1-x$, is always satisfied.
 Then, the \emph{procedure stops} yielding as \emph{output}
$z'=\frac{x+yt-1}{xt}$.\\
Case $(i.2)$. The constraints in (\ref{Sigma'}) can be rewritten in the following way
\[
\left\{
\begin{array}{lllllll}
(1-y)-\frac{x}{t}+\lambda_5+\lambda_4=\lambda_2(1-x)-x\lambda_4+\lambda_4,  \;
\lambda_3=y-\lambda_4,\\
\lambda_6= \frac{1-t}{t}-\lambda_5,\;
\lambda_1=1-y-\lambda_2, \;
\lambda_i\geq 0,\; i=1,\ldots,6\,,
\end{array}
\right.
\]
or equivalently
\[
\left\{
\begin{array}{lllllll}
(\lambda_2+\lambda_4)(1-x)=(1-y)-\frac{x}{t}+\lambda_4+\lambda_5, \;
\lambda_3=y-\lambda_4,\\
\lambda_6= \frac{1-t}{t}-\lambda_5,\;
\lambda_1=1-y-\lambda_2, \;
\lambda_i\geq 0,\; i=1,\ldots,6\,.
\end{array}
\right.
\]
As $t-yt-x>0$, that is $x<t(1-y)$, it holds that  $x<1$.  Then, the minimum of $z=\lambda_2+\lambda_4$, obtained when $\lambda_4=\lambda_5=0$, is 
\[
z'=\frac{1}{1-x}\left(1-y-\frac{x}{t}\right)=\frac{t-yt-x}{(1-x)t}\geq 0.
\] 
We observe that  by choosing $\lambda_4=\lambda_5=0$ the constraints in (\ref{Sigma'}) are satisfied, indeed they are
\[
\left\{
\begin{array}{lllllll}
\lambda_1=1-y, \;
\lambda_2=\frac{t-yt-x}{(1-x)t},\;
\lambda_3=y,\;
\lambda_4=0,  \\
\lambda_5= 0,\;
\lambda_6=\frac{1-t}{t},\;
\lambda_i\geq 0,\; i=1,\ldots,6.
\end{array}
\right.
\]
 Then, the \emph{procedure stops} yielding as \emph{output}
$z'=\frac{t-yt-x}{(1-x)t}$.\\

\emph{Case $(ii)$.} We take Step~3 of the algorithm. We denote by $\Lambda$  and $\mathcal{S}$ the vector of unknowns $(\lambda_1,\ldots,\lambda_6)$ and the set of solutions of System~(\ref{S_0'z=0}), respectively.
We consider the following linear functions (associated with the conditioning events $H_1=C, H_2=H_4=A, H_3=A \vee C$) and their maxima in $\mathcal{S}$: 
\begin{equation}\label{EQ:PHIFIGII}
\begin{array}{l}
\Phi_{1}(\Lambda)=\sum_{r:C_r\subseteq C}\lambda_r= \lambda_1+\lambda_3+\lambda_5+\lambda_6, \\ \Phi_{2}(\Lambda)=\Phi_{4}(\Lambda)=\sum_{r:C_r\subseteq A}\lambda_r=\lambda_1+\lambda_2+\lambda_3+\lambda_4, \\ 
\Phi_{3}(\Lambda)=\sum_{r:C_r\subseteq A\vee C}\lambda_r=\lambda_1+\lambda_2+\lambda_3+\lambda_4+\lambda_5+\lambda_6\,, \\
M_i=\max_{\Lambda \in \mathcal{S}} \Phi_i(\Lambda),\;\;  i=1,2,3,4\;.
\end{array}
\end{equation}

By (\ref{S_0'z=0}) we obtain: $\Phi_{1}(\Lambda)=t(1-y)+yt+x-t(1-y)+1-x-yt=1$,  $\Phi_{2}(\Lambda)=\Phi_{4}(\Lambda)=t(1-y)+0+yt+0=t$,  $\Phi_3(\Lambda)=t(1-y)+0+yt+0+x-t(1-y)+1-x-yt=1$, $\forall \Lambda \in \mathcal{S}$.
 Then,  $M_1=1$,  $M_2=M_4=t$, and $M_3=1$.
We consider two subcases: $t>0$; $t=0$.
If $t>0$, then $M_{4}>0$ and we are in the first case  of Step~3. Thus,  the \emph{procedure stops} and yields $z'=0$ as \emph{output}.  \\
If $t=0$, then
$M_{1}>0,M_{3}>0$ and $M_{2}=M_{4}=0$. Hence, we are in third case of Step~3 with $J = \{2\}, I_0=\{2,4\}$ and the procedure restarts with Step 0, with $\F$ replaced by  $\F_J=(\widebar{B}|A)$.\\

\emph{(2\textsuperscript{nd} cycle) Step $0$.\,} The constituents  associated with  $(\widebar{B}|A,\widebar{C}|A)$, contained in $A$, are
$ C_1=ABC, C_2=AB\widebar{C}, C_3=A\widebar{B}C, C_4=A\widebar{B}\widebar{C}$.
The starting system is
\begin{equation}\label{S_1'}
\left\{
\begin{array}{lllllll}
\lambda_3+\lambda_4=y(\lambda_1+\lambda_2+\lambda_3+\lambda_4), \; 
\lambda_2+\lambda_4=z(\lambda_1+\lambda_2+\lambda_3+\lambda_4),\\
\lambda_1+\lambda_2+\lambda_3+\lambda_4=1,\;
 \lambda_i\geq 0,\;\; i=1,\ldots,4\,.
\end{array}
\right.
\end{equation}
\emph{(2\textsuperscript{nd} cycle) Step 1.\,} By setting  $z=0$ in System (\ref{S_1'}), we
obtain
\begin{equation}\label{S_1'z=0}
\left\{
\begin{array}{ll}
\lambda_1=1-y,\;\;
\lambda_2=\lambda_4=0, \;\;
\lambda_3=y, \;\;
\lambda_i\geq 0,\;\; i=1,\ldots,4\,.
\end{array}
\right.
\end{equation}
As $y \in [0,1]$,  System (\ref{S_1'z=0})  is always solvable; thus, we go to Step 3.  \\
\emph{(2\textsuperscript{nd} cycle) Step 3}. 
We denote by $\Lambda$  and $\mathcal{S}$ the vector of unknowns $(\lambda_1,\ldots,\lambda_4)$ and the set of solutions of System~(\ref{S_1'z=0}), respectively. The conditioning events are  $H_2=A$ and $H_4=A$; then the associated   linear functions are:
$\Phi_2(\Lambda)=\Phi_4(\Lambda)=\sum_{r:C_r\subseteq A}\lambda_r=\lambda_1+\lambda_2+\lambda_3+\lambda_4$.
From System~(\ref{S_1'z=0}), we obtain: $\Phi_2(\Lambda)=\Phi_4(\Lambda)=1$, $\forall \Lambda \in \mathcal{S}$; so that $M_2=M_4=1$.
We are in the first case  of Step 3 of the algorithm; then the \emph{procedure stops} and yields   $z'=0$ as \emph{output}. 

To summarize,  for any  $(x,y,t)\in[0,1]^3$ on $(B|C,\widebar{B}|A,A|(A\vee C))$, we have computed the  coherent lower bound $z'$ on $\widebar{C}|A$. 
In particular,  if $t=0$, then $z'=0$. We also have   $z'=0$, when 
 $t>0$ and $t\leq x+yt\leq 1$, that is when $0<t\leq x+yt\leq 1$.  Then,  we can write that    $z'=0$, when 
 $ t\leq x+yt\leq 1$. Otherwise,  we have two cases: $(i.1)$ $z'=\frac{x+yt-1}{xt}$, if $x+yt >1$;  
 $(i.2)$ $z'=\frac{t-yt-x}{(1-x)t}$, if $x+yt <t$.

\paragraph{Computation of the upper  probability bound  $z''$ on $\widebar{C}|A$}~\\
\emph{Input} and \emph{Step 0} are  the same as  in the proof of $z'$.\\ 
\emph{Step 1.\,} By setting $z=1$ in  System (\ref{S_0'}), we   obtain
\begin{equation}\label{S_0'z=1}
\begin{array}{lll}
\left\{
\begin{array}{lllllll}
\lambda_1+\lambda_3=0, \;
\lambda_5=x(\lambda_5+\lambda_6), \\
\lambda_4=y(\lambda_2+\lambda_4), \;
\lambda_2+\lambda_4=t, \\
\lambda_2+\lambda_4+\lambda_5+\lambda_6=1, \\
\lambda_i\geq 0,\; i=1,\ldots,6\,.
\end{array}
\right.
&  \Longleftrightarrow  &
\left\{
\begin{array}{lllllll}
\lambda_1=\lambda_3=0, \;
\lambda_2=t(1-y), \\
\lambda_4=yt, \;
\lambda_5=x(1-t), \\
\lambda_6=(1-x)(1-t), \\
\lambda_i\geq 0,\; i=1,\ldots,6\,.
\end{array}
\right.
\end{array}
\end{equation}
As $(x,y,t) \in [0,1]^3$,  System (\ref{S_0'z=1})   is  solvable and we go to Step~3. \\
\emph{ Step~3.}
We denote by $\Lambda$  and $\mathcal{S}$ the vector of unknowns $(\lambda_1,\ldots,\lambda_6)$ and the set of solutions of System~(\ref{S_0'z=1}), respectively. We consider the  functions  given in (\ref{EQ:PHIFIGII}). From System~(\ref{S_0'z=1}), we  obtain:
 $\Phi_{1}(\Lambda)=0+0+x(1-t)+(1-x)(1-t)=1-t$,  $\Phi_{2}(\Lambda)=\Phi_{4}(\Lambda)=0+t(1-y)+0+yt=t$,  $\Phi_3(\Lambda)=0+t(1-y)+0+yt+x(1-t)+(1-x)(1-t)=1$, $\forall \Lambda \in \mathcal{S}$.
Then,
 $M_1=1-t$,  $M_2=M_4=t$, and $M_3=1$. If $t>0$, then $M_{4}>0$ and we are in the first case  of Step 3. Thus,  the \emph{procedure stops} and yields $z''=1$ as \emph{output}.  If $t=0$, then
$M_{1}>0,M_{3}>0$ and $M_{2}=M_{4}=0$. Hence, we are in the third case of Step 3 with $J = \{2\}, I_0=\{2,4\}$ and the procedure restarts with Step 0, with $\F$ replaced by  $\F_J=(E_2|H_2)=(\widebar{B}|A)$ and $\P$ replaced by $\P_J=y$.\\
\emph{(2\textsuperscript{nd} cycle) Step 0.\,} 
This is the same as the (2\textsuperscript{nd} cycle)  Step 0 in the proof of $z'$. \\
\emph{(2\textsuperscript{nd} cycle) Step 1.\,} By setting  $z=1$ in  System (\ref{S_0'}), we   obtain
\begin{equation}\label{S_1'z=1}
\left\{
\begin{array}{ll}
\lambda_1+\lambda_3=0, \;\;
\lambda_4=y,\;\;
\lambda_2=1-y, \;\;
\lambda_i\geq 0,\;\; i=1,\ldots,4\,.
\end{array}
\right.
\end{equation}
As $y \in [0,1]$,  System (\ref{S_1'z=1})  is always solvable; thus, we go to Step 3. \\
\emph{(2\textsuperscript{nd} cycle) Step 3.} Like in  the (2\textsuperscript{nd} cycle)  Step 3 of the proof of $z'$,
we obtain $M_4=1$. Thus, the procedure stops and yields   $z''=1$ as output.
To summarize,  for any  assessment  $(x,y,t)\in[0,1]^3$ on $(B|C,\widebar{B}|A,A|(A\vee C))$, we have computed the  coherent upper probability bound $z''$ on $\widebar{C}|A$, which is always $z''=1$.
\end{proof}
\begin{remark}
We observe that in Theorem~\ref{THM:PROPF2}  we do not presuppose, differently from the classical approach, positive probability for the conditioning events ($A$ and $C$). 
For example,  even if we assume $p(A|(A\vee C))=t>0$ we do not require positive probability for the conditioning event $A$, and $p(A)$ could be zero (indeed, since $p(A)=p(A\wedge (A\vee C))=p(A|(A\vee C))p(A\vee C)$, $p(A)>0$ implies  $p(A|(A\vee C))>0$, but not \emph{vice versa}).
\end{remark}

The next result  is based on Theorem~\ref{THM:PROPF2} and presents the set of the coherent extensions of a given interval-valued probability assessment $\mathcal{I}=([x_1,x_2]\times [y_1,y_2]\times [t_1,t_2])\subseteq [0,1]^3$  on the sequence  on $(B|C,\widebar{B}|A,A|(A\vee C))$ to the further conditional event
$\widebar{C}|A$.

\begin{theorem}\label{THM:PROPWTIV}
    Let $A,B,C$ be three logically independent events and $\mathcal{I}=([x_1,x_2]\times [y_1,y_2]\times [t_1,t_2])\subseteq [0,1]^3$  be an imprecise    assessment on  $(B|C,\widebar{B}|A,A|(A\vee C))$. Then, the set $\Sigma$ of the coherent extensions of $\mathcal{I}$ on $\widebar{C}|A$ is the interval  $[z^{*},z^{**}]$, where 
	\[
\begin{array}{ll}
[z^*,z^{**}] =
\left\{
\begin{array}{cl}
\displaystyle\left[0,1\right], &  \mbox{if } (x_2+y_2t_1\geq  t_1) \wedge (x_1+y_1t_1\leq  1),\\~
\displaystyle\big [\frac{x_1 + y_1t_1 - 1}{t_1 x_1}, 1\big] \,, &\mbox{if } x_1+y_1t_1> 1, \\~
\displaystyle\big[\frac{t_1-x_2-y_2t_1 }{t_1 (1-x_2)}, 1\big] \,, & \mbox{if } x_2+y_2t_1< t_1. 
\end{array}
\right.
\end{array}
\]
\end{theorem}
\begin{proof}
As from Proposition \ref{PROP:FIGUREIIBIS} the set $[0,1]^3$ on $(B|C,\widebar{B}|A,A|(A\vee C))$ is totally coherent, then 
$\mathcal{I}$ is totally coherent too. Then,  $\Sigma=\bigcup_{\P\in \I}[z_{\P}',z_{\P}'']=[z^*,z^{**}]$, 
where $z^*=\inf_{\P\in \I}z_{\P}'$ 
(i.e., 
$z^*=\inf\{z_{\P}':\P\in \I\}$)	and $z^{**}=\sup_{\P\in \I}z_{\P}''$ 
(i.e., 
$z^{**}=\sup\{z_{\P}':\P\in \I\}$).
We distinguish three alternative cases: $(i)$ $x_1+y_1t_1> 1$; $(ii)$ $x_2+y_2t_1< t_1$; $(iii)$ $(x_2+y_2t_1\geq  t_1) \wedge (x_1+y_1t_1\leq  1)$.\\
Of course, for all three cases $z^{**}=\sup_{\P\in \I} z''_{\P}=1$.\\
Case $(i)$. We observe that  the function  $x+yt:[0,1]^3$ is  nondecreasing in the arguments  $x,y,t$.
Then, in this case, $x+yt\geq x_1+y_1t_1> 1$ for every $\P=(x,y,t)\in \mathcal{I}$ and hence by Theorem~\ref{THM:PROPF2}
$z'_{\P}=f(x,y,t)=\frac{x + yt - 1}{t\, x}$ for every $\P\in \mathcal{I}$. 
Moreover, $f(x,y,t):[0,1]^3$  is nondecreasing in the arguments  $x,y,t$, thus $z^*=\frac{x_1 + y_1t_1 - 1}{t_1 x_1}$.\\ 
Case $(ii)$. We observe that  the function  $x+yt-t:[0,1]^3$ is  nondecreasing in the arguments  $x,y$ and nonincreasing in the argument $t$.
Then, in this case, $x+yt-t\leq x_2+y_2t_1-t_1<0$ for every $\P=(x,y,t)\in \mathcal{I}$ and hence by Theorem~\ref{THM:PROPF2}
$z'_{\P}=g(x,y,t)=\frac{t-x-yt }{t (1-x)}$ for every $\P\in \mathcal{I}$. 
Moreover, $g(x,y,t):[0,1]^3$  is nonincreasing in the arguments  $x,y$ and nondecreasing in the argument $t$.
Thus, $z^*=\frac{t_1-x_2-y_2t_1 }{t_1 (1-x_2)}$. 
Case $(iii)$. In this case there exists a vector $(x,y,t)\in \mathcal{I}$ such that $t\leq x+yt\leq 1$ and hence by Theorem~\ref{THM:PROPF2}
$z'_{\P}=0$. Thus, $z^*=0$.
\end{proof}
\begin{remark}\label{REM:AFF}
By instantiating  Theorem \ref{THM:PROPWTIV} with the imprecise assessment $\I=\{1\}\times[y_1,1]\times[t_1,1]$, where $t_1>0$, we obtain the following lower and upper bounds for the conclusion $[z^*,z^{**}]=[y_1,1]$.  Thus, for every $t_1>0$: $z^*$ depends only on the value of $y_1$.
\end{remark}

\subsection{Traditionally valid syllogisms of Figure II.}
\label{SEC:FIGIISYLLO}
In this section we consider the probabilistic interpretation of the traditionally valid  syllogisms of Figure II (Camestres, Camestrop, Baroco, Cesare, Cesaro, Festino; see Table 
\ref{TAB:AristSyl}).  Like in Figure I, all  syllogisms of Figure II without  existential import assumptions are probabilistically non-informative. Indeed, by instantiating $S$, $M$, $P$ for $A$, $B$, $C$, respectively, in Proposition \ref{PROP:FIGUREII},  we observe that  the imprecise assessment $[0,1]^3$ on $(M|P,\no{M}|S,\no{P}|S)$ is t-coherent.  For instance,  Camestres
(``\emph{Every $P$ is $M$, No $S$ is $M$, therefore No $S$ is $P$}'') without  existential import assumption corresponds to the probabilistically non-informative 
inference:  from the premises $p(M|P)=1$ and $p(\widebar{M}|S)=1$ infer that every $p(\no{P}|S)\in[0,1]$ is coherent (see Proposition \ref{PROP:FIGUREII}). 
Therefore we add the conditional event  existential import assumption:  $p(S|(S \vee P))>0$ (see Definition \ref{DEF:CEI}). In what follows, we
 construct  (s-)valid versions of the traditionally valid syllogisms  of Figure II,  by  suitable instantiations in Theorem~\ref{THM:PROPF2}.
\paragraph{Camestres} 
By instantiating $S,M,P$ in Theorem~\ref{THM:PROPF2} for $A,B,C$ with $x=y=1$ and $t>0$
it follows that  $z'=\frac{x + yt - 1}{t\, x}=1$ and $z''=1$. Then,  the set $\Sigma$    of coherent extensions on $\no{P}|S$ of the  imprecise assessment $\{1\}\times\{1\}\times(0,1]$ on  $(M|P,\no{M}|S,S|(S\vee P))$ is $\Sigma=\{1\}$.
Thus, by Definition \ref{DEF:VALID},
\begin{equation}\label{EQ:CAMESTRESIP}
\{1\}\times\{1\}\times(0,1] \mbox{ on } (M|P,\no{M}|S,S|(S\vee P)) \;\models_s\; \{1\} \mbox{ on } \no{P}|S. \end{equation}
In terms of probabilistic constraints, (\ref{EQ:CAMESTRESIP}) can be equivalently expressed by (see Table~\ref{Table:BS}) 
\begin{equation}\label{EQ:CAMESTRES} 
(p(M|P)=1, p(M|S)=0, p(S|(S \vee P))>0) \,\models_s \, p(P|S)= 0\,,
\end{equation}
which is a  s-valid version of Camestres.
\paragraph{Camestrop}
From (\ref{EQ:CAMESTRES}), by weakening the conclusion of Camestres, it  follows that 
\begin{equation}\label{EQ:CAMESTROP*}
(p(M|P)=1, p(M|S)=0, p(S|(S \vee P))>0)\; \models \; p(P|S)< 1\,,
\end{equation}
which is equivalent to 
\begin{equation}\label{EQ:CAMESTROP}
(p(M|P)=1, p(M|S)=0, p(S|(S \vee P))>0)\; \models \; p(\no{P}|S)> 0.
\end{equation}
Inference (\ref{EQ:CAMESTROP})
 is a valid (but not s-valid) version of  Camestrop (\emph{Every $P$ is $M$, No $S$ is $M$, therefore Some $S$ is not $P$}). 

\paragraph{Baroco}
By instantiating $S,M,P$ in Theorem~\ref{THM:PROPF2} for $A,B,C$ with $x=1$, any $y>0$, and  any $t>0$,
it follows that  $z'=\frac{x + yt - 1}{t\, x}=\frac{1+yt-1}{t}=y>0$. Then,  the set $\Sigma$    of coherent extensions on $\no{P}|S$ of the  imprecise assessment 
$\{1\}\times(0,1]\times(0,1]$ 
on $(M|P,\no{M}|S,S|(S\vee P))$
is $\Sigma=\bigcup_{
	\{(x,y,t)\in \{1\}\times (0,1]\times (0,1]\}} [y,1]=(0,1]$.
Thus, by Definition \ref{DEF:VALID},
\begin{equation}\label{EQ:BAROCOIP}
\{1\}\times(0,1]\times(0,1] \mbox{ on } (M|P,\no{M}|S,S|(S\vee P)) \;\models_s\; (0,1] \mbox{ on } \no{P}|S. \end{equation}
In terms of probabilistic constraints, (\ref{EQ:BAROCOIP}) can be expressed by,
\begin{equation}\label{EQ:BAROCO}
(p(M|P)=1, p(\no{M}|S)>0, p(S|(S \vee P))>0)\; \models_s \; p(\no{P}|S)> 0\,.
\end{equation}
Therefore, inference (\ref{EQ:BAROCO}) is a s-valid version of Baroco (\emph{Every $P$ is $M$, Some $S$ is not $M$, therefore Some $S$ is not $P$}). 
 Notice that Camestrop (\ref{EQ:CAMESTROP})  also follows from Baroco (\ref{EQ:BAROCO}) by strengthening the  minor premise.
\paragraph{Cesare} 
By instantiating $S,M,P$ in Theorem~\ref{THM:PROPF2} for $A,B,C$ with $x=y=0$ and any $t>0$,
it follows that  $z'=\frac{t-x-yt }{t\, (1-x)}=1$ (and $z''=1$). 
Then,  the set $\Sigma$    of coherent extensions on $\no{P}|S$ of the  imprecise assessment 
$\{0\}\times\{0\}\times(0,1]$ 
on $(M|P,\no{M}|S,S|(S\vee P)$
is $\Sigma=\{1\}$.
Thus, by Definition \ref{DEF:VALID},
\begin{equation}\label{EQ:CESAREIP}
\{0\}\times\{0\}\times(0,1] \mbox{ on } (M|P,\no{M}|S,S|(S\vee P)) \;\models_s\; \{1\} \mbox{ on } \no{P}|S. \end{equation}
In terms of probabilistic constraints, (\ref{EQ:CESAREIP}) can be expressed by,
\begin{equation*}
(p(M|P)=0, p(\no{M}|S)=0, p(S|(S \vee P))>0)\; \models_s \; p(\widebar{P}|S)=1 \,,
\end{equation*}
or equivalently by 
\begin{equation}\label{EQ:CESARE}
(p(M|P)=0, p(M|S)=1, p(S|(S \vee P))>0)\; \models_s \; p(P|S)=0 \,.
\end{equation}
Therefore, inference (\ref{EQ:CESARE}) is a s-valid version of Cesare
(\emph{No $P$ is $M$, Every $S$ is $M$, therefore No $S$ is $P$}).
Notice that Cesare is equivalent to Camestres, because 
(\ref{EQ:CESARE}) 
is equivalent to 
(\ref{EQ:CAMESTRES})
when  $M$ is replaced by  $\no{M}$ (and the  probabilities are adjusted accordingly).

\paragraph{Cesaro}
From (\ref{EQ:CESARE}), by weakening the conclusion of Cesare, it  follows that 
\begin{equation}\label{EQ:CESARO}
(p(M|P)=0, p(M|S)=1, p(S|(S \vee P))>0) \; \models \; p(\widebar{P}|S)>0\,,
\end{equation}
which is a valid (but not s-valid) version of Cesaro (\emph{No $P$ is $M$, Every $S$ is $M$, therefore Some $S$ is not $P$}). 
Notice that Cesaro is equivalent to Camestrop, because 
(\ref{EQ:CESARO}) 
is equivalent to 
(\ref{EQ:CAMESTROP})
when  $M$ is replaced by  $\no{M}$.

\paragraph{Festino}
By instantiating $S,M,P$ in Theorem~\ref{THM:PROPF2} for $A,B,C$ with $x=0$, any $y<1$ and any $t>0$, as $x+yt<t$,
it follows that  $z'=\frac{t-x-yt }{t\, (1-x)}=\frac{t-yt}{t}>1-y>0$ (and $z''=1$).
Then,  the set $\Sigma$    of coherent extensions on $P|S$ of the  imprecise assessment $\{0\}\times[0,1)\times(0,1]$ on  $(M|P,\no{M}|S,S|(S\vee P))$ is $\Sigma=\bigcup_{
	\{(x,y,t)\in \{0\}\times [0,1)\times (0,1]\}
	} [1-y,1]=(0,1]$.
Thus, by Definition \ref{DEF:VALID},
\begin{equation}\label{EQ:FESTINOIP}
\{0\}\times[0,1)\times(0,1] \mbox{ on } (M|P,\no{M}|S,S|(S\vee P)) \;\models_s\; (0,1] \mbox{ on } \no{P}|S. \end{equation}
In terms of probabilistic constraints, (\ref{EQ:FESTINOIP}) can be equivalently expressed by 
\begin{equation}\label{EQ:FESTINO} 
(p(M|P)=0, p(M|S)>0, p(S|(S \vee P))>0) \,\models_s \, p(\no{P}|S)>0 \,,
\end{equation}
which is a  s-valid version of Festino 
(\emph{No $P$ is $M$, Some $S$ is  $M$, therefore Some $S$ is not $P$}). 
Notice that Festino is equivalent to Baroco, because
(\ref{EQ:FESTINO}) 
is equivalent to 
(\ref{EQ:BAROCO})
when  $M$ is replaced by  $\no{M}$.
Cesaro (\ref{EQ:CESARO}) also follows from Festino (\ref{EQ:FESTINO}) by strengthening the minor premise.

\begin{remark}
We observe that, traditionally, the conclusions of logically valid Aristotelian syllogisms of Figure II are neither  in the form of sentence type  I (\emph{some}) nor of A (\emph{every}). In terms of our probability semantics, indeed, this must be the case even if  the existential import assumption $p(S|(S \vee P))>0$ is made:
according to Theorem~\ref{THM:PROPF2}, the upper  bound for the conclusion $p(\widebar{P}|S)$ is always 1; thus, neither sentence type I ($p(P|S)>0$, i.e. $p(\widebar{P}|S)<1$) nor sentence type A ($p(P|S)=1$, i.e. $p(\widebar{P}|S)=0$) can be validated.
\end{remark}

\section{Figure III.}
\label{SEC:FIGUREIII}
In this section, we 
 observe that the probabilistic inference of $C|A$ from the premise set $(C|B, A|B)$, which corresponds to the general form of syllogisms of Figure III, is probabilistically non-informative (Proposition \ref{PROP:FIGUREIII}).
 Therefore, we add the probabilistic constraint $p(B|(A \vee  B))>0$, as conditional event existential import assumption,  to obtain probabilistic informativeness. 
Then, we prove  the precise and imprecise probability propagation rules for the inference from $(C|B,A|B,B|(A \vee  B))$ to  $C|A$.
We apply these results  in Section~\ref{SEC:FIGIIISYLLO}, where we  study the valid syllogisms of Figure III.
\subsection{Coherence and  probability propagation   in Figure III.}
\begin{proposition}\label{PROP:FIGUREIII}
	Let $A,B,C$ be logically independent events.
	The imprecise assessment $[0,1]^3 $ on $ \mathcal{F}=(C|B,A|B,C|A)$ is totally  coherent. 
\end{proposition}
\begin{proof}
	By exchanging $B$ and $A$ and by reordering the sequence $\mathcal{F}$, Proposition \ref{PROP:FIGUREIII} is equivalent to 
 Proposition \ref{PROP:FIGUREI}.
\end{proof}	

Now we show that the  imprecise assessment  $[0,1]^3$ on the sequence of conditional events $(C|B,A|B,B|(A\vee B))$ is t-coherent. Note that  the  strategy used in the proof of Proposition \ref{PROP:FIGUREIII}  cannot be applied  for  proving Proposition \ref{PROP:FIGUREIIIBIS}.
\begin{proposition}\label{PROP:FIGUREIIIBIS}
	Let $A,B,C$ be logically independent events.
	The imprecise assessment $[0,1]^3 $ on $ \mathcal{F}=(C|B,A|B,B|(A\vee B))$ is totally  coherent. 
\end{proposition}
\begin{proof} 
	Let $\P=(x,y,t)\in[0,1]^3$ be a probability assessment on $\mathcal{F}$.
	The constituents $C_h$ and  the  points $Q_h$ associated with   $(\F,\P)$ are
	given in Table \ref{TAB:TABLE_FIGUREIIIPREMISE}. 
	\begin{table}[!h]
		\begin{minipage}{\textwidth}\centering 
			\caption{Constituents $C_h$ and  points $Q_h$ associated with  the probability  assessment    $\mathcal{P}=(x,y,t)$  on 
				$\F=(C|B,A|B,B|(A\vee B))$ involved in the premise set of Figure~III.
			}
			\label{TAB:TABLE_FIGUREIIIPREMISE}
			\begin{tabular}{llll}\hline
		& $C_h$            & $Q_h$                          &  \\
		\hline
		$C_1$ \, & $ABC$         \,  & $(1,1,1)$ \,                   & $Q_1$   \\
		$C_2$ & $AB\no{C}$      & $(0,1,1)$                    & $Q_2$   \\
		$C_3$ & $A\no{B}$       & $(x,y,0)$                    & $Q_3$   \\
		$C_4$ & $\no{A}BC$ & $(1,0,1)$                    & $Q_4$   \\
        $C_5$ & $\no{A}B\no{C}$ & $(0,0,1)$                    & $Q_5$   \\		
		$C_0$ & $\no{A} \, \no{B}$   &$(x,y,t)$                       & $Q_0=\mathcal{P} $  \\\hline
	\end{tabular}
		\end{minipage}
	\end{table}
	By Theorem \ref{COER-P0}, coherence of  $\P=(x,y,z)$ on $\F$ requires that the following system is solvable
	\[\begin{array}{l}
	(\Systemsigma)  \hspace{1 cm}
	\mathcal{P}=\sum_{h=1}^5 \lambda_hQ_h,\;
	\sum_{h=1}^5 \lambda_h=1,\; \lambda_h\geq 0,\, h=1,\ldots,6,
	\end{array}
	\]
or equivalently
\begin{equation}\label{EQ:FIGURIIISYSTEMBIS}
\begin{array}{lll}
\left\{
\begin{array}{lllllll}
\lambda_1+\lambda_4=x(\lambda_1+\lambda_2+\lambda_4+\lambda_5), \\
\lambda_1+\lambda_2=y(\lambda_1+\lambda_2+\lambda_4+\lambda_5), \\
\lambda_1+\lambda_2+\lambda_4+\lambda_5=t(\lambda_1+\lambda_2+\lambda_3+\lambda_4+\lambda_5),\\
\lambda_1+\lambda_2+\lambda_3+\lambda_4+\lambda_5=1,\\
\lambda_i\geq 0,\; i=1,\ldots,5\,,
\end{array}
\right.
&  \Longleftrightarrow  &
\left\{
\begin{array}{lllllll}
\lambda_1+\lambda_4=xt, \\
\lambda_1+\lambda_2=yt, \\
\lambda_1+\lambda_2+\lambda_4+\lambda_5=t,\\
\lambda_3=1-t,\\
\lambda_i\geq 0,\; i=1,\ldots,5\,,
\end{array}
\right.
\end{array}
\end{equation}
that is
\[
\left\{
\begin{array}{lllllll}
\lambda_2=yt-\lambda_1, \\
\lambda_3=1-t,\\
\lambda_4=xt-\lambda_1, \\
\lambda_5=t-xt-yt+\lambda_1,\\
\lambda_i\geq 0,\; i=1,\ldots,5\,.
\end{array}
\right.
\]
System $(\Systemsigma) $ is solvable because  $t\max\{0,x+y-1\} \leq t\min\{x,y\}$, for every $(x,y,t)\in[0,1]^3$ and the set of  solutions  $\mathcal{S}$  consists of  the vectors $\Lambda=(\lambda_1,\ldots,\lambda_5)$ such that
\[
\left\{
\begin{array}{lllllll}
t\max\{0,x+y-1\}\leq \lambda_1\leq t\min\{x,y\},\\
\lambda_2=yt-\lambda_1, \\
\lambda_3=1-t,\\
\lambda_4=xt-\lambda_1, \\
\lambda_5=t-xt-yt+\lambda_1.\\
\end{array}
\right.
\]
Moreover, for each $\Lambda\in \mathcal{S}$ it holds that $\Phi_{1}(\Lambda)=\Phi_{2}(\Lambda)=\sum_{h:C_h\subseteq B}=\lambda_1+\lambda_2+\lambda_4+\lambda_5=t$, and $\Phi_{3}(\Lambda)=\sum_{h:C_h\subseteq A\vee B}\lambda_h=1$.
If $t>0$,  it follows that, for each $\Lambda\in \mathcal{S}$, $\Phi_{1}(\Lambda)=\Phi_{2}(\Lambda)>0$, and $\Phi_{3}(\Lambda)>0$. Then,
 $I_0=\emptyset$ and by Theorem \ref{COER-P0}, the assessment $(x,y,t)$ is coherent.
If $t=0$, it follows that  for each $\Lambda\in \mathcal{S}$, $\Phi_{1}(\Lambda)=\Phi_{2}(\Lambda)=0$. Then,  $I_0=\{1,2\}$ 
and as it is well known that the sub-assessment $(x,y)$ on $(C|B,A|B)$ is  coherent for every $(x,y)\in[0,1]^2$, it follows by Theorem \ref{COER-P0} that  $(x,y,t)$ is coherent.
 Then, $(x,y,t)$ is coherent for every $(x,y,t)\in[0,1]^3$.
\end{proof}

   The next  theorem presents the  coherent probability  propagation rules in  Figure~III under the conditional event existential import assumption. 
\begin{theorem}\label{THM:PROPF3}
	Let $A,B,C$ be three logically independent events and $(x,y,t)\in[0,1]^3$ be a (coherent)  assessment on the family $(C|B,A|B, B|(A\vee B))$. Then, the  extension $z=p(C|A)$ is coherent if and only if $z\in[z',z'']$, where 
	\[
	\small
	\begin{array}{ll}
	z'=
	\left\{
	\begin{array}{cl}
	0, &  \mbox{if } t(x+y-1)\leq 0, \\~
	\displaystyle \frac{t(x+y-1)}{1-t(1-y)}\,, &\mbox{if } t(x+y-1)>0,
	\end{array}
	\right. \;\;
	z''=	\left\{
	\begin{array}{cl}
	1, &  \mbox{if }  t(y-x)\leq 0, \\~
	\displaystyle 1-\frac{t(y-x)}{1-t(1-y)}\,, &\mbox{if } t(y-x)>0.
	\end{array}
	\right.\\
	\end{array}
	\]	
\end{theorem}
\begin{proof}
	In order to compute the  lower and upper probability bounds $z'$ and $z''$ on the further event $C|A$ (i.e., the conclusion), 
	we  apply  Algorithm~\ref{Alg} in a symbolic way. 
	\paragraph{Computation of the lower probability bound  $z'$ on $C|A$}~\\
	\emph{Input.\,}  The assessment $(x,y,t)$ on $\F=(C|B,A|B, B|(A\vee B))$ and the event  $C|A$.\\
	\emph{Step 0.\,}  The constituents  associated with  $(C|B,A|B, B|(A\vee B),C|A)$ are
	$C_0=\widebar{A}\widebar{B},\;C_1=ABC,\, C_2=A\widebar{B}C,\, C_3=AB\widebar{C},\, C_4=A \widebar{B}  \widebar{C},\, C_5=\widebar{A}BC,\, C_6=\widebar{A} B \widebar{C}
	$.
	We observe	that  $\mathcal{H}_{0}= A \vee B$; then, the constituents contained in  $\mathcal{H}_{0}$ are $C_1,\ldots,C_6$.
	We construct the  starting
	system with the unknowns $\lambda_1,\ldots,\lambda_6,z$:
	\begin{equation}
	\label{S_0'F3}
	\begin{array}{lccccl}
	\left\{
	\begin{array}{lllllll}
	\lambda_1+\lambda_2=z(\lambda_1+\lambda_2+\lambda_3+\lambda_4), \\
	\lambda_1+\lambda_5=x(\lambda_1+\lambda_3+\lambda_5+\lambda_6),  \\
	\lambda_1+\lambda_3=y(\lambda_1+\lambda_3+\lambda_5+\lambda_6),  \\
	\lambda_1+\lambda_3+\lambda_5+\lambda_6=t(\sum_{i=1}^6\lambda_i), \\ 
	\sum_{i=1}^6\lambda_i=1, \;
	\lambda_i\geq 0,\; i=1,\ldots,6\,,
	\end{array}
	\right.
	&
	\Longleftrightarrow  &
	\left\{
	\begin{array}{lllllll}
	\lambda_1+\lambda_2=z(\lambda_1+\lambda_2+\lambda_3+\lambda_4),  \\
	\lambda_1+\lambda_5=xt,  \\
	\lambda_1+\lambda_3=yt,  \\
	\lambda_1+\lambda_3+\lambda_5+\lambda_6=t,\\
	\sum_{i=1}^6\lambda_i=1, \;
	\lambda_i\geq 0,\; i=1,\ldots,6\,.
	\end{array}
	\right.
	\end{array}
	\end{equation}
	
	\emph{Step 1.\,} By setting  $z=0$ in  System (\ref{S_0'F3}), we   obtain
	\begin{equation}\label{S_0'z=0F3}
	\begin{array}{lcl}
	\left\{
	\begin{array}{lllllll}
	\lambda_1+\lambda_2=0, \;\;
	\lambda_3=yt, \;\;	\lambda_5=xt, \\
	\lambda_3+\lambda_5+\lambda_6=t, \\
	\lambda_3+\lambda_4+\lambda_5+\lambda_6=1, \\
	\lambda_i\geq 0,\; i=1,\ldots,6\,.
	\end{array}
	\right.
	&
	\quad \Longleftrightarrow \quad &
	\left\{
	\begin{array}{lllllll}
	\lambda_1=\lambda_2=0, \\
	\lambda_3=yt, \;\;
	\lambda_4=1-t,\;\;
	\lambda_5=xt, \\
	\lambda_6=t(1-x-y), \\
	\lambda_i\geq 0,\; i=1,\ldots,6\,.
	\end{array}
	\right.
	\end{array}
	\end{equation}
As $(x,y,t) \in [0,1]^3$,
	the conditions $\lambda_h\geq 0$, $h=1,\ldots,5$, in 
	System (\ref{S_0'z=0F3}) are all satisfied. Then, System (\ref{S_0'z=0F3}), i.e.      System  (\ref{S_0'F3}) with $z=0$,  
	is solvable  if and only if 
	${\lambda_6=t(1-x-y)\geq 0}$.
	We distinguish two cases:  $(i)$  $t(1-x-y)<0$ (i.e. $t>0$ and $x+y>1$); $(ii)$ $t(1-x-y)\geq 0$, (i.e. $t=0$ or $(t>0) \wedge (x+y\leq 1)$).  In Case~$(i)$,  System (\ref{S_0'z=0F3})  is  not solvable and we go to Step~2 of the algorithm. In Case~$(ii)$,  System (\ref{S_0'z=0F3}) is  solvable and we go to Step~3. 
	
	\emph{Case $(i)$.}  By Step 2 we have the following linear programming problem:\\
	\emph{Compute $\gamma'=\min(\sum_{i:C_i\subseteq AC}  \lambda_r)=\min(  \lambda_1+\lambda_2)$ subject to:}
	\begin{equation}\label{Sigma'F3}
	\left\{
	\begin{array}{lllllll}
	\lambda_1+\lambda_5=x(\lambda_1+\lambda_3+\lambda_5+\lambda_6),\;  
	\lambda_1+\lambda_3=y(\lambda_1+\lambda_3+\lambda_5+\lambda_6), \\
	\lambda_1+\lambda_3+\lambda_5+\lambda_6=t(\sum_{i=1}^6\lambda_i),\; 
	\lambda_1+\lambda_2+\lambda_3+\lambda_4=1, \;\;\\
	\lambda_i\geq 0,\; i=1,\ldots,6.
	\end{array}
	\right.
	\end{equation}
	We notice that $y$ is positive since $x+y>1$ (and $(x,y,t) \in [0,1]^3$). Then, also $1-t(1-y)$ is positive and  the constraints in (\ref{Sigma'F3}) can be rewritten as
\[
	\begin{array}{ll}
	\left\{
	\begin{array}{lllllll}
	\lambda_1+\lambda_5=xt(1+\lambda_5+\lambda_6), \\
	\lambda_1+\lambda_3=yt(1+\lambda_5+\lambda_6), \\
	\lambda_5+\lambda_6=(t-yt)(1+\lambda_5+\lambda_6),\\ 
	\lambda_1+\lambda_2+\lambda_3+\lambda_4=1, \;\;\\
	\lambda_i\geq 0,\; i=1,\ldots,6,\\
	\end{array}
	\right.
	\end{array}
	\Longleftrightarrow
	\begin{array}{ll}
	\left\{
	\begin{array}{lllllll}
	\lambda_5+\lambda_6=\frac{t(1-y)}{1-t(1-y)}, \;\; \\
	\lambda_1+\lambda_5=xt(1+\frac{t(1-y)}{1-t(1-y)})=\frac{xt}{1-t(1-y)}, \\
	\lambda_1+\lambda_3=yt(1+\frac{t(1-y)}{1-t(1-y)})=\frac{yt}{1-t(1-y)}, \\
	\lambda_1+\lambda_2+\lambda_3+\lambda_4=1, \;\;\\
	\lambda_i\geq 0,\; i=1,\ldots,6,\\
	\end{array}
	\right.
	\end{array}
	\]
	\begin{equation}\label{EQ:SOLF3}
	\Longleftrightarrow
	\begin{array}{ll}
	\left\{
	\begin{array}{lllllll}
	\max\{0,\frac{t(x+y-1)}{1-t(1-y)}\}\leq \lambda_1\leq \min\{x,y\}\frac{t}{1-t(1-y)},\\
	0\leq \lambda_2\leq \frac{1-t}{1-t(1-y)},\;\;\;	\lambda_3=\frac{yt}{1-t(1-y)}-\lambda_1, \;\;\;
	\lambda_4=\frac{1-t}{1-t(1-y)}-\lambda_2, \;\;\\
	\lambda_5=\frac{xt}{1-t(1-y)}-\lambda_1, \;\;\;
	\lambda_6=\frac{t(1-x-y)}{1-t(1-y)}+\lambda_1. \;\; \\
	\end{array}
	\right.
	\end{array}
	\end{equation}
	Thus, by recalling that $x+y-1>0$, the minimum $\gamma'$ of $\lambda_1 +\lambda_2$ subject to (\ref{Sigma'F3}), or equivalently subject to (\ref{EQ:SOLF3}), is obtained at $(\lambda_1',\lambda_2')=(\frac{t(x+y-1)}{1-t(1-y)},0)$.
	The \emph{procedure stops} yielding as \emph{output} 
	$z'=\gamma'=\lambda_1'+\lambda_2'=\frac{t(x+y-1)}{1-t(1-y)}$.\\
	
	\emph{Case $(ii)$.}   We take Step~3 of the algorithm. We denote by $\Lambda$  and $\mathcal{S}$ the vector of unknowns $(\lambda_1,\ldots,\lambda_6)$ and the set of solutions of System~(\ref{S_0'z=0F3}), respectively.
	We consider the following linear functions (associated with the conditioning events $H_1=H_2=B, H_3=A\vee B, H_4=A$) and their maxima in $\mathcal{S}$: 
	
	\begin{equation}\label{EQ:PHIF3}
	\begin{array}{l}
	\Phi_{1}(\Lambda)=\Phi_{2}(\Lambda)=\sum_{r:C_r\subseteq B}\lambda_r= \lambda_1+\lambda_3+\lambda_5+\lambda_6, \\
	\Phi_{3}(\Lambda)=\sum_{r:C_r\subseteq A \vee B}\lambda_r=\lambda_1+\lambda_2+\lambda_3+\lambda_4+\lambda_5+\lambda_6, \\ 
	\Phi_{4}(\Lambda)=\sum_{r:C_r\subseteq A}\lambda_r=\lambda_1+\lambda_2+\lambda_3+\lambda_4,\; 
	M_i=\max_{\Lambda \in \mathcal{S}} \Phi_i(\Lambda),\;  i=1,2,3,4\;.
	\end{array}
	\end{equation}
	By (\ref{S_0'z=0F3}) we obtain: $\Phi_{1}(\Lambda)=\Phi_{2}(\Lambda)=0+yt+xt+t-xt-yt=t$,  $\Phi_3(\Lambda)=1$, $\Phi_{4}(\Lambda)=yt+1-t=1-t(1-y)$, $\forall \Lambda \in \mathcal{S}$.
	Then,  $M_1=M_2=t$,  $M_3=1$, and $M_4=1-(1-y)t$.
	We consider two subcases: $t<1$; $t=1$. 
	If $t<1$, then $M_{4}=yt+1-t>yt\geq0$; so that $M_{4}>0$  and we are in the first case  of Step~3 (i.e., $M_{n+1}>0$). Thus,  the \emph{procedure stops} and yields $z'=0$ as \emph{output}.  
	If $t=1$, then
	$M_{1}=M_{2}=M_{3}=1>0$ and $M_{4}=y$. Hence, we are	in the first case  of Step~3 (when $y>0$) or in the second case of Step~3 (when $y=0$). Thus,  the \emph{procedure stops} and yields $z'=0$ as \emph{output}.  
	\paragraph{Computation of the upper  probability bound  $z''$ on $C|A$}
	\emph{Input} and \emph{Step 0} are  the same as  in the proof of $z'$. 
	\emph{Step 1.\,} By setting $z=1$ in  System (\ref{S_0'F3}), we   obtain
	\begin{equation*}
	\left\{
	\begin{array}{lllllll}
	\lambda_1+\lambda_2=\lambda_1+\lambda_2+\lambda_3+\lambda_4,\; 
	\lambda_1+\lambda_5=xt,\; 
	\lambda_1+\lambda_3=yt,\\
	\lambda_1+\lambda_3+\lambda_5+\lambda_6=t, \;
	\lambda_1+\lambda_2+\lambda_3+\lambda_4+\lambda_5+\lambda_6=1, \;
	\lambda_i\geq 0,\; i=1,\ldots,6\,,
	\end{array}
	\right.
	\end{equation*}
	or equivalently
	\begin{equation}\label{S_0'z=1F3}
	\begin{array}{lll}
	\left\{
	\begin{array}{lllllll}
	\lambda_3=\lambda_4=0,\;
	\lambda_1+\lambda_5=xt,\\
	\lambda_1=yt,\;
	\lambda_1+\lambda_5+\lambda_6=t,\\
	\lambda_1+\lambda_2+\lambda_5+\lambda_6=1, \\
	\lambda_i\geq 0,\; i=1,\ldots,6\,;
	\end{array}
	\right.
	
	&  \Longleftrightarrow  &
	\left\{
	\begin{array}{lllllll}
	\lambda_1=yt,\;
	\lambda_2=1-t, \;
	\lambda_3=\lambda_4=0,\\
	\lambda_5=(x-y)t,\;
	\lambda_6=t(1-x),\\
	\lambda_i\geq 0,\; i=1,\ldots,6\,.
	\end{array}
	\right.
	\end{array}
	\end{equation}
	As $(x,y,t) \in [0,1]^3$, the inequalities $\lambda_h\geq 0$, $h=1,2,3,4,6$ are satisfied.
	Then,
	System (\ref{S_0'z=1F3}), i.e. System  (\ref{S_0'F3}) with $z=1$,  is solvable  if and only if $\lambda_5=(x-y)t\geq 0$.
	We distinguish two cases:  $(i)$ $(x-y)t<0$, i.e. $x<y$ and  $t>0$; $(ii)$
	$(x-y)t\geq 0$, i.e.  $x\geq y$ or  $t=0$.  In Case~$(i)$,  System~(\ref{S_0'z=1F3})   is  not solvable and we go to Step~2 of the algorithm. In Case~$(ii)$,  System~(\ref{S_0'z=1F3})  is  solvable and we go to Step~3. \\
	\emph{Case $(i)$.} By Step 2 we have the following linear programming problem:\\
	\emph{Compute $\gamma''=\max(  \lambda_1+\lambda_2)$ subject to
		the constraints in (\ref{Sigma'F3})}. 
	As $(x,y,t) \in [0,1]^3$ and $x<y$, it follows that  $\min\{x,y\}=x$ and $y>0$. Then, in this case the quantity $1-t(1-y)$ is positive and the constraints in (\ref{Sigma'F3}) can be rewritten as in (\ref{EQ:SOLF3}). 
	Thus, the maximum $\gamma''$ of $\lambda_1 +\lambda_2$ subject to  (\ref{EQ:SOLF3}), is obtained at $(\lambda_1'',\lambda_2'')=(\frac{xt}{1-t(1-y)},\frac{1-t}{1-t(1-y)})$.
	The \emph{procedure stops} yielding as \emph{output} 
	$z''=\gamma''=\lambda_1''+\lambda_2''=\frac{xt}{1-t(1-y)}+\frac{1-t}{1-t(1-y)}=
	\frac{1-t+xt}{1-t+yt}=1-\frac{t(y-x)}{1-t+yt}$.\\
	
	\emph{Case} $(ii)$. We take Step 3 of the algorithm.
	We denote by $\Lambda$  and $\mathcal{S}$ the vector of unknowns $(\lambda_1,\ldots,\lambda_6)$ and the set of solutions of System~(\ref{S_0'z=1F3}), respectively. We consider the  functions  $\Phi_i(\Lambda)$ and the maxima $M_i$, $i=1,2,3,4$, given in (\ref{EQ:PHIF3}). From System~(\ref{S_0'z=1F3}), we observe that the functions $\Phi_1, \ldots, \Phi_4$ are constant for every  $\Lambda \in \mathcal{S}$, in particular it holds that
	$\Phi_1(\Lambda)=\Phi_2(\Lambda)=t$,  $\Phi_3(\Lambda)=1$ and $\Phi_4(\Lambda)=yt+1-t+0+0=1-t(1-y)$ for every $\Lambda \in \mathcal{S}$. So that $M_1=M_2=t$,  $M_3=1$, and $M_4=1-t(1-y)$. 
We consider two subcases: $t<1$; $t=1$.\\
	If $t<1$, then $M_{4}=yt+1-t>yt\geq0$; so that $M_{4}>0$  and we are in the first case  of Step~3 (i.e., $M_{n+1}>0$). Thus,  the \emph{procedure stops} and yields $z''=1$ as \emph{output}.  \\
	If $t=1$, then
	$M_{1}=M_{2}=M_{3}=1>0$ and $M_{4}=y$. Hence, we are	in the first case  of Step~3 (when $y>0$) or in the second case of Step~3 (when $y=0$). Thus,  the \emph{procedure stops} and yields $z''=1$ as \emph{output}. 	
\end{proof}
\begin{remark}\label{REM:FIG3POS}
	From Theorem \ref{THM:PROPF3},  we obtain  $z'>0$ if and only if $t(x+y-1)>0$. Moreover, we obtain  $z''<1$ if and only if $t(y-x)>0$. Moreover, it is easy to verify that 
	\[
	z'_{\{x,y,t\}}+	z''_{\{1-x,y,t\}}=1,
	\]
	where $z'_{\{x,y,t\}}$ and 	$z''_{\{1-x,y,t\}}$ are the lower bound and the upper bound of the two assessments $(x,y,z)$ and  $(1-x,y,z)$ on $((C|B),(A|B),B|(A\vee B))$, respectively.
\end{remark}
Based on Theorem~\ref{THM:PROPF3}, the next result presents the set of  coherent extensions of a given interval-valued probability assessment $\mathcal{I}=([x_1,x_2]\times [y_1,y_2]\times [t_1,t_2])\subseteq [0,1]^3$  on    $(C|B,A|B,B|(A\vee B))$ to the further conditional event
$C|A$.

\begin{theorem}\label{THM:PROPWTIVF3}
	Let $A,B,C$ be three logically independent events and $\mathcal{I}=([x_1,x_2]\times [y_1,y_2]\times [t_1,t_2])\subseteq [0,1]^3$  be an imprecise    assessment on  $(C|B,A|B,B|(A\vee B))$. Then, the set $\Sigma$ of the coherent extensions of $\mathcal{I}$ on $C|A$ is the interval  $[z^{*},z^{**}]$, where 
	\[
	\begin{array}{llll}
	z^* &=&
	\left\{
	\begin{array}{cl}
	0, &  \mbox{if } t_1(x_1+y_1-1)\leq 0,  \\~
	\displaystyle \frac{t_1(x_1+y_1-1)}{1-t_1(1-y_1)}\,
	, 
	&\mbox{if } t_1(x_1+y_1-1)>0, \;\;\ \text{ and }
	\end{array}
	\right.\vspace{0.05cm}
	\\  
	z^{**}&=&	\left\{
	\begin{array}{cl}
	\displaystyle
	1, &  \mbox{if }  t_1(y_1-x_2)\leq 0,\\	
	\displaystyle 1-\frac{t_1(y_1-x_2)}{1-t_1(1-y_1)}\,, &\mbox{if } t_1(y_1-x_2)>0.
	\end{array}
	\right.
	\end{array}
	\]	
\end{theorem}
\begin{proof}
	Since the set $[0,1]^3$ on $(C|B,A|B,B|(A\vee B))$ is totally coherent (Proposition~\ref{PROP:FIGUREIIIBIS}), it follows that 
	$\mathcal{I}$ is also totally coherent. 
	For every precise assessment $\P=(x,y,t)\in \I$, we denote by $[z'_{\P},z''_{\P}]$ the interval of the coherent extension of $\P$ on $C|A$, where $z'_{\P}$ and $z''_{\P}$ coincide with   $z'$ and $z''$, respectively, as	defined in Theorem \ref{THM:PROPF3}. 
	Then,  $\Sigma=\bigcup_{\P\in \I}[z_{\P}',z_{\P}'']=[z^*,z^{**}]$, 
	where $z^*=\inf_{\P\in \I}z_{\P}'$
	and $z^{**}=\sup_{\P\in \I}z_{\P}''$.
	\\
	Concerning the computation of $z^*$ 
	we distinguish the following  alternative cases: 
	$(i)$ $t_1(x_1+y_1-1)\leq 0$; $(ii)$ $t_1(x_1+y_1>  1)>0$. 
	Case $(i)$. By Theorem~\ref{THM:PROPF3}  it holds that $z'_{\P}=0$ for $\P=(x_1,y_1,t_1)$.
	Thus, $\{z'_{\P}:\P\in\I\}\supseteq\{0\}$ and hence $z^*=0$.\\	 
	Case $(ii)$. We note that
	the function  $t(x+y-1):[0,1]^3$ is  nondecreasing in the arguments  $x,y,t$. 
	Then, 
	$t(x+y-1)\geq t_1(x_1+y_1-1)>0$ for every $(x,y,t)\in \I$. Hence by Theorem~\ref{THM:PROPF3},
	$z'_{\P}=\frac{t(x+y-1)}{1-t(1-y)}$ for every $\P\in \mathcal{I}$. 
	Moreover, the function $\frac{t(x+y-1)}{1-t(1-y)}$  is nondecreasing in the arguments  $x,y,t$ over the restricted  domain $\I$; then, $\frac{t(x+y-1)}{1-t(1-y)}\geq \frac{t_1(x_1+y_1-1)}{1-t_1(1-y_1)}$. Thus,
	$
	z^*=\inf\{z_{\P}':\P\in\I\}=\inf\Big\{\frac{t(x+y-1)}{1-t(1-y)}:(x,y,z)\in\I\Big\}=\frac{t_1(x_1+y_1-1)}{1-t_1(1-y_1)}$. \\
	Concerning the computation of $z^{**}$ 
	we distinguish the following  alternative cases: 
	$(i)$ $t_1(y_1- x_2)\leq 0$; $(ii)$ $t_1(y_1- x_2)> 0$. 
	Case $(i)$. 
	By Theorem~\ref{THM:PROPF3}  it holds that $z''_{\P}=1$ for $\P=(x_2,y_1,t_1)\in \I$.
	Thus, 
	$\{z''_{\P}:\P\in\I\}\supseteq \{1\}$ and hence $z^{**}=1$.\\	 		
	Case $(ii)$. We observe that $t(y-x)\geq t_1(y-x)\geq t_1(y_1-x)\geq t_1(y_1-x_2)>0$ for every $(x,y,t)\in\I$. Then, the condition $t(y-x)>0$ is satisfied   for every $\P=(x,y,t)\in \mathcal{I}$ and hence by Theorem~\ref{THM:PROPF3},
	$z''_{\P}=1-\frac{t(y-x)}{1-t(1-y)}$ for every $\P\in \mathcal{I}$. 
	The function $1-\frac{t(y-x)}{1-t(1-y)}$  is nondecreasing in the argument  $x$ and it is nonincreasing in the arguments $y,t$ over the restricted  domain $\I$. 
	Thus, $1-\frac{t(y-x)}{1-t(1-y)}\leq 1-\frac{t(y-x_2)}{1-t(1-y)}\leq 1-\frac{t_1(y_1-x_2)}{1-t_1(1-y_1)}$ for every $(x,y,t)\in\I$. Then
	$
	z^{**}=\sup\{z_{\P}'':\P\in\I\}=\sup\Big\{1-\frac{t(y-x)}{1-t(1-y)}:(x,y,z)\in\I\Big\}=1-\frac{t_1(y_1-x_2)}{1-t_1(1-y_1)}$.
\end{proof}
\subsection{Traditionally valid syllogisms of Figure III.}
\label{SEC:FIGIIISYLLO}
In this section we consider the probabilistic interpretation of the traditionally valid  syllogisms of Figure III (Darapti, Datisi, Disamis,   Felapton, Ferison, and  Bocardo; see Table 
\ref{TAB:AristSyl}).  Like in Figure I and in Figure II, all  syllogisms of Figure III without  existential import assumptions are probabilistically non-informative. Indeed, by instantiating $S$, $M$, $P$ for $A$, $B$, $C$, respectively, in Proposition \ref{PROP:FIGUREIII},  we observe that  the imprecise assessment $[0,1]^3$ on $(P|M,S|M,P|S)$ is t-coherent. Thus,  for instance, 
from the premises $p(P|M)=1$ and $p(S|M)>0$ infer that every $p(P|S)\in[0,1]$ is coherent. This means that  Datisi 
(``\emph{Every $M$ is $P$, Some $M$ is $S$, therefore Some $S$ is $P$}'')
without  existential import assumption is not valid.
Therefore we add the conditional event  existential import assumption:  $p(M|(S \vee M))>0$ (see Definition \ref{DEF:CEI}). In what follows, we
construct  (s-)valid versions of the traditionally valid syllogisms  of Figure III,  by  suitable instantiations in Theorem~\ref{THM:PROPF3}.
\paragraph{Darapti} 
By instantiating $S,M,P$ in Theorem~\ref{THM:PROPF3} for $A,B,C$ with $x=1$, any $y=1$, and any $t>0$,
as $t(x+y-1)=t>0$, 
it follows that  $z'= \frac{t(x+y-1)}{1-t(1-y)}=t>0$. Concerning the upper bound $z''$,  as $t(y-x)=0$, it holds that $z''=1$. 
Then,  the set $\Sigma$    of coherent extensions on $P|S$ of the  imprecise assessment $\{1\}\times\{1\}\times(0,1]$ on  $(P|M,S|M,M|(S\vee M))$ is $\Sigma=\bigcup_{\{(x,y,t)\in\{1\}\times \{1\} \times (0,1]\}} [t,1]=\bigcup_{\{t \in (0,1]\}} [t,1]=(0,1]$.
 Thus, by Definition \ref{DEF:VALID},
\begin{equation}\label{EQ:DARAPTIIP}
\{1\}\times\{1\}\times(0,1] \mbox{ on } (P|M,S|M,M|(S\vee M)) \;\models_s\; (0,1] \mbox{ on } P|S. \end{equation}
In terms of probabilistic constraints, (\ref{EQ:DARAPTIIP}) can be expressed by
\begin{equation}\label{EQ:DARAPTI} 
(p(P|M)=1, p(S|M)=1, p(M|(S \vee M))>0) \,\models_s \, p(P|S)>0 \,,
\end{equation}
which is a  s-valid version of Darapti.
\paragraph{Datisi} 
By instantiating $S,M,P$ in Theorem~\ref{THM:PROPF3} for $A,B,C$ with $x=1$, any $y>0$, and any $t>0$,
 as $t(x+y-1)=ty>0$, 
it follows that  $z'= \frac{t(x+y-1)}{1-t(1-y)}=\frac{ty}{1-t(1-y)}>0$. Concerning the upper bound $z''$,  as $t(y-x)=t(y-1)\leq0$, it holds that $z''=1$. 
Then,  the set $\Sigma$    of coherent extensions on $P|S$ of the  imprecise assessment $\{1\}\times(0,1]\times(0,1]$ on  $(P|M,S|M,M|(S\vee M))$ is $\Sigma=\bigcup_{\{(x,y,t)\in\{1\}\times (0,1]\times(0,1]\}} [\frac{ty}{1-t(1-y)},1]$.
We now prove that $\Sigma=(0,1]$. Of course, $\Sigma \subseteq [0,1]$.  
Moreover, as for  $(y,t)\in(0,1]\times(0,1]$ it holds that  $\frac{ty}{1-t(1-y)}>0$, then  $0\notin \Sigma$ and hence  $\Sigma \subseteq (0,1]$. \emph{Vice versa}, let $z\in(0,1]$. By choosing any pair $(y,t)\in (0,1]\times(0,1] $ such that 
 $0<t\leq z$ and $y=1$, we obtain 
\[
 \frac{ty}{1-t(1-y)} =t\leq z \leq 1,
\]
which implies that $z\in \Sigma$. Thus, by Definition \ref{DEF:VALID},
\begin{equation}\label{EQ:DATISIIP}
\{1\}\times(0,1]\times(0,1] \mbox{ on } (P|M,S|M,M|(S\vee M)) \;\models_s\; (0,1] \mbox{ on } P|S. \end{equation}
In terms of probabilistic constraints, (\ref{EQ:DATISIIP}) can be expressed by
\begin{equation}\label{EQ:DATISI} 
(p(P|M)=1, p(S|M)>0, p(M|(S \vee M))>0) \,\models_s \, p(P|S)>0 \,,
\end{equation}
which is a  s-valid version of Datisi.
Therefore, inference (\ref{EQ:DATISI}) is a probabilistically informative version of Datisi.

\paragraph{Disamis} We instantiate $S,M,P$ in Theorem~\ref{THM:PROPF3} for $A,B,C$ with any $x>0$,  $y=1$, and any $t>0$. We observe that  the imprecise assessment $I=(0,1]\times\{1\}\times(0,1]$ on  $(P|M,S|M,M|(S\vee M))$ coincides with $I'\cup I''$, where $I'=\{1\}\times\{1\}\times(0,1]$ and $I''=(0,1)\times\{1\}\times(0,1]$ (notice that here $(0,1)$ denotes the open unit interval). Then,  
the set $\Sigma$    of coherent extensions on $P|S$ of the  imprecise assessment $I$ 
on $P|S$ 
 coincides with  $\Sigma'\cup\Sigma''$, where 
$\Sigma'$ and $\Sigma''$  are the sets of coherent extensions  of the  two assessments $I'$  and $I''$, respectively.
In case of $I'$ (which implies $x=1$), it holds that  $\Sigma'=\bigcup_{\{(x,y,t)\in\{1\}\times \{1\} \times (0,1]\}} [t,1]=(0,1]$ (see Darapti).
In case of  $I''$ (which implies $x>0$), as $t(x+y-1)=tx>0$, 
it follows that  $z'= \frac{t(x+y-1)}{1-t(1-y)}=tx>0$; 
 concerning the upper bound, as $t(y-x)=t(1-x)>0$, it holds that $z''=1-\frac{t(y-x)}{1-t(1-y)}=1-t(1-x)$.  
Then,  $\Sigma''=\bigcup_{\{(x,y,t)\in\{(0,1)\times \{1\} \times (0,1]\}} [tx,1-t(1-x)]=(0,1)$.
Hence, $\Sigma=\Sigma'\cup\Sigma''=(0,1]$.
 Thus, by Definition \ref{DEF:VALID},
\begin{equation}\label{EQ:DISAMISIP}
(0,1]\times\{1\}\times(0,1] \mbox{ on } (P|M,S|M,M|(S\vee M)) \;\models_s\; (0,1] \mbox{ on } P|S. \end{equation}
In terms of probabilistic constraints, (\ref{EQ:DISAMISIP}) can be expressed by
\begin{equation}\label{EQ:DISAMIS} 
(p(P|M)>0, p(S|M)=1, p(M|(S \vee M))>0) \,\models_s \, p(P|S)>0 \,,
\end{equation}
which is a  s-valid version of Disamis (``\emph{Some $M$ is $P$,
	Every $M$ is  $S$, therefore Some $S$ is  $P$}''). Notice that Darapti also follows from Disamis and from Datisi by strengthening the premise.

\paragraph{Felapton}
By instantiating $S,M,P$ in Theorem~\ref{THM:PROPF3} for $A,B,C$ with $x=0$, any $y=1$, and any $t>0$,
as $t(x+y-1)=0$, 
it follows that  $z'=0$. Concerning the upper bound $z''$,  as $t(y-x)=t>0$, it holds that $z''=1-\frac{t(y-x)}{1-t(1-y)}=1-t$. 
Then,  the set $\Sigma$    of coherent extensions on $P|S$ of the  imprecise assessment $\{0\}\times\{1\}\times(0,1]$ on  $(P|M,S|M,M|(S\vee M))$ is $\Sigma=\bigcup_{\{(x,y,t)\in\{0\}\times\{1\}\times(0,1]\}} [0,1-t]$.
Equivalently,  the set $\no{\Sigma}$  of coherent extensions on $\no{P}|S$  is $\no{\Sigma}=\bigcup_{\{(x,y,t)\in\{0\}\times\{1\}\times(0,1]\}} [t,1]=\bigcup_{\{t\in (0,1]\}} [t,1]=(0,1]$.
Thus, by Definition \ref{DEF:VALID},
\begin{equation}\label{EQ:FELAPTONIP}
\{0\}\times\{1\}\times(0,1] \mbox{ on } (P|M,S|M,M|(S\vee M)) \;\models_s\; (0,1] \mbox{ on } \no{P}|S. \end{equation}
In terms of probabilistic constraints, (\ref{EQ:FELAPTONIP}) can be expressed by
\begin{equation}\label{EQ:FELAPTON} 
(p(P|M)=0, p(S|M)=1, p(M|(S \vee M))>0) \,\models_s \, p(\no{P}|S)>0 \,,
\end{equation}
which is a  s-valid version  of Felapton.  Notice that Felapton is equivalent to Darapti, because 
(\ref{EQ:FELAPTON}) 
is equivalent to 
(\ref{EQ:DARAPTI})
when  $P$ is replaced by  $\no{P}$ (and the  probabilities are adjusted accordingly).

\paragraph{Ferison}
By instantiating $S,M,P$ in Theorem~\ref{THM:PROPF3} for $A,B,C$ with $x=0$, any $y>0$, and any $t>0$,
as $t(x+y-1)=t(y-1)\leq 0$, 
it follows that  $z'=0$. Concerning the upper bound $z''$,  as $t(y-x)=ty>0$, it holds that $z''=1-\frac{t(y-x)}{1-t(1-y)}=1-\frac{ty}{1-t(1-y)}$. 
Then,  the set $\Sigma$    of coherent extensions on $P|S$ of the  imprecise assessment $\{0\}\times(0,1]\times(0,1]$ on  $(P|M,S|M,M|(S\vee M))$ is $\Sigma=\bigcup_{\{(x,y,t)\in \{0\}\times(0,1]\times(0,1]\}} [0,1-\frac{ty}{1-t(1-y)}]$.
Equivalently, as $p(\no{P}|S)=1-p(P|S)$, 
the set of  coherent extensions on $\no{P}|S$, denoted by $\no{\Sigma}$, of the  imprecise assessment $\{0\}\times(0,1]\times(0,1]$ on  $(P|M,S|M,M|(S\vee M))$
is $\no{\Sigma}=\bigcup_{\{(x,y,t)\in \{0\}\times(0,1]\times(0,1]\}} [\frac{ty}{1-t(1-y)},1]=(0,1]$.
Thus, by Definition~\ref{DEF:VALID},
\begin{equation}\label{EQ:FERISONIP}
\{0\}\times(0,1]\times(0,1] \mbox{ on } (P|M,S|M,M|(S\vee M)) \;\models_s\; (0,1] \mbox{ on } \no{P}|S. \end{equation}
In terms of probabilistic constraints, (\ref{EQ:FERISONIP}) can be expressed by
\begin{equation}\label{EQ:FERISON} 
(p(P|M)=0, p(S|M)>0, p(M|(S \vee M))>0) \,\models_s \, p(\no{P}|S)>0 \,,
\end{equation}
which is a  s-valid version of Ferison. 
Notice that 
 Ferison (\ref{EQ:FERISON})   is equivalent to
  Datisi (\ref{EQ:DATISI}),
when  $P$ is replaced by  $\no{P}$.

\paragraph{Bocardo}
 We instantiate $S,M,P$ in Theorem~\ref{THM:PROPF3} for $A,B,C$ with any $x<1$,  $y=1$, and any $t>0$. 
 We observe that  the imprecise assessment $I=[0,1)\times\{1\}\times(0,1]$ on  $(P|M,S|M,M|(S\vee M))$ coincides with $I'\cup I''$, where $I'=\{0\}\times\{1\}\times(0,1]$ and $I''=(0,1)\times\{1\}\times(0,1]$ (notice that here $(0,1)$ denotes the open unit interval). Then,  
 the set $\Sigma$    of coherent extensions on $P|S$ of the  imprecise assessment $I$ 
 on $P|S$ 
 coincides with  $\Sigma'\cup\Sigma''$, where 
 $\Sigma'$ and $\Sigma''$  are the sets of coherent extensions  of the  two assessments $I'$  and $I''$, respectively. In case of $I'$ (which implies $x=0$), it holds that
 $\Sigma'=\bigcup_{\{(x,y,t)\in\{0\}\times \{1\} \times (0,1]\}} [0,1-t]=[0,1)$ (see the set $\Sigma$ in Felapton).
In case of $I''$ (which implies $0<x<1$), it holds that  
 $\Sigma''=(0,1)$ (see the set  $\Sigma''$ in Disamis). 
 Hence, $\Sigma=\Sigma'\cup\Sigma''=[0,1)$.
 Thus, by Definition \ref{DEF:VALID},
\begin{equation}\label{EQ:BOCARDOIP}
[0,1)\times\{1\}\times(0,1] \mbox{ on } (P|M,S|M,M|(S\vee M)) \;\models_s\; [0,1) \mbox{ on } P|S. \end{equation}
In terms of probabilistic constraints, (\ref{EQ:BOCARDOIP}) can be expressed by
\begin{equation*}
(p(P|M)<1, p(S|M)=1, p(M|(S \vee M))>0) \,\models_s \, p(P|S)<1 \,,
\end{equation*}
which is equivalent to 
\begin{equation}\label{EQ:BOCARDO} 
(p(\no{P}|M)>0, p(S|M)=1, p(M|(S \vee M))>0) \,\models_s \, p(\no{P}|S)>0 \,.
\end{equation}
Formula (\ref{EQ:BOCARDO})
 is a  s-valid version of Bocardo  (\emph{``Some  $M$ is not $P$, Every $M$ is  $S$, therefore Some $S$ is not $P$''}).
 We observe that 
 Bocardo (\ref{EQ:BOCARDO})   is equivalent to
  Disamis (\ref{EQ:DISAMIS}),
when  $P$ is replaced by  $\no{P}$.

\begin{remark}
Notice that, traditionally, the conclusions of logically valid Aristotelian syllogisms of Figure III are neither  in the form of sentence type  A (\emph{every}) nor of E (\emph{no}).
	In terms of our probability semantics, 
	we study  which  assessments $(x,y,t)$ on $(P|M,S|M,S|(S\vee M))$  propagate to
	$z'=z''=p(P|S)=1$  in order to validate A in the conclusion. 
	According to Theorem~\ref{THM:PROPF3}, $(x,y,t)\in[0,1]^3$ propagates to $z'=z''=1$ if and only if 
	\[ \small
	\begin{cases}
	(x,y,t)\in[0,1]^3, \\
	t(x+y-1)>0,\\
	z'=\frac{t(x+y-1)}{1-t(1-y)}=1,\\
	t(y-x)\leq 0,
	\end{cases}
	\Longleftrightarrow
	\begin{cases}
	(x,y,t)\in[0,1]^3,\\
	1+yt-t>0,\\
	tx=1, 
	ty\leq 1,
	\end{cases}
	\Longleftrightarrow
	\begin{cases}
	x=1, \\
	0<y\leq 1,\\
	t=1.
	\end{cases}
	\]	 
	Then, $z'=z''=1$ if and only if $(x,y,t)=(1,y,1)$, with $0<y\leq 1$.
	However, for the syllogisms it would be too strong to require $t=1$ as an existential import assumption, we only require  that $t>0$. 	 
	Similarly,  in order to validate E in the conclusion, it can be shown that 
	assessments $(x,y,t)$ on $(P|M,S|M,S|(S\vee M))$  propagate to the conclusion
	$z'=z''=p(P|S)=0$ if and only if  $(x,y,t)=(0,y,1)$, with $0<y\leq 1$.
	Therefore, if $t$ is just positive  neither A nor E can be validate within in our probability semantics of Figure III.
\end{remark}

\section{Applications to nonmonotonic reasoning.}
\label{SEC:NM}
We recall that the \emph{default} ${H\normally E}$ denotes the sentence ``\emph{$E$ is a plausible consequence of  $H$}''  (see, e.g., \cite{kraus90}). Moreover, the  \emph{negated  default} $H\nnormally E$ denotes the sentence ``\emph{it is not the case, that: $E$ is a plausible consequence of  $H$}''. 
Based on Definition~8 in \cite{gilio16}, we interpret the default $H\normally E$ by  
the probability assessment 
$p(E|H)=1$;  while the negated default $H\nnormally E$ is interpreted by the imprecise probability assessment
$p(E|H)<1$. Thus,
as the  probability assessment $p(E|H)>0$ is equivalent to ${p(\widebar{E}|H)<1}$, the negated default $H\nnormally \widebar{E}$ is also interpreted by $p(E|H)>0$.
Then, the basic syllogistic sentence types (see Table \ref{Table:BS}) can be interpreted in terms of defaults or negated defaults as follows:
\begin{itemize}
\item[(A)] $S\normally P$ (\emph{Every $S$ is $P$}, $p(P|S)=1$);
\item[(E)] $S\normally \no{P}$
(\emph{No $S$ is $P$}, $p(\no{P}|S)=1$);
\item[(I)] $S\nnormally \no{P}$
(\emph{Some $S$ is $P$}, $p(P|S)>0$);
\item[(O)] $S\nnormally P$
(\emph{Some $S$ is  not $P$}, $p(\no{P}|S)>0$).
\end{itemize}
For example, recall the
probabilistic modus Barbara
(\ref{EQ:BARBARA}), which is strictly valid and thus valid, can be expressed in terms of defaults and negated defaults as follows:
$(M\normally P, S\normally M, 
(S\vee M) \nnormally \no{S})\,\models\, S\normally P$. As pointed out in \cite{gilio16}  the default version of Barbara amounts to  the well-known  inference rule Weak Transitivity.
We recall that Weak Transitivity is valid in the nonmonotonic  System R (\cite{Lehmann92}, i.e., System P (\cite{kraus90}) plus Rational Monotonicity), because it  is equivalent to Rational Monotonicity  
(\cite[Theorem~2.1]{freund1991}). For other nonmonotonic versions of transitivity  see \cite{bezzazi96,bezzazi97}.
We present the   default versions of the (logically valid) syllogisms   of Figures I, II, and III  in Table~\ref{TAB:AristSylDef}.
These versions,  which involve  defaults and negated defaults, are valid in our approach and
can serve as  inference rules for nonmonotonic reasoning.

Moreover, we observe that some syllogisms can be expressed in defaults only without using negated defaults. For example, if the conditional event existential import of Barbara is strengthened by $p(S|(S\vee M))=1$, we obtain the following  valid default rule:
\begin{equation}\label{EQ:BARBARADEFAULTS}
 (M\normally P, S\normally M, 
(S\vee M) \normally S)\,\models_s\, S\normally P.
\end{equation}
Note that inference (\ref{EQ:BARBARADEFAULTS}) still satisfies  AAA of Figure I.  In probabilistic terms inference (\ref{EQ:BARBARADEFAULTS})
 means that the premises   p-entails the conclusion  (see Section 10.2 in \cite{GiSa19}),
 i.e.,
  \begin{equation}\label{EQ:BARBARAP}
 (p(P|M)=1, p(M|S)=1, 
 p(S|(S\vee M) )=1)\,\models_s\, p(P|S)=1.
 \end{equation}
 In general, the probability propagation rules for the three figures can be used to generate new syllogisms. For example, from the probability propagation rule of Figure III  (Theorem \ref{THM:PROPF3}) we obtain 
 a valid syllogism of type AAA with  the (stronger) existential import $p(M|(S\vee M))=1$, which is in terms of defaults:
\begin{equation}\label{EQ:AAAFigIII}
(M\normally P,
M\normally S, 
(S\vee M) \normally M)\,\models_s\, S\normally P.	\end{equation}
Equations (\ref{EQ:BARBARADEFAULTS}) and (\ref{EQ:AAAFigIII}) 
 are p-valid inference rules for nonmonotonic reasoning and constitute  syllogisms which are beyond traditional Aristotelian syllogisms (since, traditionally, AAA does not describe a valid syllogism of Figure III).

 The procedure  of replacing negated defaults by defaults, for obtaining inference (\ref{EQ:BARBARADEFAULTS}),  can also yield new syllogisms. 
\begin{table}[!h]
	\begin{minipage}{\textwidth}\centering 
		\begin{tabular}{lll}      \hline
			\multicolumn{3}{c}{Figure I }\\
			AAA & Barbara & $(M\normally P, S\normally M, 
			(S\vee M) \nnormally \no{S})\,\models_s\, S\normally P.$\\
			AAI & Barbari &
			$(M\normally P, S\normally M, 
			(S\vee M) \nnormally \no{S})\,\models\, S\nnormally \no{P}.$\\
			AII & Darii &
			$(M\normally P, S\nnormally \no{M}, 
			(S\vee M) \nnormally \no{S})\,\models_s\, S\nnormally \no{P}.$\\
				EAE & Celarent &$(M\normally \no{P}, S\normally M, 
				(S\vee M) \nnormally \no{S})\,\models_s\, S\normally\no{P}.$\\
			EAO & Celaront &						$(M\normally \no{P}, S\normally M, 
			(S\vee M) \nnormally \no{S})\,\models\, S\nnormally P.$\\
			EIO & Ferio &$(M\normally \no{P}, S\nnormally \no{M}, 
			(S\vee M) \nnormally \no{S})\,\models_s\, S\nnormally P.$\\
			\hline			
			\multicolumn{3}{c}{Figure II }\\
				AEE & Camestres &
			$(P\normally M,
			 S\normally \no{M}, 
			(S\vee P) \nnormally \no{S})\,\models_s\, S\normally \no{P}.$\\
			AEO & Camestrop \, &	$(P\normally M,
			S\normally \no{M}, 
			(S\vee P) \nnormally \no{S})\,\models\, S\nnormally P.$\\
			AOO & Baroco \, &	$(P\normally M,
			S\nnormally M, 
			(S\vee P) \nnormally \no{S})\,\models_s\, S\nnormally P.$\\
			EAE & Cesare &	$(P\normally \no{M},
			S\normally M, 
			(S\vee P) \nnormally \no{S})\,\models_s\, S\normally \no{P}.$\\
			EAO & Cesaro &
			$(P\normally \no{M},
			S\normally M, 
			(S\vee P) \nnormally \no{S})\,\models\, S\nnormally P.$\\
			EIO & Festino &
			$(P\normally \no{M},
S\nnormally \no{M}, 
(S\vee P) \nnormally \no{S})\,\models_s\, S\nnormally P.$\\
			\hline
			
			\multicolumn{3}{c}{Figure III}\\
			AAI& Darapti &
$(M\normally P,
M\normally S, 
(S\vee M) \nnormally \no{M})\,\models_s\, S\nnormally \no{P}.$\\						
			AII & Datisi &			$(M\normally P,
			M\nnormally \no{S}, 
			(S\vee M) \nnormally \no{M})\,\models_s\, S\nnormally \no{P}.$\\
	IAI & Disamis &$(M\nnormally \no{P},
	M\normally S, 
	(S\vee M) \nnormally \no{M})\,\models_s\, S\nnormally \no{P}.$\\
			EAO & Felapton &
$(M\normally \no{P},
M\normally S, 
(S\vee M) \nnormally \no{M})\,\models_s\, S\nnormally P.$\\			
			EIO & Ferison &
$(M\normally \no{P},
M\nnormally \no{S}, 
(S\vee M) \nnormally \no{M})\,\models_s\, S\nnormally P.$\\			
			OAO\,\, & Bocardo &$(M\nnormally P,
			M\normally S, 
			(S\vee M) \nnormally \no{M})\,\models_s\, S\nnormally P.$\\ \hline
		\end{tabular}
	\end{minipage}
	\caption{
		Traditional (logically valid) Aristotelian syllogisms of Figure I, II, and III (see Table~\ref{TAB:AristSyl}) in terms of   defaults and negated defaults, under the conditional event existential import assumption.}
	\label{TAB:AristSylDef}
\end{table}

\section{Conversion and reduction.}
\label{SEC:CONV}
The most prominent methods of proof in Aristotelian syllogistics are \emph{conversion}, \emph{reductio} (by conversion), and \emph{reductio ad impossibile} (by the compound law of transposition; see, e.g., \cite{lukasiewicz57,sep-aristotle-logic}). In this section we give some probabilistic results  to show that \emph{conversion} and  \emph{reductio} (by conversion) do not hold in our approach (Section~\ref{SEC:RedByConv}). Then we show to what extent  \emph{reductio ad impossibile} can be applied within our approach: we will observe that by  the application of the compound law of transposition to syllogisms of Figure~II and Figure~III the syllogisms can be reduced  to Figure I and that they are hence \emph{valid}. However, this method does not allow for distinguishing between  \emph{valid} and \emph{s-valid} syllogisms and hence  \emph{reductio ad impossibile} is not \emph{s-validity} preserving (Section~\ref{SEC:RedByImp}).

\subsection{Reductio by conversion.}\label{SEC:RedByConv}

According to Aristotle,   three rules 
 of \emph{conversion} are sound (see, e.g., \cite{lukasiewicz57,sep-aristotle-logic}). Conversion means that the  term position can be interchanged in sentence types I and E (i.e., \emph{some $S$ is $P$}  logically implies   \emph{some $P$ is $S$} and \emph{no $S$ is $P$}  logically implies   \emph{no $P$ is $S$}, respectively) and that \emph{every $S$ is $P$} logically implies \emph{some $P$ is  $S$}. However, the assessment $p(P|S)$ does not constrain $p(S|P)$. Indeed,  as we now show
 in Proposition~\ref{PROP:ASYMMETRY}, 
 the assessment $(x,y)$ on $(P|S,S|P)$ is coherent for every $(x,y)\in[0,1]^2$. Therefore,  none of these three rules of conversion hold in our approach. 
\begin{proposition}
[Asymmetry of term order]
\label{PROP:ASYMMETRY}
   Let  $P,S$ be two logically independent events.  The imprecise assessment 
   $[0,1]^2$ on $\F=(P|S,S|P)$ is 
   t-coherent.
\end{proposition}
\begin{proof}
   Let  $P,S$ be two logically independent events.  We show that the imprecise assessment 
   $[0,1]^2$ on $\F=(P|S,S|P)$ is totally coherent, by showing that every precise assessment
   $\P=(x,y)\in[0,1]^2$ on $\F$ is coherent. Let $\P=(x,y)\in[0,1]^2$ be a precise assessment on $\F$. Then,
the constituents $C_h$ and  the  points $Q_h$ associated with   $(\F,\P)$ are
\[
C_1=SP, \;C_2=S\no{P}, \;C_3=\no{S}P,\; C_0=(x,y),
\]
and
\[
Q_1=(1,1), \; Q_2=(0,y), \;Q_3=(x,0),\; Q_0=(x,y)=\P.
\]
 We observe that $C_1\vee C_2\vee C_3=S\vee P$. By Theorem \ref{COER-P0}, coherence of  $\P$ on $\F$ requires that the following system
\[\begin{array}{l}
(\Systemsigma)  \hspace{1 cm}
\mathcal{P}=\sum_{h=1}^3 \lambda_hQ_h,\;
\sum_{h=1}^3 \lambda_h=1,\; \lambda_h\geq 0,\, h=1,\ldots,3,
\end{array}
\]
 or equivalently  
\[
\left\{\begin{array}{l}
\lambda_1+x\lambda_3=x,\;\;
\lambda_1+y\lambda_2=y,\\
\lambda_1+\lambda_2+\lambda_3=1,\;\;
\lambda_h\geq 0,\; h=1,2,3,
\end{array}
\right.
\]
is solvable. In geometrical terms, this  means that  
the condition $\P \in \mathfrak{I}$ is satisfied, where   $\mathfrak{I}$ is the convex hull of $Q_1, Q_2, Q_3$. 
We distinguish three cases: $(i)$ $x\neq 0$ and $y\neq 0$; $(ii)$ $x=0$; $(iii)$  $y=0$.\\
Case $(i)$.  We observe that 
$\P=\frac{xy}{x+y-xy}Q_1+
\frac{y(1-x)}{x+y-xy}Q_2+\frac{x(1-y)}{x+y-xy}Q_3$, indeed it holds that 
\[
\begin{array}{ll}
\frac{xy}{x+y-xy}(1,1)+
\frac{y(1-x)}{x+y-xy}(0,y)+\frac{x(1-y)}{x+y-xy}(x,0)=
\left(\frac{x^2+xy-x^2y}{x+y-xy},\frac{xy+y^2-xy^2}{x+y-xy}  \right)=\left(\frac{x(x+y-xy)}{x+y-xy},\frac{y(x+y-xy)}{x+y-xy}  \right)=(x,y).
\end{array}
\]
Thus,  system 	$(\Systemsigma)  $ is solvable and a solution is
$\Lambda=(\lambda_1,\lambda_2,\lambda_3)=(\frac{xy}{x+y-xy}, \frac{y(1-x)}{x+y-xy},\frac{x(1-y)}{x+y-xy})$. 
From  (\ref{EQ:PHI}) we obtain that
\[
\Phi_{1}(\Lambda)=\sum_{h:C_h\subseteq S}\lambda_h=\lambda_1+\lambda_2=\frac{y}{x+y-xy}>0,\;\;
\Phi_{2}(\Lambda)=\sum_{h:C_h\subseteq P}\lambda_h=
\lambda_1+\lambda_3=\frac{x}{x+y-xy}>0.
\]
Let $\mathcal{S}'=\{(\frac{xy}{x+y-xy},\frac{y(1-x)}{x+y-xy},\frac{x(1-y)}{x+y-xy})\}$ denote a subset of the set $\mathcal{S}$ of all solutions of $(\Systemsigma) $.  
Then, ${M_1'=\max\{\Phi_{1}:\Lambda\in\mathcal{S}'\}>0}$ and ${M_2'=\max\{\Phi_{2}:\Lambda\in\mathcal{S}'\}>0}$
and hence $I_0'=\emptyset $ (as defined in (\ref{EQ:I0'})).
 By Theorem \ref{COER-P0'},
 as $(\Systemsigma) $ is solvable and $I_0'=\emptyset $, the assessment $(x,y)\in\,]0,1]^2$ is coherent .\\
Case $(ii)$. In this case, as $x=0$, it holds that $\P=(0,y)=Q_2$. Thus,  system 	$(\Systemsigma)  $ is solvable and a solution is
$\Lambda=(\lambda_1,\lambda_2,\lambda_3)=(0,1,0)$. 
From  (\ref{EQ:PHI}) we obtain that
$
\Phi_{1}(\Lambda)=\sum_{h:C_h\subseteq S}\lambda_h=\lambda_1+\lambda_2=1$ and
$
\Phi_{2}(\Lambda)=\sum_{h:C_h\subseteq P}\lambda_h=
\lambda_1+\lambda_3=0$. Let $\mathcal{S}'=\{(0,1,0)\}$ denote a subset of the set $\mathcal{S}$ of all solutions of $(\Systemsigma) $.  
Then, $M_1'>0$ and $M_2'=0$
and hence $I_0'=\{2\}$ (as defined in (\ref{EQ:I0'})).
We recall that the sub-assessment $\P_0'=(y)$ on $\F_0'=\{S|P\}$ is  coherent for every $y\in[0,1]$. Then, 
 by  Theorem \ref{COER-P0'} the assessment  $(0,y)$ on $\F$ is coherent for every $y\in[0,1]$.
 Then, every assessment $(x,y)\in\{0\}\times[0,1]$ is coherent.
\\
Case $(iii)$. In this case, as $y=0$, it holds that $\P=(x,0)=Q_3$. Thus,  system 	$(\Systemsigma)  $ is solvable and a solution is
$\Lambda=(\lambda_1,\lambda_2,\lambda_3)=(0,0,1)$. 
From  (\ref{EQ:PHI}) we obtain that
$
\Phi_{1}(\Lambda)=\sum_{h:C_h\subseteq S}\lambda_h=\lambda_1+\lambda_2=0$ and 
$
\Phi_{2}(\Lambda)=\sum_{h:C_h\subseteq P}\lambda_h=
\lambda_1+\lambda_3=1$. Let $\mathcal{S}'=\{(0,0,1)\}$ denote a subset of the set $\mathcal{S}$ of all solutions of $(\Systemsigma) $.  
Then, $M_1'=0$ and $M_2'>0$
and hence $I_0'=\{1\}$ (as defined in (\ref{EQ:I0'})).
We recall that the sub-assessment $\P_0'=(x)$ on $\F_0'=\{P|S\}$ is  coherent for every $x\in[0,1]$. Then, 
 by  Theorem \ref{COER-P0'} the assessment  $(x,0)$ on $\F$ is coherent for every $x\in[0,1]$.
 Then, every assessment $(x,y)\in[0,1]\times\{0\}$ is coherent.
 \\
Therefore, every assessment $(x,y)\in[0,1]^2$ is coherent and hence the imprecise assessment $[0,1]^2$ on $\mathcal{F}$ is t-coherent.
\end{proof} 
 Moreover, Aristotle also proposed methods of {\em reduction} to prove validity. The method of reduction by conversion consists in ``reducing'' all  syllogisms to ``perfect'' syllogisms of Figure~1.  Only syllogisms of Figure~1 are perfect because the ``transitivity of the connexion between the terms [is] obvious at a glance'' (\cite[p. 73]{kneale84}): perfect syllogisms can be seen as self-evident without requiring further proof (for a discussion of ``perfect" see, e.g., \cite{ebert15}). More specifically, Aristotle's full program consists in showing validity by reduction to  \emph{Barbara} and \emph{Celarent}. Since this method of reduction requires conversion (\cite[p. 236]{kneale84}), reduction is also not valid in our approach. 

 In the next two remarks we observe that the conditional event existential import of Figure II  does not follow from any syllogism of Figure I (Remark~\ref{REM:FIGIINORED}) and \emph{vice versa} (Remark~\ref{REM:FIGIINOREReversed}). More specifically, assuming any degrees of belief in the premises of syllogisms of Figure~I (Figure~II, respectively) does not imply a positive degree of belief in the conditional event existential import of Figure~II, i.e., $p(S|(S\vee P)>0$ (Figure~I, i.e., $p(S|(S\vee M)>0$, respectively).


\begin{remark}\label{REM:FIGIINORED}
    Let $\P=(x,y,t,0)$, with $(x,y,t)\in[0,1]^3$, be a probability assessment on $\mathcal{F}=(P|M,M|S,S|(S\vee M),S|(S\vee P))$, where $S,M$, and $P$ are three logically independent events. We show that $\P=(x,y,t,0)$ on $\F$ is coherent for every $(x,y,t)\in[0,1]^3$.
	The constituents $C_h$ and  the  points $Q_h$ associated with   $(\F,\P)$ are
	given in Table \ref{TAB:TABLE_FIGUREIIRED}. 
	\begin{table}[!h]
		\begin{minipage}{\textwidth}
			\caption{Constituents $C_h$ and  points $Q_h$ associated with  the probability  assessment    $\mathcal{P}=(x,y,t,0)$  on 
				$\F=(P|M,M|S,S|(S\vee M),S|(S\vee P))$.
			}
			\label{TAB:TABLE_FIGUREIIRED}
		\centering 	\begin{tabular}{llll}\hline
		& $C_h$            & $Q_h$                          &  \\
		\hline
		$C_1$ \, & $SMP$         \,  & $(1,1,1,1)$ \,                   & $Q_1$   \\
		$C_2$ & $SM\no{P}$      & $(0,1,1,1)$                    & $Q_2$   \\
		$C_3$ & $S\no{M}P$       & $(x,0,1,1)$                    & $Q_3$   \\
		$C_4$ & $S\no{M}\,\no{P}$      & $(x,0,1,1)$                    & $Q_4$   \\
		$C_5$ & $\no{S}MP$ & $(1,y,0,0)$                    & $Q_5$   \\
		$C_6$ & $\no{S}M\no{P}$  & $(0,y,0,0)$                  & $Q_6$   \\  
		$C_7$ & $\no{S}\,\no{M}P$  & $(x,y,t,0)$                  & $Q_7$   \\
		$C_0$ & $\no{S}\,\no{M} \, \no{P}$   &$(x,y,t,0)$                       & $Q_0=\mathcal{P} $  \\\hline
	\end{tabular}
		\end{minipage}
	\end{table}
	By Theorem \ref{COER-P0}, coherence of  $\P=(x,y,t,0)$ on $\F$ requires that the following system is solvable
	\[\begin{array}{l}
	(\Systemsigma)  \hspace{1 cm}
	\mathcal{P}=\sum_{h=1}^7 \lambda_hQ_h,\;
	\sum_{h=1}^7 \lambda_h=1,\; \lambda_h\geq 0,\, h=1,\ldots,7,
	\end{array}
	\]
In geometrical terms, this  means that  
the condition $\P \in \mathfrak{I}$ is satisfied, where   $\mathfrak{I}$ is the convex hull of $Q_1, \ldots, Q_7$. We observe that 
$\P=Q_6$.  Thus,  system 	$(\Systemsigma)  $ is solvable and a solution is
$\Lambda=(\lambda_1,\ldots,\lambda_7)=(0,0,0,0,0,0,1)$. 
From  (\ref{EQ:PHI}) we obtain that
$\Phi_{1}(\Lambda)=\sum_{h:C_h\subseteq M}\lambda_h=\lambda_1+\lambda_2+\lambda_5+\lambda_6=0$,
$\Phi_{2}(\Lambda)=\sum_{h:C_h\subseteq S}\lambda_h=\lambda_1+\lambda_2+\lambda_3+\lambda_4=0$,
$
\Phi_{3}(\Lambda)=\sum_{h:C_h\subseteq S\vee M}\lambda_h=\lambda_1+\lambda_2+\lambda_3+\lambda_4+\lambda_5+\lambda_6=0$,
 $\Phi_{4}(\Lambda)=\sum_{h:C_h\subseteq S\vee P}\lambda_h=
\lambda_1+\lambda_2+\lambda_3+\lambda_4+\lambda_5+\lambda_7=1$. Let $\mathcal{S}'=\{(0,0,0,0,0,0,1)\}$ denote a subset of the set $\mathcal{S}$ of all solutions of $(\Systemsigma) $.   Then, $M_1'=M_2'=M_3'=0$, $M_4'=1$ and hence $I_0'=\{1,2,3\}$ (as defined in (\ref{EQ:I0'})).
 By Theorem \ref{COER-P0'},
 as $(\Systemsigma) $ is solvable and $I_0'=\{1,2,3\}$,
  it is sufficient to check the coherence of 
 the sub-assessment  $\P_0'=(x,y,t)$ on $\F_0'=
 (P|M,M|S,S|(S\vee M))$ in order to check the coherence of $(x,y,t,0)$ on $\F$. From   Proposition~\ref{PROP:FIGUREIBIS}, by replacing $A,B,C$ with $S,M,P$, respectively, it holds that  $(x,y,t)$ on $\F_0'=(P|M,M|S,S|(S\vee M))$ is coherent for every $(x,y,t)\in[0,1]^3$. Therefore, $(x,y,t,0)$ on $\F$ is coherent for every $(x,y,t)\in[0,1]^3$.
 \end{remark}
\begin{remark}\label{REM:FIGIINOREReversed}
    Let $\P=(x,y,t,0)$, with $(x,y,t)\in[0,1]^3$, be a probability assessment on $\mathcal{F}=(M|P,\no{M}|S,S|(P\vee S),S|(S\vee M))$, where $S,M$, and $P$ are three logically independent events. We show that $\P=(x,y,t,0)$ on $\F$ is coherent for every $(x,y,t)\in[0,1]^3$.
	The constituents $C_h$ and  the  points $Q_h$ associated with   $(\F,\P)$ are
	given in Table \ref{TAB:TABLE_FIGUREIIREDREV}. 
	\begin{table}[!h]
		\begin{minipage}{\textwidth}
			\caption{Constituents $C_h$ and  points $Q_h$ associated with  the probability  assessment    $\mathcal{P}=(x,y,t,0)$  on 
				$\F=(M|P,\no{M}|S,S|(S\vee P),S|(S\vee M))$.
			}
			\label{TAB:TABLE_FIGUREIIREDREV}
		\centering 	\begin{tabular}{llll}\hline
		& $C_h$            & $Q_h$                          &  \\
		\hline
		$C_1$ \, & $SMP$         \,  & $(1,0,1,1)$ \,                   & $Q_1$   \\
		$C_2$ & $SM\no{P}$      & $(x,0,1,1)$                    & $Q_2$   \\
		$C_3$ & $S\no{M}P$       & $(0,1,1,1)$                    & $Q_3$   \\
		$C_4$ & $S\no{M}\,\no{P}$      & $(x,1,1,1)$                    & $Q_4$   \\
		$C_5$ & $\no{S}MP$ & $(1,y,0,0)$                    & $Q_5$   \\
		$C_6$ & $\no{S}M\no{P}$  & $(x,y,t,0)$                  & $Q_6$   \\  
		$C_7$ & $\no{S}\,\no{M}P$  & $(0,y,0,0)$                  & $Q_7$   \\
		$C_0$ & $\no{S}\,\no{M} \, \no{P}$   &$(x,y,t,0)$                       & $Q_0=\mathcal{P} $  \\\hline
	\end{tabular}
		\end{minipage}
	\end{table}
	By Theorem \ref{COER-P0}, coherence of  $\P=(x,y,t,0)$ on $\F$ requires that the following system is solvable
	\[\begin{array}{l}
	(\Systemsigma)  \hspace{1 cm}
	\mathcal{P}=\sum_{h=1}^7 \lambda_hQ_h,\;
	\sum_{h=1}^7 \lambda_h=1,\; \lambda_h\geq 0,\, h=1,\ldots,7,
	\end{array}
	\]
In geometrical terms, this  means that  
the condition $\P \in \mathfrak{I}$ is satisfied, where   $\mathfrak{I}$ is the convex hull of $Q_1, \ldots, Q_7$. We observe that 
$\P=Q_6$.  Thus,  system 	$(\Systemsigma)  $ is solvable and a solution is
$\Lambda=(\lambda_1,\ldots,\lambda_7)=(0,0,0,0,0,1,0)$. 
From  (\ref{EQ:PHI}) we obtain that
$\Phi_{1}(\Lambda)=\sum_{h:C_h\subseteq P}\lambda_h=\lambda_1+\lambda_3+\lambda_5+\lambda_7=0$, 
 $\Phi_{2}(\Lambda)=\sum_{h:C_h\subseteq S}\lambda_h=\lambda_1+\lambda_2+\lambda_3+\lambda_4$,
 $\Phi_{3}(\Lambda)=\sum_{h:C_h\subseteq S\vee P}\lambda_h=
\lambda_1+\lambda_2+\lambda_3+\lambda_4+\lambda_5+\lambda_7=0$, and 
 $\Phi_{4}(\Lambda)=\sum_{h:C_h\subseteq S\vee M}\lambda_h=
\lambda_1+\lambda_2+\lambda_3+\lambda_4+\lambda_5+\lambda_6=1$. Let $\mathcal{S}'=\{(0,0,0,0,0,1,0)\}$ denote a subset of the set $\mathcal{S}$ of all solutions of $(\Systemsigma) $.   Then, $M_1'=M_2'=M_3'=0$, $M_4'=1$ and hence $I_0'=\{1,2,3\}$ (as defined in (\ref{EQ:I0'})).
 By Theorem \ref{COER-P0'},
 as $(\Systemsigma) $ is solvable and $I_0'=\{1,2,3\}$,
  it is sufficient to check the coherence of 
 the sub-assessment  $\P_0'=(x,y,t)$ on $\F_0'=(M|P,\no{M}|S,S|(P\vee S))$ in order to check the coherence of $(x,y,t,0)$ on $\F$. From   Proposition~\ref{PROP:FIGUREIIBIS}, by replacing $A,B,C$ with $S,M,P$, respectively, it holds that  $(x,y,t)$ on $\F_0'=(M|P,\no{M}|S,S|(P\vee S))$ is coherent for every $(x,y,t)\in[0,1]^3$. Therefore, $(x,y,t,0)$ on $\F$ is coherent for every $(x,y,t)\in[0,1]^3$.
 \end{remark}
We observe from Remark \ref{REM:FIGIINORED} that the conditional event existential import of Figure II does not follows from any syllogism of Figure I, since $p(P|M)=x$, $p(M|S)=y$, and $p(S|(S\vee M))=t>0$ does not imply $p(S|(S\vee P))>0$, because the assessment $p(S|(S\vee P))=0$ is  a coherent extension from $(x,y,t)$. Likewise,
we observe from Remark \ref{REM:FIGIINOREReversed} that  the conditional event existential import of Figure I does not follows from any syllogism of Figure II, since $p(M|P)=x$, $p(\no{M}|S)=y$, and $p(S|(S\vee P))=t>0$ does not imply $p(S|(S\vee M))>0$, because the assessment $p(S|(S\vee M))=0$ is a coherent extension from  $(x,y,t)$.

In Aristotelian syllogistics, for example,  \emph{Cesare} can be reduced by conversion to \emph{Celarent} as follows (see, e.g., \cite{sep-aristotle-logic}, table in Section 5.4): from the premises of \emph{Cesare}, i.e.,  \emph{no $P$ is $M$} and \emph{every $S$ is $M$} it follows by conversion that   \emph{no $M$ is $P$} and \emph{every $S$ is $M$}, which in turn implies by \emph{Celarent} that \emph{no $S$ is $P$}. In our approach however,  this inference does not hold: we observe that the premises  of \emph{Cesare}, i.e., $p(M|P)=0, p(M|S)=1, p(S|(S \vee P))>0$ (see (\ref{EQ:CESARE})),
do 
not imply the premises of \emph{Celarent}, i.e., $p(P|M)=0, p(M|S)=1, p(S|(S \vee M))>0$, 
(see (\ref{EQ:CELARENT})), since $p(S|S\vee M)=0$ is coherent  under the premises of \emph{Cesare} (Remark~\ref{REM:FIGIINOREReversed}). 
Therefore, \emph{Cesare} can not be reduced by conversion to \emph{Celarent} (which requires that $p(S|S\vee M)>0$) in our approach.

\subsection{Reductio ad impossibile.}\label{SEC:RedByImp}
Aristotle described also validity proofs by \emph{reductio ad impossibile}. According to {\L}ukasiewicz, this should correspond to the application of the \emph{compound law of transposition} (\cite[p. 56]{lukasiewicz57}), i.e., \emph{if ($A$ and $B$, then $C$), then (if $A$ and not-$C$, then not-$B$)}. For example, by instantiating \emph{Barbara} in the first conditional (i.e., for $A$ and $B$ \emph{Barbara}'s premises and for  $C$ its conclusion), implies logically \emph{Baroco} by suitable instantiations in the second conditional of the compound law of transposition (\cite[p. 56]{lukasiewicz57}).

In terms of defaults, it can be easily shown that \emph{reductio ad impossibile} holds in our approach because each valid syllogism of Figure II  and Figure III can be reduced to a valid syllogism of Figure I. 
For example, \emph{Camestres} of Figure~II  can be reduced to \emph{Darii} of Figure~I as follows. The compound law of transposition applied to \emph{Camestres} yields the inference: \emph{from $P\normally M$, $S\nnormally \no{P}$, and 
			$(S\vee P) \nnormally \no{S}$ infer $S\nnormally \no{M}$}. By interchanging  $P$ and $M$, we obtain the valid inference \emph{Darii}:  \emph{from $M\normally P$, $S\nnormally \no{M}$, and 
			$(S\vee M) \nnormally \no{S}$ infer $S\nnormally \no{P}$}. For a   sample reduction of a Figure III syllogism to Figure I consider \emph{Darapti}. The application of the compound law of transposition to \emph{Darapti} yields the inference \emph{from $S\normally \no{P}$, 
$M\normally S$, and 
$(S\vee M) \nnormally \no{M})$ infer $M\nnormally P$}. By interchanging $S$ and $M$, we obtain the valid inference \emph{Celaront} of Figure I:  \emph{from $M\normally \no{P}$, 
$S\normally M$, and 
$(S\vee M) \nnormally \no{S})$ infer $S\nnormally P$}. Note that while \emph{Darapti} is \emph{s-valid}, \emph{Celaront} is \emph{valid} but not \emph{s-valid} in our semantics. Since the compound law of transposition ignores this difference, it does not preserve  \emph{s-validity}.

 Our validity proofs are based on the probability propagation rules, which are different for each figure. To what extent they may be reduced to each other, given the asymmetries of the term order between the figures and the different existential import assumptions, is a topic of future research.

\section{Generalized quantifiers.}
\label{SEC:GQ}
The basic syllogistic sentence types A, E, I, O involve quantifiers which we represent by special cases of probability evaluations, namely equal to 1 or 0 for the universal quantifiers, and excluding 0 or 1 for the particular quantifiers.
A natural generalization of such quantifiers is to use thresholds between  0 and 1. Then, we obtain generalized (or intermediate) quantifiers (see, e.g., \cite{barwise81,peters06,peterson00}).  For instance,  the statement 
\emph{Most  $S$ are $P$} (sentence type T, for the notation see \cite{peterson00} see also \cite{Murinov2016IntermediateSI}) can be interpreted by the conditional probability assessment $p(P|S)\geq  x$, where $x$ denotes a suitable threshold (e.g., greater than 0.5). Likewise, the statement \emph{Most  $S$ are not-$P$} (sentence type D) can be interpreted by $p(\no{P}|S)\geq  x$ with a suitable threshold.
The choice of the threshold depends on the context of the speaker.
By using such sentences, we can construct and check the validity of syllogisms involving generalize quantifiers. Specifically, validity can be investigated by suitable instatiations of the probability propagation rules, we proved in the previous sections. Consider for instance the following generalization of Baroco: 
\begin{quote}
\begin{tabular}{ll}
ADD-Figure II:
  &\emph{All $P$ are $M$}. \\ 
  &\emph{Most  $S$ are not-$M$}. \\
  & Therefore, \emph{Most  $S$ are not-$P$}.\\
  \end{tabular}
\end{quote}
In this syllogism, the first premise is of the sentence type A but the second premise and the conclusion consist of   sentence type D. In our semantics this syllogism is interpreted as follows: from the premises $p(M|P)=1$ and  $ p(\widebar{M}|S)  \geq y$  and  the conditional event existential import assumption $p(S|(S \vee P))>0$ infer  the conclusion $p(\no{P}|S)\geq y$, where $y>0.5$. 
To prove  the validity of this syllogism instantiate  $S,M,P$ in Theorem~\ref{THM:PROPWTIV} 
for $A,B,C$ with $x_1=x_2=1$, $y_1>0$,  $y_2=1$,  $t_1>0$, and $t_2=1$.  Then,  we obtain  $z^*=y_1$ and $z^{**}=1$.
 Therefore,  the set $\Sigma$    of coherent extensions on $\no{P}|S$ of the  imprecise assessment $\{1\}\times[y_1,1]\times(0,1]$ on  $(M|P,\no{M}|S,S|(S\vee P))$ is $\Sigma=[y_1,1]$.
Then, we obtain the following generalization of  Equation (\ref{EQ:BAROCO}):
\begin{equation}\label{EQ:BAROCOGQ} 
(p(M|P)=1,   p(\widebar{M}|S)  \geq y_1,  p(S|(S \vee P))>0) \,\models_s \, p(\no{P}|S)\geq y_1 \,.
\end{equation}
By choosing  $y_1>0.5$, Equation (\ref{EQ:BAROCOGQ})   validates ADD-Figure II.

Likewise, we can obtain an extension of Darii (\ref{EQ:DARII}) involving generalized quantifiers. Indeed, by instantiating $S,M,P$ in Theorem~\ref{THM:PROPWTIVFIGI} for $A,B,C$ with $x_1=x_2=1$, $y_1>0$,  $y_2=1$,  $t_1>0$, and $t_2=1$.  Then,  we obtain  
$z^*=\max\left\{0,x_1y_1-\frac{(1-t_1)(1-x_1)}{t_1}\right\}=y_1$ and $z^{**}=1$.
 Therefore,  the set $\Sigma$    of coherent extensions on $P|S$ of the  imprecise assessment $\{1\}\times[y_1,1]\times(0,1]$ on  $(P|M,M|S,S|(S\vee M))$ is $\Sigma=[y_1,1]$.
By Definition \ref{DEF:VALID},
\begin{equation}\label{EQ:DARIIGQ} 
	(p(P|M)=1,   p(M|S)  \geq y_1,  p(S|(S \vee M))>0) \,\models_s \, p(P|S)\geq y_1 \,.
\end{equation}
Equation (\ref{EQ:DARIIGQ}) generalizes (\ref{EQ:DARII}) and validates  the following generalized syllogism, when $y_1>0.5$
\begin{quote}

\begin{tabular}{ll}
ATT-Figure I:    &	\emph{All $M$ are $P$}. \\ 
    &\emph{Most  $S$ are $M$}. \\ 
    &Therefore, \emph{Most  $S$ are $P$}.
\end{tabular}

\end{quote}
While, as pointed in Remark \ref{REM:VALID}, 
valid inferences can generally be obtained from strengthening the premises or weakening the conclusion of valid inferences, it is also possible to check the validity of syllogisms with weaker premises by exploiting the respective probability propagation rules. For an example of generalized syllogism of Figure III consider the following which is a version of Darapti with weakened premises:
\begin{quote}
\begin{tabular}{ll}
TTI-Figure III: &\emph{Most  $M$ are $P$}. \\
&\emph{Most  $M$ are $S$}. \\ 
&Therefore, \emph{some  $S$ is $P$}.\\
\end{tabular}
\end{quote}
Indeed, by instantiating $S,M,P$ in Theorem~\ref{THM:PROPWTIVF3} for $A,B,C$ with $x_1=y_1>0.5$, $x_2=y_2=1$,  $t_1=t>0$, and $t_2=1$.  Then, as $x_1+y_1-1>0$, we obtain  $
z^*=  \frac{t(2x_1-1)}{1-t(1-x_1)}$ and, as $t_1(y_1-x_2)\leq 0$, we obtain $z^{**}=1$. We observe that $z^*>t(2x_1-1)>0$ because $\frac{t}{1-t(1-x_1)}>t$ and $x_1>0.5$.
Therefore,  the set $\Sigma$    of coherent extensions on $P|S$ of the  imprecise assessment $[x_1,1]\times[x_1,1]\times(0,1]$ on  $(P|M,M|S,S|(S\vee M))$, with $x_1>0.5$, is $\Sigma=\bigcup_{t\in (0,1]} [\frac{t(2x_1-1)}{1-t(1-x_1)},1]=(0,1]$.
Then,
\begin{equation}\label{EQ:DARAPTIWEAK} 
	(p(P|M)\geq x_1,   p(S|M)  \geq x_1,  p(M|(S \vee M))>0) \,\models_s \, p(P|S)>0 \,,
\end{equation}
validates TTI-Figure III, which is a generalization of Darapti involving generalized quantifiers where the premises are weakened. 

By applying the probability propagation rules (for precise or interval-valued probability assessments) of Figures I, II, and III 
further syllogisms with generalized quantifiers can be obtained. 
\section{Concluding remarks.}
\label{SEC:CONCL}
In this paper we presented a probabilistic interpretation of the basic syllogistic sentence types (A, E, I, O) and suitable existential import assumptions in terms of 
probabilistic constraints.
By exploiting 
 coherence, we introduced 
the notion of validity and strict validity for probabilistic inferences involving  imprecise probability assessments.

For each Figure I, II, and III, we verified the  coherence of any probability assessment in $[0,1]^3$ on the three conditional events which are involved in the major and the minor premise and the conclusion.
These results
show  that, without existential import assumption,  all traditionally valid syllogisms are probabilistically non-informative.
We also  verified  for all three figures the total  coherence of the imprecise assessment $[0,1]^3$ on the conditional events in the  premise set including the  existential import.
Then, 
we derived 
 the interval of  all coherent extensions on  the conclusion
 for every coherent  (precise or interval-valued) probability  assessment  on the premise set
 for each  of the three figures. These results were then exploited to prove the validity or strict validity of our probabilistic interpretation of all traditionally valid syllogisms of the three figures: Barbara, Barbari, Darii, Celarent, Celaront, and Ferio  of Figure I;
  Camestres, Camestrop, Baroco, Cesare, Cesaro, and Festino of Figure II;
  Datisi, Darapti, Ferison, Felapton, Disamis, and  Bocardo of Figure III.
  As  mentioned before the coherence approach is more general compared to the standard approaches where the conditional probability  $p(E|H)$ is defined by $p(E\wedge H)/p(H)$,  where $p(H)$ must be positive. Indeed, we showed that the conditional event existential import assumption (which is sufficient for validity) is weaker than the requirement of positive conditioning events for the conditional events involved in the syllogisms.
  
We then built
a   bridge from our probability semantics of the Aristotelian syllogisms to nonmonotonic reasoning by interpreting the basic syllogistic sentence types by suitable defaults and negated defaults. We also showed how some new valid syllogisms can be obtained  by strengthening our existential import assumption.
 Moreover,  by this procedure, the traditionally not valid AAA of Figure III can be validated.
 These new syllogisms, which are  expressed in terms of defaults only, are  p-valid inference rules which we propose, together with  default versions of the  traditional ones  for future research in nonmotononic reasoning.
Then we investigated Aristotelian methods of proof within our framwork. We observed that \emph{reductio} by conversion does not work while  \emph{reductio ad impossibile} can be applied in our approach. However, the method of \emph{reductio ad impossibile} by suitable applications of the \emph{compound law of transposition} yields only validity by reducing syllogisms of Figure~II and Figure~III to Figure~I: it ignores  our distinction between \emph{valid} and \emph{s-valid} syllogisms.   
Finally, we  showed how the probability propagation rules can  be used to analyze the validity and the strict validity of syllogisms involving generalized quantifiers. Specifically, sentence like \emph{most $S$ are $P$} can be interpreted by imprecise probability assessments.

We presented a general method to  validate probabilistically non-informative inferences by adding additional premises. These additional premises can be  existential import assumptions, (negated) defaults or other probabilistic constraints. These   methods can be used to solve  inference problems   in general with applications in various disciplines. For instance, our probabilistic interpretation of Aristotelian syllogisms can serve as new rationality  framework for  the psychology of reasoning, which has a long tradition of using syllogistics for assessing the rationality of human inference. Moreover, our results on generalized quantifiers can be applied for investigating semantic and pragmatic problems involving quantification in linguistics. Furthermore, our  bridges to nonmonotonic reasoning show the applicability of the proposed approach in reasoning under uncertainty,  knowledge representation,   and artificial intelligence. 
This  selection of applications  points 
to new  bridges among our semantics, Aristotelian syllogistics, and various disciplines. 

 We will devote future work to apply our semantics to nonmonotonic reasoning and its relation to probability logic (see, e.g., \cite{Hawthorne96,hawthorne07}). Specifically, we will investigate the validity of our default versions of the syllogisms in the light of   different  systems of nonmonotonic reasoning.  
  
Future work will also be  devoted to the full probabilistic analysis of Figure IV.  Indeed,      categorical syllogisms of Figure IV go beyond the scope of this paper for two reasons. 
Firstly, they were  introduced after Aristotle's \emph{Analytica Priora} and are therefore not considered as (proper) Aristotelian syllogisms. 
Secondly, 
in contrast to the first three figures, based on preliminary results,  
there seems  not to exist a unique conditional event existential import assumption  for validating syllogisms of  Figure IV  (\cite{PfSaFIGIV}).  Therefore, 
 several  probability propagation rules should be developed only for this figure, which cannot be done in this paper owing to lack of space.

Finally, another strand of future research will focus on further generalizations of Aristotelian syllogisms by applying the theory of compounds of conditionals under coherence (see, e.g., \cite{GiSa14,GiSa19}). While, in the present paper, we connected the syllogistic terms $S$ and $P$ in the basic syllogistic sentence types by conditional events $P|S$,  this theory of compounds of conditionals allows for obtaining generalized syllogistic sentence types like  \emph{If   $S_1$ are $P_1$, then   $S_2$ are $P_2$} (i.e., $(P_2|S_2)|(P_1|S_1)$) by suitable nestings of conditional events. 
Interestingly, 
in the context of conditional syllogisms, 
 the resulting uncertainty propagation rules coincide with the respective non-nested versions   (see, e.g., \cite{ECSQARU17,SPOG18,SGOP20}). 
 Future research is needed to investigate whether similar results can be obtained in the context of such generalized  Aristotelian syllogisms.

The  various possibilities for applications and generalizations of Aristotelian syllogisms  call for future research and highlight   the impressive research impact of Aristotle's original work.\\

{\bf Acknowledgments.}
We thank two anonymous referees for useful comments. 
Niki Pfeifer was supported by the German Federal Ministry of Education and Research (BMBF project 01UL1906X: ``Logic and philosophy of science of reasoning under uncertainty''). Giuseppe Sanfilippo is 
also supported by the project FFR-2023 of the University of Palermo and is a member of the  INdAM-GNAMPA Research Group, Italy.


 

\end{document}